\documentclass[11pt,a4paper,leqno]{article}
\usepackage{a4wide}
\usepackage{latexsym}
\usepackage{amsmath}
\usepackage{amssymb}
\newtheorem{defin}{Definition}
\newtheorem{lemma}{Lemma}
\newtheorem{prop}{Proposition}
\newtheorem{theo}{Theorem}

\newenvironment{proof}{\medskip\par\noindent{\bf Proof}}{\hfill $\Box$
\medskip\par}
\begin{document}
\title{Multiscale Gevrey asymptotics in boundary layer expansions for some initial value problem with
merging turning points}
\author{{\bf A. Lastra, S. Malek}\\
University of Alcal\'{a}, Departamento de F\'{i}sica y Matem\'{a}ticas,\\
Ap. de Correos 20, E-28871 Alcal\'{a} de Henares (Madrid), Spain,\\
University of Lille 1, Laboratoire Paul Painlev\'e,\\
59655 Villeneuve d'Ascq cedex, France,\\
{\tt alberto.lastra@uah.es}\\
{\tt Stephane.Malek@math.univ-lille1.fr }}
\date{July, 7 2017}
\maketitle
\thispagestyle{empty}
{ \small \begin{center}
{\bf Abstract} 
\end{center}
We consider a nonlinear singularly perturbed PDE leaning on a complex perturbation parameter $\epsilon$.
The problem possesses an irregular singularity in time at the origin and involves a set of so-called
moving turning points merging to 0 with $\epsilon$. We construct outer solutions for time located in
complex sectors that are kept away from the origin at a distance equivalent to a positive power of
$|\epsilon|$ and we
build up a related family of sectorial holomorphic inner solutions for small time inside some boundary layer. We
show that both outer and inner solutions have Gevrey asymptotic expansions as $\epsilon$ tends to 0 on
appropriate sets of sectors that cover a neighborhood of the origin in $\mathbb{C}^{\ast}$. We observe
that their Gevrey orders are distinct in general.

}
\bigskip \bigskip

\section{Introduction}

Within this paper, we focus on a family of nonlinear singularly perturbed equations sharing the shape
\begin{multline}
Q(\partial_{z})P(t,\epsilon)u(t,z,\epsilon) +
P_{1}(t,\epsilon)Q_{1}(\partial_{z})u(t,z,\epsilon)Q_{2}(\partial_{z})u(t,z,\epsilon) \\
= f(t,z,\epsilon) +
P_{2}(t,\epsilon,\partial_{t},\partial_{z})u(t,z,\epsilon) \label{main_PDE_u_intro}
\end{multline}
where $Q,Q_{1},Q_{2},P,P_{1},P_{2}$ stand for polynomials with complex coefficients and $f(t,z,\epsilon)$
denotes a holomorphic function near the origin regarding $t$ and $\epsilon$ in $\mathbb{C}$ and on some
horizontal strip $H_{\beta} = \{z \in \mathbb{C} / |\mathrm{Im}(z)| < \beta \}$ for some $\beta>0$ w.r.t $z$.

This work is a continuation of a study initiated in the contribution \cite{ma2}. Namely, in \cite{ma2}
we considered an equation of the form (\ref{main_PDE_u_intro}) in the case when $P(0,\epsilon)$ is
identically vanishing near 0 that corresponds to a situation
which is analog to one of a \emph{turning point} at $t=0$ (we refer to \cite{wa} and \cite{frsc} for a detailed
outline of this terminology in the context of ODEs). The requirements imposed on the main equation forced the
polynomial $t \mapsto P(t,\epsilon)$ to have an isolated root at $t=0$ whereas the other moving roots
depending upon $\epsilon$ stay apart from a fixed disc enclosing the origin. Under
suitable constraints, we established the existence of a set of actual holomorphic solutions
$y_{p}(t,z,\epsilon)$, meromorphic at $\epsilon=0$ and $t=0$, $0 \leq p \leq \varsigma-1$, for some
integer $\varsigma \geq 2$, defined on domains $\mathcal{T} \times H_{\beta} \times \mathcal{E}_{p}$,
for some prescribed open bounded sector $\mathcal{T}$ at centered at 0 and
$\mathcal{E} = \{ \mathcal{E}_{p} \}_{0 \leq p \leq \varsigma-1}$ is some well chosen set of open bounded
sectors which covers a neighborhood of 0 in $\mathbb{C}^{\ast}$. Furthermore, for convenient integers
$a,b \in \mathbb{Z}$ we have shown that all the functions $\epsilon^{a}t^{b}y_{p}(t,z,\epsilon)$ share
w.r.t $\epsilon$ a common asymptotic expansion
$\hat{y}(t,z,\epsilon) = \sum_{n \geq 0} y_{n}(t,z) \epsilon^{n}$ with bounded holomorphic coefficients
on $\mathcal{T} \times H_{\beta}$. This asymptotic expansion turns out to be of Gevrey type whose order depends
both on data relying on the highest order term of $P_{2}$ which is of irregular type in the sense of \cite{man}
displayed as $\epsilon^{\Delta}t^{\delta q + m}\partial_{t}^{\delta}R(\partial_{z})$ for some positive
integers $\Delta,\delta,q,m>0$, $R$ some polynomial and on the two polynomial $P$ and $P_{1}$ that frame
the turning point at $t=0$.

In this work, we aim attention at a different situation regarding the localization of turning points that is not
covered in our previous study. Namely, we assume that $t \mapsto P(t,\epsilon)$ does not vanish at $t=0$
but possess at least one root leaning on $\epsilon$, a so-called \emph{moving turning point}, which merges to the origin
as $\epsilon$ tends to 0 (see Lemma 1). Our target is to carry out a comparable statement as in \cite{ma2}
namely the construction of sectorial holomorphic solutions and asymptotic expansions of Gevrey type as $\epsilon$
tends to the origin. Nevertheless, the whole picture looks rather different from our previous investigation.
More precisely, according to the presence of the shrinking turning points, the solutions we construct by means of
Laplace and inverse Fourier transforms are only defined w.r.t $t$ on some boundary layer domains which turn out
to be sectors with vertex at 0, with radius that depends on some positive power of $|\epsilon|$ and 
approaches 0 with $\epsilon$. Besides, we can exhibit another family of solutions of (\ref{main_PDE_u_intro})
provided that $t$ remains away from the origin on some unbounded sector with inner radius being proportional
to some positive power of $|\epsilon|$, tending to 0 with $\epsilon$.

In order to explain the manufacturing of these solutions, we need to specify the nature of the forcing term
$f(t,z,\epsilon)$ which is constituted with two terms, one piece is polynomial in $t,\epsilon$ with bounded
holomorphic coefficients on any strip $H_{\beta'} \varsubsetneq H_{\beta}$ with $0 < \beta' < \beta$ and
the other part represented as a function $F^{\theta_{F}}(t,z,\epsilon)$ which solves a singularly perturbed
nonhomogeneous linear ODE of the form
$$ F_{2}(-\epsilon^{\gamma}\partial_{t})F^{\theta_{F}}(t,z,\epsilon) = I_{\theta_{F}}(t,z,\epsilon) $$
for some polynomial $F_{2}(x)$ with complex coefficients not vanishing at $x=0$, some real number
$\gamma>1/2$ and $I_{\theta_{F}}$ some rational function in $t,\epsilon$ and bounded
holomorphic w.r.t $z$ on $H_{\beta'}$ (see Remark 1). This equation is of irregular type at $t=\infty$ and
regular at $t=0$ (we indicate some text book on complex ODEs such as \cite{ba2}, \cite{hssi} for a definition of
these classical notions). According to this last assumption, we stress the fact that the solutions described above
actually solve a PDE with rational coefficients in $t,\epsilon$ and bounded holomorphic w.r.t $z$ on
$H_{\beta'}$ with a shape similar to (\ref{main_PDE_u_intro}) as displayed in Remark 2.

Our first main construction can be outlined as follows. Under appropriate restriction on the shape of
(\ref{main_PDE_u_intro}), we can select a set $\mathcal{E}^{\infty} =
\{ \mathcal{E}_{j}^{\infty} \}_{0 \leq j \leq \iota-1}$ of bounded sectors with aperture
slightly larger than $\pi/\gamma$, for some $\iota \geq 2$, which covers
a neighborhood of 0 in $\mathbb{C}^{\ast}$ and pick up directions $\{ \mathfrak{u}_{j} \}_{0 \leq j \leq \iota-1}$
in $\mathbb{R}$ for which a family of solutions $v^{\mathfrak{u}_{j}}(t,z,\epsilon)$ of the main equation
(\ref{main_PDE_u_intro}), for a specific choice of $\theta_{F}=\mathfrak{u}_{j}$ in the forcing term
$F^{\theta_{F}}$ described above, can be
built up as a usual Laplace and Fourier inverse transform
$$ v^{\mathfrak{u}_{j}}(t,z,\epsilon) =
\frac{\epsilon^{\gamma_{0}}}{(2\pi)^{1/2}} \int_{-\infty}^{+\infty} \int_{L_{\mathfrak{u}_{j}}}
W^{\mathfrak{u}_{j}}(u,m,\epsilon) \exp( -\frac{t}{\epsilon^{\gamma}}u ) e^{izm} du dm $$
along the halfline $L_{\mathfrak{u}_{j}} = \mathbb{R}_{+}e^{i \mathfrak{u}_{j}}$, for some real number
$\gamma_{0}$, where $W^{\mathfrak{u}_{j}}(u,m,\epsilon)$ represents a function with at most exponential growth
of order 1 on a sector enclosing $L_{\mathfrak{u}_{j}}$ w.r.t $u$, with exponential decay w.r.t $m$ on
$\mathbb{R}$ and analytic dependence on $\epsilon$ near 0. In addition, for each fixed
$\epsilon \in \mathcal{E}_{j}^{\infty}$, the restriction
$(t,z) \mapsto v^{\mathfrak{u}_{j}}(t,z,\epsilon)$ is bounded and holomorphic on
$\mathcal{T}_{\epsilon}^{\infty} \times H_{\beta'}$, where $\mathcal{T}_{\epsilon}^{\infty}$ stands for an
unbounded sector with inner radius proportional to $|\epsilon|^{\gamma - \Gamma}$ for some real number
$0 \leq \Gamma < \gamma$ (Theorem 2). Furthermore, we explain why the functions
$\epsilon^{-\gamma_{0}}v^{\mathfrak{u}_{j}}(t,z,\epsilon)$ own w.r.t $\epsilon$ a common asymptotic expansion
$\hat{O}_{t}(\epsilon) = \sum_{k \geq 0} O_{t,k}\epsilon^{k}$ whose coefficients $O_{t,k}$ represent
bounded holomorphic functions on $H_{\beta'}$ which can be called \emph{outer expansions} owing to the fact
that it is valid for any fixed value of $t$ in the vicinity of 0 as $\epsilon$ tends to 0. Accordingly, we call
$v^{\mathfrak{u}_{j}}(t,z,\epsilon)$ the \emph{outer solutions} of (\ref{main_PDE_u_intro}). Besides, we
can indicate the nature of this asymptotic expansion that turns out to be of Gevrey order (at most)
$1/\gamma$, ensuring that $\epsilon^{-\gamma_{0}}v^{\mathfrak{u}_{j}}$ can be labeled as $\gamma-$sum of
$\hat{O}_{t}(\epsilon)$ on $\mathcal{E}_{j}^{\infty} \cap D(0,\sigma_{t})$ for some radius
$\sigma_{t}$ outlined in (\ref{defin_sigma_t}) (Theorem 3). We may notice that this Gevrey order
$1/\gamma$ is substantially related to the Stokes phenomena stemming from the solutions
$F^{\mathfrak{u}_{j}}(t,z,\epsilon)$ of the ODE (\ref{ODE_F_theta_F}).

Now, we proceed to the description of what we call the \emph{inner solutions} of (\ref{main_PDE_u_intro}). Submitted
to additional requirements on the coefficients of (\ref{main_PDE_u_intro}), we can choose a set
$\mathcal{E} = \{ \mathcal{E}_{p} \}_{0 \leq p \leq \varsigma-1}$ built up with bounded sectors
with opening barely larger than $\frac{\pi}{\chi \kappa}$ for some integer $\kappa \geq 1$ and real
number $\chi > \frac{1}{2\kappa}$, for some $\varsigma \geq 2$, which covers a neighborhood of 0 in
$\mathbb{C}^{\ast}$ and raise a set of real directions $\{ \mathfrak{d}_{p} \}_{0 \leq p \leq \varsigma-1}$
such that for each direction $\mathfrak{u}_{j}$ coming up from one single outer solution
$v^{\mathfrak{u}_{j}}$, $0 \leq j \leq \iota-1$, one can construct a family of solutions
$u^{\mathfrak{d}_{p},j}(t,z,\epsilon)$ to (\ref{main_PDE_u_intro}), $0 \leq p \leq \varsigma-1$, that is
represented as a Laplace transform of some order $\kappa$ and Fourier inverse transform
$$
u^{\mathfrak{d}_{p},j}(t,z,\epsilon) =
\epsilon^{-m_{0}}\frac{\kappa}{(2\pi)^{1/2}}\int_{-\infty}^{+\infty}
\int_{L_{\mathfrak{d}_{p}}} \omega_{\kappa}^{\mathfrak{d}_{p},j}(u,m,\epsilon)
\exp( -( \frac{u}{\epsilon^{\alpha}t})^{\kappa} ) e^{izm} \frac{du}{u} dm
$$
where the inner integration is performed along the halfline $L_{\mathfrak{d}_{p}} =
\mathbb{R}_{+}e^{i \mathfrak{d}_{p}}$, for some positive integer $m_{0} \geq 1$, negative rational number
$\alpha < 0$ and where $\omega_{\kappa}^{\mathfrak{d}_{p},j}(u,m,\epsilon)$ stands for a function with at most
exponential growth of order $\kappa$ on a sector containing $L_{\mathfrak{d}_{p}}$ w.r.t $u$, with
exponential decay w.r.t $m$ on $\mathbb{R}$ and analytic dependence w.r.t $\epsilon$ in the vicinity
of the origin. Besides, for each fixed $\epsilon \in \mathcal{E}_{p}$, the projection
$(t,z) \mapsto u^{\mathfrak{d}_{p},j}(t,z,\epsilon)$ is bounded and holomorphic on
$\mathcal{T}_{\epsilon,\chi - \alpha} \times H_{\beta'}$, for $\mathcal{T}_{\epsilon,\chi-\alpha} =
X\epsilon^{\chi - \alpha}$ where $X$ stands for some fixed bounded sector centered at 0. (Theorem 1). Moreover,
we justify why the functions
$\epsilon^{m_0}u^{\mathfrak{d}_{p},j}(t,z,\epsilon)$ admit w.r.t $\epsilon$ a common asymptotic
expansion $\hat{I}^{j}(\epsilon) = \sum_{k \geq 0} I_{k}^{j} \epsilon^{k}$ with coefficients
$I_{k}^{j}$ belonging to a Banach space of bounded holomorphic functions on $X \times H_{\beta'}$ that may be
called \emph{inner expansion} since it is only legitimated for $t$ on the boundary layer set
$\mathcal{T}_{\epsilon,\chi - \alpha}$ which shrinks to 0 with $\epsilon$. We can also specify the type of
asymptotic expansion which turns out to be of Gevrey order (at most) $\frac{1}{\chi \kappa}$. As a result,
$\epsilon^{m_0}u^{\mathfrak{d}_{p},j}(t,z,\epsilon)$ can be identified as $\chi \kappa-$sum of
$\hat{I}^{j}(\epsilon)$ on $\mathcal{E}_{p}$ (Theorem 3). By construction, the integer $\kappa$ crops up
in the highest order term of the operator $P_{2}$ which is assumed to be of the form
$\epsilon^{\Delta_{D}}t^{\delta_{D}(\kappa+1)}\partial_{t}^{\delta_{D}}R_{D}(\partial_{z})$ for some
integers $\Delta_{D} \geq 1$, $\delta_{D} \geq 2$ and a polynomial $R_{D}$. The real number
$\chi$ is in particular related by a set of inequalities to the integers $\kappa$,$\Delta_{D}$,$\delta_{D}$,
the powers of $\epsilon$ and $t$ in $P,P_{1}$, to the real number $\gamma$ and the forcing term
$f(t,z,\epsilon)$. As outgrowth, we observe that this Gevrey order $\frac{1}{\chi \kappa}$ involves informations
emanating from the moving turning points and the irregularity of the operator $P_{2}$ at $t=0$.

These so-called inner and outer expansions come into play in vast literature on what is commonly named
\emph{matched asymptotic expansions}. For further details on this subject, we refer to classical textbooks
such as \cite{beor}, \cite{ec}, \cite{la}, \cite{om}, \cite{sk}, \cite{wa}. We point out the recent work by
A. Fruchard and R. Sch\"{a}fke on composite asymptotic expansions, see \cite{frsc}, which provides a solid
bedrock for the method of matching and furnishes hands-on criteria for the study of the nature of these asymptotic
expansions which can be shown of Gevrey type with the same order for both inner and outer expansions for
several families of
singularly perturbed ODEs. In our work, we observe however an interesting situation in which the Gevrey order
of the outer and of the inner expansions turn out to be different in general (see the two examples after Theorem 3).
Nevertheless, we observe a scaling gap that prevents our inner and outer solutions to share a common domain in time
$t$ for all $\epsilon$ small enough, see Remark 3. More work is needed if one wants to analytically continue and
match our inner and outer solutions. This stays beyond the scope of our approach and we leave it for future inspection.

It is worthwhile noting that a similar phenomenon of parametric multiple scale asymptotics related to moving
turning points has been observed in a recent work \cite{sutak} by K. Suzuki and Y. Takei for singularly perturbed second order
ODEs of the form
$$ \epsilon^{2}\psi''(z,\epsilon) = (z - \epsilon^{2}z^{2})\psi(z,\epsilon) $$
whose moving turning point $z=1/\epsilon^{2}$ tends to infinity which turns out to be an irregular singularity of
the equation. In particular, they have shown that the power series part $\hat{\varphi}_{\pm}(z,\epsilon)$
of its WKB solution $\hat{\psi}_{\pm}(z,\epsilon)=\exp( \pm \frac{1}{\epsilon} \int^{z} \sqrt{z} dz )
\hat{\varphi}_{\pm}(z,\epsilon)$ presents a double scale structure of Gevrey order $1/4$ and $1$ for all fixed $z$,
in being $(4,1)-$multisummable w.r.t $\epsilon$ except for a finite number of singular directions. Furthermore,
a second example involving three distinct Gevrey levels has been worked out by Y. Takei in \cite{tak}.\medskip

\noindent The paper is organized as follows.\\
In Section 2, after recalling some ground facts about Fourier transforms acting on spaces of functions with
exponential decay on $\mathbb{R}$, we disclose the main problem (\ref{main_PDE_u}) of our study.\\
In Section 3, we build up our inner solutions. We start by reminding the definition and first properties of our
Borel-Laplace transforms of order $k>0$. Then, we redefine some Banach spaces with exponential growth on sectors
of order $\kappa$ and exponential decay on the real line as introduced in our previous work \cite{ma2}. In Section
3.3, we search for conjectural time rescaled formal solutions and present the convolution equation satisfied by
their Borel transforms. In Section 3.4, we solve this latter convolution problem within the Banach spaces mentioned
above with the help of a fixed point procedure. In the last subsection, we construct a family of actual
holomorphic solutions to (\ref{main_PDE_u}) related to a good covering in $\mathbb{C}^{\ast}$ w.r.t $\epsilon$, for
small time $t$ belonging to an $\epsilon-$depending boundary layer.\\
In Section 4, we shape our outer solutions. We begin with the description of basic operations on classical
Laplace transforms. Then, we introduce Banach spaces of functions with exponential growth of order 1 on sectors
and exponential decay on $\mathbb{R}$ which are a slender modification of the ones described in Section 3.
In Section 4.3, we seek for speculative solutions of (\ref{main_PDE_u}) represented as classical Laplace and
inverse Fourier transforms and we exhibit a related nonlinear convolution equation (\ref{main_conv_eq_W}) in the Borel
plane and Fourier space. In Section 4.4, we find solutions of (\ref{main_conv_eq_W}) located in the Banach spaces
quoted above using again a fixed point argument. In the ending subsection, for a suitable good covering
in $\mathbb{C}^{\ast}$ w.r.t $\epsilon$ we distinguish a set of holomorphic solutions to
(\ref{main_PDE_u}) for both large time and small time $t$ kept distant from the origin by a quantity
proportional to a positive power of $|\epsilon|$.\\
In Section 5, we investigate the asymptotic expansions of the inner and outer solutions. We first bring to mind
the Ramis-Sibuya cohomological approach for $k-$summability of formal series. Then, we discuss the existence
of a common asymptotic expansion of Gevrey order $\frac{1}{\chi \kappa}$ for the inner solutions and of Gevrey
order $1/\gamma$ for the outer solutions on the corresponding coverings w.r.t $\epsilon$.

\section{Outline of the main problem}

\subsection{Fourier transforms}
In this subsection, we recall without proofs some properties of the inverse Fourier
transform acting on continuous functions with exponential decay on $\mathbb{R}$, see \cite{lama1},
Proposition 7 for more details.

\begin{defin} Let $\beta > 0$ and $\mu > 1$ be real numbers. We denote $E_{(\beta,\mu)}$ the
vector space of functions
$h : \mathbb{R} \rightarrow \mathbb{C}$ such that
$$ ||h(m)||_{(\beta,\mu)} = \sup_{m \in \mathbb{R}} (1+|m|)^{\mu}\exp(\beta|m|)|h(m)| $$
is finite. The space $E_{(\beta,\mu)}$ endowed with the norm $||.||_{(\beta,\mu)}$ becomes a Banach space.
\end{defin}

\noindent As stated in Proposition 5 from \cite{lama1}, we notice that

\begin{prop} Let $Q_{1}(X),Q_{2}(X),R(X) \in \mathbb{C}[X]$ be polynomials such that
\begin{equation}
\mathrm{deg}(R) \geq \mathrm{deg}(Q_{1}) \ \ , \ \ \mathrm{deg}(R) \geq \mathrm{deg}(Q_{2}) \ \ , \ \ R(im) \neq 0,
\label{cond_R_Q1_Q2}
\end{equation}
for all $m \in \mathbb{R}$. Assume that $\mu > \max( \mathrm{deg}(Q_{1})+1, \mathrm{deg}(Q_{2})+1 )$. Then, there exists a
constant $C_{5}>0$ (depending on $Q_{1},Q_{2},R,\mu$) such that
\begin{multline}
|| \frac{1}{R(im)} \int_{-\infty}^{+\infty} Q_{1}(i(m-m_{1})) f(m-m_{1}) Q_{2}(im_{1})g(m_{1}) dm_{1} ||_{(\beta,\mu)}\\
\leq C_{5} ||f(m)||_{(\beta,\mu)}||g(m)||_{(\beta,\mu)}
\end{multline}
for all $f(m),g(m) \in E_{(\beta,\mu)}$. Therefore, $(E_{(\beta,\mu)},||.||_{(\beta,\mu)})$ becomes a Banach algebra for the product
$\star$ defined by
$$ f \star g (m) = \frac{1}{R(im)} \int_{-\infty}^{+\infty} Q_{1}(i(m-m_{1})) f(m-m_{1}) Q_{2}(im_{1})g(m_{1}) dm_{1}.$$
As a particular case, when $f,g \in E_{(\beta,\mu)}$ with $\beta>0$ and $\mu>1$, the classical convolution
product
$$ f \ast g (m) = \int_{-\infty}^{+\infty} f(m-m_{1})g(m_{1}) dm_{1} $$
belongs to $E_{(\beta,\mu)}$.
\end{prop}

\begin{prop} 
1) Let $f : \mathbb{R} \rightarrow \mathbb{R}$ be a continuous function and a constant $C>0$ such that
$|f(m)| \leq C \exp(-\beta |m|)$ for all $m \in \mathbb{R}$, for some $\beta > 0$. The inverse Fourier
transform of $f$ is defined by the integral representation
$$ \mathcal{F}^{-1}(f)(x) = \frac{1}{ (2\pi)^{1/2} } \int_{-\infty}^{+\infty} f(m) \exp( ixm ) dm $$
for all $x \in \mathbb{R}$. It turns out that the function $\mathcal{F}^{-1}(f)$ extends to an analytic function
on the horizontal strip
\begin{equation}
H_{\beta} = \{ z \in \mathbb{C} / |\mathrm{Im}(z)| < \beta \}. \label{strip_H_beta}
\end{equation}
Let $\phi(m) = im f(m)$. Then, we have the commuting relation
\begin{equation}
\partial_{z} \mathcal{F}^{-1}(f)(z) = \mathcal{F}^{-1}(\phi)(z) \label{dz_fourier}
\end{equation}
for all $z \in H_{\beta}$.\\
2) Let $f,g \in E_{(\beta,\mu)}$ and let $\psi(m) = \frac{1}{(2\pi)^{1/2}} f * g(m)$, the convolution product of $f$ and $g$, for all $m \in \mathbb{R}$.
From Proposition 1, we know that $\psi \in E_{(\beta,\mu)}$. Moreover, the next formula
\begin{equation}
\mathcal{F}^{-1}(f)(z)\mathcal{F}^{-1}(g)(z) = \mathcal{F}^{-1}(\psi)(z) \label{prod_fourier}
\end{equation}
holds for all $z \in H_{\beta}$.
\end{prop}

\subsection{Display of the main problem}

Let $q \geq 1$, $M,Q \geq 0$ and $D \geq 2$ be integers. For $1 \leq l \leq q$, let $k_{l}$ be a non negative
integer such that $1 \leq k_{l} < k_{l+1}$ for $l \in \{1,\ldots,q-1\}$. For all $0 \leq l \leq q$, let $m_{l}$ be
a non negative integer and $a_{l}$ be a complex number not equal to 0.
For all $0 \leq l \leq M$, we consider non negative
integers $h_{l}$, $\mu_{l}$ and a complex number $c_{l}$ such that $1 \leq h_{l} < h_{l+1}$
for $l \in \{0,\ldots,M-1\}$.
For all $0 \leq l \leq Q$, we denote $n_{l}$ and $b_{l}$ non negative integers such that
$1 \leq b_{l} < b_{l+1}$ for $l \in \{0,\ldots,Q-1 \}$. For $1 \leq l \leq D$, we set
nonnegative integers $\Delta_{l}$, $d_{l}$ and $\delta_{l}$ such that
$1 \leq \delta_{l} < \delta_{l+1}$ for $l \in \{1,\ldots,D-1 \}$.

Let $Q(X),Q_{1}(X),Q_{2}(X),R_{l}(X) \in \mathbb{C}[X]$, $1 \leq l \leq D$, be polynomials that satisfy
\begin{multline}
\mathrm{deg}(Q) = \mathrm{deg}(R_{D}) \geq \mathrm{deg}(R_{l}), \ \
\mathrm{deg}(R_{D}) \geq \max( \mathrm{deg}(Q_{1}), \mathrm{deg}(Q_{2}) ), \\
Q(im) \neq 0 \ \ , \ \ R_{D}(im) \neq 0 \label{constraints_degree_coeff_Q_Rl}
\end{multline}
for all $m \in \mathbb{R}$, all $1 \leq l \leq D-1$.

We focus on the following nonlinear singularly perturbed PDE
\begin{multline}
(\sum_{l=1}^{q} a_{l} \epsilon^{m_{l}} t^{k_l} +
a_{0}\epsilon^{m_{0}}) Q(\partial_{z}) u(t,z,\epsilon) +
(\sum_{l=0}^{M} c_{l} \epsilon^{\mu_{l}} t^{h_{l}})Q_{1}(\partial_{z})u(t,z,\epsilon)
Q_{2}(\partial_{z})u(t,z,\epsilon) \\
= \sum_{j=0}^{Q} b_{j}(z) \epsilon^{n_j} t^{b_j} + F^{\theta_{F}}(t,z,\epsilon) +
\sum_{l=1}^{D} \epsilon^{\Delta_l} t^{d_l} \partial_{t}^{\delta_l} R_{l}(\partial_{z})u(t,z,\epsilon)
\label{main_PDE_u}
\end{multline}
We make the following crucial assumption on the integers $m_{l}$, $0 \leq l \leq q$. We take for granted that
\begin{equation}
m_{0} > m_{l_{1}} \label{constraint_m0_ml} 
\end{equation}
for some $l_{1} \in \{ 1, \ldots, q \}$. Let
$P(t,\epsilon) = \sum_{l=1}^{q} a_{l} \epsilon^{m_{l}} t^{k_l} + a_{0}\epsilon^{m_{0}}$. The roots of
$t \mapsto P(t,\epsilon)$ are
called, in our context, \emph{turning points} of the equation (\ref{main_PDE_u}), with analogy to the
situation concerned with ordinary differential equations. See
\cite{frsc}, \cite{wa} for more details. In the next lemma, we supply some information about the position of some
roots of $P$. 
\begin{lemma} Under the constraint
(\ref{constraint_m0_ml}), the polynomial $t \mapsto P(t,\epsilon)$ has at least one root in the disc
$D(0,|\epsilon|^{\mu_{P}})$ centered at 0 with radius $|\epsilon|^{\mu_{P}}$ for some small enough real
number $\mu_{P}>0$ (depending only on $m_{l},k_{l}$, $1 \leq l \leq q$ and $m_{0}$), provided that
$|\epsilon|$ is taken small enough.
\end{lemma}
\begin{proof} According to the restriction (\ref{constraint_m0_ml}) we may assume that
$$ m_{0} > \mathrm{min}_{l=1}^{q} m_{l} = \{ m_{j_1},\ldots,m_{j_{h}} \} $$
for some $1 \leq j_{s} \leq q$, $1 \leq s \leq h$, with $j_{1} < \ldots < j_{h}$.
We can rewrite
\begin{multline*}
P_{1}(t,\epsilon) = P(t,\epsilon)/\epsilon^{m_{j_1}} = \sum_{l=1}^{q} a_{l} \epsilon^{m_{l} - m_{j_1}}
t^{k_l} + a_{0}\epsilon^{m_{0} - m_{j_1}} = a_{j_1}t^{k_{j_1}} \\
+
\sum_{l \in \{ 1,\ldots,q \},l \neq j_{1}} a_{l} \epsilon^{m_{l} - m_{j_1}} t^{k_l} + a_{0} \epsilon^{m_{0} - m_{j_1}}
\end{multline*}
Let $P_{0}(t) = a_{j_1}t^{k_{j_1}}$. Recall that, by construction, $a_{j_1} \neq 0$. We plan to show that
\begin{equation}
|P_{1}(t,\epsilon) - P_{0}(t)| < |P_{0}(t)| \label{ineq_P1_P0}
\end{equation}
holds for some appropriate $\mu>0$, for all $t$ in the circle $C(0,|\epsilon|^{\mu})$ centered at 0
with radius $|\epsilon|^{\mu}$, provided that $\epsilon$ is small enough. Actually, we will observe that the
quantity $\sup_{t \in C(0,|\epsilon|^{\mu})} |P_{1}(t,\epsilon) - P_{0}(t)|/|P_{0}(t)|$ tends to 0 as $\epsilon$
tends to 0, which in particular yields the inequality (\ref{ineq_P1_P0}). Indeed, we can write
\begin{multline}
|P_{1}(t,\epsilon) - P_{0}(t)|/|P_{0}(t)| \leq
\sum_{l \in \{1, \ldots,q\},l \neq j_{1}}^{q} |\frac{a_{l}}{a_{j_1}}| |\epsilon|^{m_{l} - m_{j_1}} |t|^{k_{l} - k_{j_1}} +
|\frac{a_{0}}{a_{j_1}}||\epsilon|^{m_{0} - m_{j_1}}|t|^{-k_{j_1}}\\
= \sum_{l \in \{1, \ldots, q \},l \neq j_{1}}^{q} |\frac{a_{l}}{a_{j_1}}| |\epsilon|^{m_{l} - m_{j_1} + \mu(k_{l} - k_{j_1})} +
|\frac{a_{0}}{a_{j_1}}||\epsilon|^{m_{0} - m_{j_1} - k_{j_1}\mu} \label{maj_P1_minus_P0}
\end{multline}
for all $t \in C(0,|\epsilon|^{\mu})$.
We take $\mu > 0$ (which depends only on $m_{l},k_{l}$, $1 \leq l \leq q$ and $m_{0}$) such that
$$ m_{0} - m_{j_1} - k_{j_1}\mu > 0 \ \ , \ \ m_{l} - m_{j_1} + \mu( k_{l} - k_{j_1} ) > 0 $$
holds, for all $1 \leq l \leq q$ with $l \neq j_{1}$. Notice that such a $\mu>0$ exists since, by
construction, $m_{0} > m_{j_1}$, hence we can take $\mu < (m_{0} - m_{j_1})/k_{j_1}$. Furthermore,
for $l \notin \{ j_{1},\ldots,j_{h} \}$, we have that $m_{l} - m_{j_1}>0$. Hence, if
$k_{l} > k_{j_1}$, then $m_{l} - m_{j_1} + \mu(k_{l} - k_{j_1}) > 0$ holds for any $\mu>0$ and if
$k_{l} < k_{j_1}$, then we may take $0 < \mu < (m_{l} - m_{j_1})/(k_{j_1} - k_{l})$. Likewise,
for $l \in \{ j_{2},\ldots,j_{h} \}$ (in case $h \geq 2$), $m_{l} - m_{j_1}=0$ and $k_{l} - k_{j_1}>0$ since
$l > j_{1}$, therefore $m_{l} - m_{j_1} + \mu( k_{l} - k_{j_1} ) > 0$ holds for any $\mu>0$.

As a result, for $|\epsilon|$ small enough, the right handside of the inequality (\ref{maj_P1_minus_P0}) can be
taken strictly smaller than 1. Hence the inequality (\ref{ineq_P1_P0}) holds. Now, we can apply Rouch\'e's theorem
which states that $t \mapsto P_{1}(t,\epsilon)$ and $P_{0}(t)$ have the same number of roots
(counted with multiplicity) inside the disc $D(0,|\epsilon|^{\mu})$. In conclusion, $t \mapsto P_{1}(t,\epsilon)$
possesses $k_{j_1}$ roots inside $D(0,|\epsilon|^{\mu})$ for $|\epsilon|$ small enough.
\end{proof}
The lemma above ensures that (\ref{main_PDE_u}) possesses at least one (movable) turning point
which tends to 0 as $\epsilon$ tends to 0. Notice that this case is not covered by our previous work \cite{ma2},
where $t=0$ is assumed to be a turning point of the equation and all movable turning points
depending on $\epsilon$ remain outside a fixed disc enclosing the origin, see Remark 2 therein.\medskip

The coefficients $b_{j}(z)$ are displayed as follows. For all $0 \leq j \leq Q$,
we consider functions $m \mapsto B_{j}(m)$ that belong to the Banach space $E_{(\beta,\mu)}$ for some
$\mu > \max( \mathrm{deg}(Q_{1}) + 1, \mathrm{deg}(Q_{2}) + 1 )$ and $\beta > 0$. We set
\begin{equation}
b_{j}(z) = \mathcal{F}^{-1}(m \mapsto B_{j}(m))(z) \ \ , \ \ 0 \leq j \leq Q, \label{defin_b_j} 
\end{equation}
where $\mathcal{F}^{-1}$ denotes the Fourier inverse transform defined in Proposition 2. By construction,
$b_{j}(z)$ defines a holomorphic function on the horizontal strip
$H_{\beta} = \{ z \in \mathbb{C} / |\mathrm{Im}(z)|<\beta \}$ which is bounded on every substrip
$H_{\beta'}$ for any given $0 < \beta' < \beta$.

The function
$F^{\theta_{F}}(t,z,\epsilon)$ is a part of the forcing term given as an integral transform
\begin{equation}
F^{\theta_{F}}(t,z,\epsilon) = \frac{\epsilon^{n_{F}}}{(2\pi)^{1/2}}
\int_{-\infty}^{+\infty} \int_{L_{\theta_{F}}} \omega_{F}(u,m)( \exp( -\frac{t}{\epsilon^{\gamma}} u ) - 1 )
e^{izm} du dm \label{defin_F_tzepsilon}
\end{equation}
where $n_{F} \geq 0$ is some integer, $\gamma > 1/2$ is a real number and $\omega_{F}(\tau,m)$ is a function defined
as
\begin{equation}
\omega_{F}(\tau,m) = C_{F}(m)e^{-K_{F}\tau} \frac{F_{1}(\tau)}{F_{2}(\tau)} \label{defin_omega_F}
\end{equation}
where $C_{F}(m)$ belongs to the Banach space $E_{(\beta,\mu)}$,
$K_{F}>0$ is a real number and $F_{1}(\tau),F_{2}(\tau)$ are two polynomials with coefficients in
$\mathbb{C}$ such that $\mathrm{deg}(F_1) \leq \mathrm{deg}(F_2)$. The path of integration
$L_{\theta_{F}} = \{ ue^{\sqrt{-1}\theta_{F}} / u \in [0,+\infty) \}$ is chosen in such a way that it avoids the
roots of $F_{2}(\tau)$ and with $\theta_{F} \in (-\pi/2,\pi/2)$.

We first assert that the function $F^{\theta_{F}}(t,z,\epsilon)$ is well defined and bounded holomorphic in
time $t$ on some $\epsilon-$depending neighborhood of 0, in space $z$ on any strip $H_{\beta'}$ with
$0 < \beta' < \beta$, provided that $\epsilon$ is not vanishing in the vicinity of the origin. Namely, let us
select a real number $\delta_{1}^{0}>0$ with $\cos(\theta_{F}) > \delta_{1}^{0}/K_{F}$. We introduce the disc
$$ D_{F;\epsilon} = \{ t \in \mathbb{C} / |t| < (-\delta_{1}^{0} + K_{F}\cos(\theta_{F}))
|\epsilon|^{\gamma} \}.$$
According to the assumptions made above, one can sort a constant $C_{F_{1},F_{2}}>0$ with
\begin{equation}
\left| \frac{F_{1}(u)}{F_{2}(u)} \right| \leq C_{F_{1},F_{2}} \label{bounds_F1_over_F2}
\end{equation}
for $u \in L_{\theta_{F}}$. Then, the next estimates
\begin{multline}
|F^{\theta_{F}}(t,z,\epsilon)| \leq \frac{|\epsilon|^{n_{F}}C_{F_{1},F_{2}}}{(2\pi)^{1/2}}
\int_{-\infty}^{+\infty} \int_{0}^{+\infty} |C_{F}(m)|
\exp( -K_{F}r \cos(\theta_{F}) )\\
\times ( \exp(|\frac{t}{\epsilon^{\gamma}}|r) + 1) e^{-m \mathrm{Im}(z)} dr dm \leq
\frac{|\epsilon|^{n_{F}}C_{F_{1},F_{2}}||C_{F}(m)||_{(\beta,\mu)}}{(2\pi)^{1/2}}
\int_{-\infty}^{+\infty} (1 + |m|)^{-\mu} e^{-(\beta- \beta')|m|} dm\\
\times (\int_{0}^{+\infty} e^{-\delta_{1}^{0}r} + e^{-K_{F}r \cos(\theta_{F})} dr)
\end{multline}
hold for all $t \in D_{F;\epsilon}$, $\epsilon \in D(0,\epsilon_{0}) \setminus \{ 0 \}$, for some
$\epsilon_{0} > 0$ and $z \in H_{\beta'}$,
for any $0 < \beta' < \beta$. As a consequence, $(t,z) \mapsto F^{\theta_{F}}(t,z,\epsilon)$ represents a
holomorphic bounded function on $D_{F;\epsilon} \times H_{\beta'}$ for any $0 < \beta' < \beta$ and
$\epsilon \in D(0,\epsilon_{0}) \setminus \{ 0 \}$.

In a second place, we check that $F^{\theta_{F}}(t,z,\epsilon)$ represents a holomorphic function w.r.t $t$
on some $\epsilon-$depending unbounded sectorial domain away from the origin, in space $z$ on strips
$H_{\beta'}$ with $0 < \beta' < \beta$, for any given $\epsilon$ belonging to some suitable bounded sector
centered at 0. More specifically, we consider an unbounded open sector $U_{\theta_F}$ centered at 0 with
an aperture chosen in a manner that it bypasses all the roots of the polynomial $F_{2}(\tau)$.
Let $\delta_{2}^{\infty}>0$ be a positive real number. We sort a bounded
sector $\mathcal{E}^{\infty}$ centered at 0 with opening contained in the range $(\pi/\gamma,2\pi)$,
a positive real number $\delta_{1}^{\infty}$ and suitable directions
$\alpha_{\infty} < \beta_{\infty}$ taken in a way that
there exists some direction $\theta_{F}^{\Delta}$ (that may depend on $\epsilon$ and $t$) satisfying
$e^{i \theta_{F}^{\Delta}} \in U_{\theta_F}$ and fulfills the next demand
\begin{equation}
\theta_{F}^{\Delta} + \mathrm{arg}(\frac{t}{\epsilon^{\gamma}}) \in (-\frac{\pi}{2},\frac{\pi}{2}) \ \ , \ \
\cos( \theta_{F}^{\Delta} + \mathrm{arg}(\frac{t}{\epsilon^{\gamma}}) ) > \delta_{1}^{\infty}
\end{equation}
for all $\epsilon \in \mathcal{E}^{\infty}$ and $t$ belonging to the unbounded sector
$$ \mathcal{T}_{F;\epsilon}^{\infty} = \{ t \in \mathbb{C}^{\ast} /
|t| > \frac{K_{F} + \delta_{2}^{\infty}}{\delta_{1}^{\infty}}|\epsilon|^{\gamma} \ \ , \ \
\alpha_{\infty} < \mathrm{arg}(t) < \beta_{\infty} \}. $$
Through a path deformation argument, we may observe that $F^{\theta_F}(t,z,\epsilon)$ can be rewritten as the sum
\begin{multline*}
F^{\theta_F}(t,z,\epsilon) =
\frac{\epsilon^{n_F}}{(2\pi)^{1/2}} \left( \int_{-\infty}^{+\infty} \int_{L_{\theta_{F}^{\Delta}}}
\omega_{F}(u,m) \exp( -\frac{t}{\epsilon^{\gamma}} u )e^{izm} du dm \right. \\
\left. - \int_{-\infty}^{+\infty} \int_{L_{\theta_F}} \omega_{F}(u,m) e^{izm} du dm \right)
\end{multline*}
for all $t \in \mathcal{T}_{F;\epsilon}^{\infty}$, $z \in H_{\beta'}$ with $0 < \beta' < \beta$ and
$\epsilon \in \mathcal{E}^{\infty}$. As above, we can select a constant $C_{F_{1},F_{2}}>0$ for which
(\ref{bounds_F1_over_F2}) holds. Then, from the latter decomposition, we deduce that
\begin{multline}
|F^{\theta_F}(t,z,\epsilon)| \leq \frac{|\epsilon|^{n_F}}{(2\pi)^{1/2}}C_{F_{1},F_{2}}
\left( \int_{-\infty}^{+\infty} \int_{0}^{+\infty}
|C_{F}(m)| \right. \\
\times \exp( K_{F}r - \frac{|t|}{|\epsilon|^{\gamma}}r
\cos( \theta_{F}^{\Delta} + \mathrm{arg}(\frac{t}{\epsilon^{\gamma}}) ) ) e^{-m \mathrm{Im}(z)} dr dm
\\
\left. + \int_{-\infty}^{+\infty} \int_{0}^{+\infty} |C_{F}(m)| \exp( -K_{F}r \cos(\theta_{F}) )
e^{-m \mathrm{Im}(z)} dr dm \right)\\
\leq \frac{|\epsilon|^{n_F}C_{F_{1},F_{2}}||C_{F}(m)||_{(\beta,\mu)}}{(2\pi)^{1/2}}
\int_{-\infty}^{+\infty} (1 + |m|)^{-\mu} e^{-(\beta - \beta')|m|} dm\\
\times ( \int_{0}^{+\infty} e^{-\delta_{2}^{\infty}r} dr + \int_{0}^{+\infty}
e^{-K_{F}r \cos(\theta_{F}) } dr )
\end{multline}
holds for all $t \in \mathcal{T}_{F;\epsilon}^{\infty}$, $z \in H_{\beta'}$ with $0 < \beta' < \beta$ and
$\epsilon \in \mathcal{E}^{\infty}$. In particular, we see that $(t,z) \mapsto F^{\theta_{F}}(t,z,\epsilon)$
represents a holomorphic bounded function on
$\mathcal{T}_{F;\epsilon}^{\infty} \times H_{\beta'}$ for any $0 < \beta' < \beta$, when $\epsilon$ belongs to
$\mathcal{E}^{\infty}$.\medskip

\noindent {\bf Remark 1.} Let us set
$$ c_{F}(z) = \frac{1}{(2\pi)^{1/2}}\int_{-\infty}^{+\infty} C_{F}(m) e^{izm} dm \ \ , \ \
c_{F_{1},F_{2},\theta_{F}} = \int_{L_{\theta_F}} e^{-K_{F}u}\frac{F_{1}(u)}{F_{2}(u)} du $$
and expand $F_{1}(u) = \sum_{k=0}^{\mathrm{deg}(F_1)} F_{1,k}u^{k}$. According to the identities
concerning the classical Laplace transform coming next
in Lemma 4, we can claim that $F^{\theta_F}(t,z,\epsilon)$ solves the next singularly perturbed inhomogeneous
linear ODE
\begin{equation}
F_{2}(-\epsilon^{\gamma}\partial_{t})F^{\theta_F}(t,z,\epsilon) =
\epsilon^{n_F}c_{F}(z) \left( \sum_{k=0}^{\mathrm{deg}(F_1)} F_{1,k}
\frac{k!}{(K_{F} + \frac{t}{\epsilon^{\gamma}})^{k+1}} - F_{2}(0)c_{F_{1},F_{2},\theta_{F}} \right).
\label{ODE_F_theta_F}
\end{equation}
Our choice for the piece of forcing term $F^{\theta_F}(t,z,\epsilon)$ can be considered as very specific. However,
for the sake of simplicity and clarity, we have taken it as an explicit solution of a inhomogeneous basic
singularly perturbed ODE which is irregular at $t=\infty$ and regular at $t=0$. But all the forthcoming results
disclosed in this paper may also work for forcing terms being solutions of more general singularly perturbed ODEs
sharing the same behaviour at $t=0$ and $t=\infty$ as our model.

\section{Construction of inner solutions to the main problem}

Within this section, we build up solutions of the main equation (\ref{main_PDE_u}) for time $t$ located on
small $\epsilon-$depending sectorial domains in the vicinity of the origin, whose radius is proportional to
some positive power of $|\epsilon|$.

\subsection{Borel-Laplace transforms of order $k$}

In this section, we review some basic statements concerning a $k-$Borel summability method of formal power
series which is a slightly modified version of the more classical procedure (see~\cite{ba}, Section 3.2).
This novel version has already been used in our most recent works such as \cite{lama1}, \cite{ma2}.

\begin{defin}
Let $k\ge1$ be an integer. Let $(m_{k}(n))_{n\ge1}$ be the sequence
$$m_{k}(n)=\Gamma\left(\frac{n}{k}\right)=\int_{0}^{\infty}t^{\frac{n}{k}-1}e^{-t}dt,\qquad n\ge1.$$
Let $(\mathbb{E},\left\|\cdot\right\|_{\mathbb{E}})$ be a complex Banach space. We say a formal power series
$$\hat{X}(T)=\sum_{n=1}^{\infty}a_{n}T^{n}\in T\mathbb{E}[[T]]$$ is $m_{k}-$summable with respect to $T$
in the direction $d \in \mathbb{R}$ if the following assertions hold:
\begin{enumerate}
\item There exists $\rho>0$ such that the $m_{k}-$Borel transform of $\hat{X}$, $\mathcal{B}_{m_{k}}(\hat{X})$, is absolutely convergent for $|\tau|<\rho$, where
$$\mathcal{B}_{m_{k}}(\hat{X})(\tau)=\sum_{n=1}^{\infty}\frac{a_{n}}{\Gamma\left(\frac{n}{k}\right)}\tau^{n}\in\tau\mathbb{E}[[\tau]].$$
\item The series $\mathcal{B}_{m_{k}}(\hat{X})$ can be analytically continued in a sector
$S=\{\tau \in \mathbb{C}^{\star}:|d-\arg(\tau)|<\delta\}$ for some $\delta>0$. In addition to this,
the extension is of exponential growth at most $k$ in $S$, meaning that there exist $C,K>0$ such that
$$\left\|\mathcal{B}_{m_{k}}(\hat{X})(\tau)\right\|_{\mathbb{E}}\le Ce^{K|\tau|^{k}},\quad \tau \in S.$$
\end{enumerate}
Under these assumptions, the vector valued Laplace transform of $\mathcal{B}_{m_{k}}(\hat{X})$ along
direction $d$ is defined by
$$\mathcal{L}_{m_{k}}^{d}\left(\mathcal{B}_{m_{k}}(\hat{X})\right)(T)=
k\int_{L_{\gamma}}\mathcal{B}_{m_{k}}(\hat{X})(u)e^{-(u/T)^k}\frac{du}{u},$$
where $L_{\gamma}$ is the path parametrized by $u\in[0,\infty)\mapsto ue^{i\gamma}$, for some
appropriate direction $\gamma$ depending on $T$, such that $L_{\gamma}\subseteq S$ and
$\cos(k(\gamma-\arg(T)))\ge\Delta>0$ for some $\Delta>0$.

The function $\mathcal{L}_{m_{k}}^{d}(\mathcal{B}_{m_{k}}(\hat{X}))$ is well defined and turns out to be a
holomorphic and bounded function in any sector of the form
$S_{d,\theta,R^{1/k}}=\{T\in\mathbb{C}^{\star}:|T|<R^{1/k},|d-\arg(T)|<\theta/2\}$, for
some $\frac{\pi}{k}<\theta<\frac{\pi}{k}+2\delta$ and $0<R<\Delta/K$. This function is known
as the $m_{k}-$sum of the formal power series $\hat{X}(T)$ in the direction $d$.
\end{defin}

The following are some elementary properties concerning the $m_{k}-$sums of formal power series which will be
crucial in our procedure.

1) The function $\mathcal{L}_{m_{k}}^{d}(\mathcal{B}_{m_{k}}(\hat{X}))(T)$ admits $\hat{X}(T)$ as its
Gevrey asymptotic expansion of order $1/k$ with respect to $T$ in $S_{d,\theta,R^{1/k}}$. More precisely,
for every $\frac{\pi}{k}<\theta_1<\theta$, there exist $C,M>0$ such that
$$\left\|\mathcal{L}^{d}_{m_{k}}(\mathcal{B}_{m_{k}}(\hat{X}))(T)-
\sum_{p=1}^{n-1}a_{p}T^{p}\right\|_{\mathbb{E}}\le CM^{n}\Gamma(1+\frac{n}{k})|T|^{n},$$
for every $n\ge2$ and $T\in S_{d,\theta_{1},R^{1/k}}$. Watson's lemma (see Proposition 11 p.75 in \cite{ba2})
allows us to affirm that $\mathcal{L}^{d}_{m_{k}}(\mathcal{B}_{m_{k}}(\hat{X}))(T)$ is unique provided that
the opening $\theta_1$ is larger than $\frac{\pi}{k}$.

2) Whenever $\mathbb{E}$ is a Banach algebra, the set of holomorphic functions having Gevrey
asymptotic expansion of order $1/k$ on a sector with values in $\mathbb{E}$ turns out
to be a differential algebra (see Theorem 18, 19 and 20 in \cite{ba2}). This, and the
uniqueness provided by Watson's lemma allow us to obtain some properties on $m_{k}-$summable
formal power series in direction $d$.

By $\star$ we denote the product in the Banach algebra and also the Cauchy product of formal power
series with coefficients in $\mathbb{E}$. Let $\hat{X}_{1}$, $\hat{X}_{2}\in T\mathbb{E}[[T]]$ be
$m_{k}-$summable formal power series in direction $d$. Let $q_1\ge q_2\ge1$ be integers.
Then $ \hat{X}_{1}+\hat{X}_{2}$, $\hat{X}_{1}\star \hat{X}_{2}$ and $T^{q_1}\partial_{T}^{q_2}\hat{X}_{1}$,
which are elements of $T\mathbb{E}[[T]]$, are $m_{k}-$summable in direction $d$. Moreover, one has
$$\mathcal{L}_{m_{k}}^{d}(\mathcal{B}_{m_{k}}(\hat{X}_{1}))(T)+
\mathcal{L}_{m_{k}}^{d}(\mathcal{B}_{m_{k}}(\hat{X}_{2}))(T)=
\mathcal{L}_{m_{k}}^{d}(\mathcal{B}_{m_{k}}(\hat{X}_{1}+\hat{X}_{2}))(T),$$
$$\mathcal{L}_{m_{k}}^{d}(\mathcal{B}_{m_{k}}(\hat{X}_{1}))(T)\star
\mathcal{L}_{m_{k}}^{d}(\mathcal{B}_{m_{k}}(\hat{X}_{2}))(T)=
\mathcal{L}_{m_{k}}^{d}(\mathcal{B}_{m_{k}}(\hat{X}_{1}\star\hat{X}_{2}))(T),$$
$$T^{q_1}\partial_{T}^{q_2}\mathcal{L}^{d}_{m_{k}}(\mathcal{B}_{m_{k}}(\hat{X}_{1}))(T)
=\mathcal{L}_{m_{k}}^{d}(\mathcal{B}_{m_{k}}(T^{q_1}\partial_{T}^{q_2}\hat{X}_{1}))(T),$$
for every $T\in S_{d,\theta,R^{1/k}}$.

The next proposition is written without proof for it can be found in \cite{lama1}, Proposition 6.

\begin{prop}
Let $\hat{f}(t)=\sum_{n \geq 1}f_nt^n$ and $\hat{g}(t) = \sum_{n \geq 1} g_{n} t^{n}$ that belong to
$\mathbb{E}[[t]]$, where $(\mathbb{E},\left\|\cdot\right\|_{\mathbb{E}})$ is a Banach algebra. Let
$k,m\ge1$ be integers. The following formal identities hold.
$$\mathcal{B}_{m_{k}}(t^{k+1}\partial_{t}\hat{f}(t))(\tau)=k\tau^{k}\mathcal{B}_{m_{k}}(\hat{f}(t))(\tau),$$
$$\mathcal{B}_{m_{k}}(t^{m}\hat{f}(t))(\tau)=
\frac{\tau^{k}}{\Gamma\left(\frac{m}{k}\right)}\int_{0}^{\tau^{k}}
(\tau^{k}-s)^{\frac{m}{k}-1}\mathcal{B}_{m_{k}}(\hat{f}(t))(s^{1/k})\frac{ds}{s}$$
and
$$
\mathcal{B}_{m_k}( \hat{f}(t) \star \hat{g}(t) )(\tau) = \tau^{k}\int_{0}^{\tau^{k}}
\mathcal{B}_{m_k}(\hat{f}(t))((\tau^{k}-s)^{1/k}) \star \mathcal{B}_{m_k}(\hat{g}(t))(s^{1/k})
\frac{1}{(\tau^{k}-s)s} ds.
$$
\end{prop}

\subsection{Banach spaces with exponential growth and exponential decay}

In this section, we recall the definition and display useful properties of Banach spaces as defined in our
previous work, \cite{ma2}. We denote $D(0,\rho)$ the open disc centered at $0$ with radius $\rho>0$ in
$\mathbb{C}$ and by $\bar{D}(0,\rho)$ its closure. Let 
$S_{d}$ be an open unbounded sector in direction $d \in \mathbb{R}$ and $\mathcal{E}$ be an
open sector with finite radius $r_{\mathcal{E}}$, both centered at $0$ in $\mathbb{C}$. By
convention, these sectors do not contain the origin in $\mathbb{C}$.\medskip

\begin{defin} Let $\nu,\rho>0$ and $\beta>0,\mu>1$ be real numbers. Let $\kappa \geq 1$ be an integer and
$\chi > 0$ be some real number. Let $\epsilon \in \mathcal{E}$. We denote $F_{(\nu,\beta,\mu,\chi,\kappa,\epsilon)}^{d}$
the vector space of continuous functions $(\tau,m) \mapsto h(\tau,m)$ on
$(\bar{D}(0,\rho) \cup S_{d}) \times \mathbb{R}$, which are holomorphic w.r.t $\tau$ on $D(0,\rho) \cup S_{d}$
and such that
\begin{multline*}
||h(\tau,m)||_{(\nu,\beta,\mu,\chi,\kappa,\epsilon)} \\=
\sup_{\tau \in \bar{D}(0,\rho) \cup S_{d},m \in \mathbb{R}} (1+|m|)^{\mu}\exp(\beta|m|)
\frac{1 + |\frac{\tau}{\epsilon^{\chi}}|^{2\kappa}}{|\frac{\tau}{\epsilon^{\chi}}|}
\exp( - \nu|\frac{\tau}{\epsilon^{\chi}}|^{\kappa} ) |h(\tau,m)|
\end{multline*}
is finite. One can check that the normed space
$(F_{(\nu,\beta,\mu,\chi,\kappa,\epsilon)}^{d},||.||_{(\nu,\beta,\mu,\chi,\kappa,\epsilon)})$ is a
Banach space.
\end{defin}

\noindent Throughout the whole section, we keep the notations of Definitions in this section.\medskip

\noindent The next lemma and proposition are kept almost unchanged as stated in Section 2 of \cite{ma2} and
we decide to omit their proofs for avoiding overlapping with our previous work.

\begin{lemma} Let $\gamma_{1} \geq 0$, $\gamma_{2} \geq 1$ be integers and $\gamma_{3} \in \mathbb{R}$.
Let $R(X)$ be a polynomial that belongs to $\mathbb{C}[X]$ such that
$R(im) \neq 0$ for all $m \in \mathbb{R}$. We take a function $B(m)$ located in $E_{(\beta,\mu)}$
and we consider a continuous function
$a_{\gamma_{1},\kappa}(\tau,m)$ on $(\bar{D}(0,\rho) \cup S_{d}) \times \mathbb{R}$,
holomorphic w.r.t $\tau$ on $D(0,\rho) \cup S_{d}$ such that
$$|a_{\gamma_{1},\kappa}(\tau,m)| \leq \frac{1}{(1+|\tau|^{\kappa})^{\gamma_1}|R(im)|}$$
for all $\tau \in \bar{D}(0,\rho) \cup S_{d}$, all $m \in \mathbb{R}$.

Then, the function $\epsilon^{-\gamma_{3}} \tau^{\gamma_2}B(m)a_{\gamma_{1},\kappa}(\tau,m)$ belongs
to $F_{(\nu,\beta,\mu,\chi,\kappa,\epsilon)}^{d}$. Moreover, there exists a constant
$C_{1}>0$ (depending on $\nu$,$\kappa$ and $\gamma_{2}$) such that
\begin{equation}
||\epsilon^{-\gamma_{3}} \tau^{\gamma_2} B(m)
a_{\gamma_{1},\kappa}(\tau,m) ||_{(\nu,\beta,\mu,\chi,\kappa,\epsilon)} \leq
C_{1}\frac{||B(m)||_{(\beta,\mu)}}{\mathrm{inf}_{m \in \mathbb{R}}|R(im)|}
|\epsilon|^{\chi \gamma_{2} - \gamma_{3}}
\end{equation}
for all $\epsilon \in \mathcal{E}$.
\end{lemma}

\begin{prop} Let $\gamma_{j}$, $0 \leq j \leq 3$, be real numbers with $\gamma_{1} \geq 0$.
Let $R(X),R_{D}(X)$ be polynomials
with complex coefficients such that $\mathrm{deg}(R) \leq \mathrm{deg}(R_{D})$ and with
$R_{D}(im) \neq 0$
for all $m \in \mathbb{R}$. We consider a continuous function
$a_{\gamma_{1},\kappa}(\tau,m)$ on $(\bar{D}(0,\rho) \cup S_{d}) \times \mathbb{R}$,
holomorphic w.r.t $\tau$ on $D(0,\rho) \cup S_{d}$ such that
$$|a_{\gamma_{1},\kappa}(\tau,m)| \leq \frac{1}{(1+|\tau|^{\kappa})^{\gamma_1}|R_{D}(im)|}$$
for all $\tau \in \bar{D}(0,\rho) \cup S_{d}$, all $m \in \mathbb{R}$. We make the next assumptions
\begin{equation}
\frac{1}{\kappa} + \gamma_{3} + 1 > 0 \ \ , \ \ \gamma_{2}+\gamma_{3}+2 \geq 0 \ \ , \ \
\gamma_{2} > -1. \label{constraints_gamma_j_kappa_conv_op}
\end{equation}
1) If $1 + \gamma_{3} \leq 0$, then there exists a constant $C_{2}>0$ (depending on
$\nu,\kappa,\gamma_{2},\gamma_{3}$ and $R(X),R_{D}(X)$) such that
\begin{multline}
||\epsilon^{-\gamma_{0}} a_{\gamma_{1},\kappa}(\tau,m) R(im) \tau^{\kappa}
\int_{0}^{\tau^{\kappa}} (\tau^{\kappa} - s)^{\gamma_2} s^{\gamma_3}
f(s^{1/\kappa},m) ds||_{(\nu,\beta,\mu,\chi,\kappa,\epsilon)}\\
\leq C_{2}|\epsilon|^{\chi \kappa(\gamma_{2}+\gamma_{3}+2)-\gamma_{0}}
||f(\tau,m)||_{(\nu,\beta,\mu,\chi,\kappa,\epsilon)} \label{norm_estim_conv_op_1}
\end{multline}
for all $f(\tau,m) \in F_{(\nu,\beta,\mu,\chi,\kappa,\epsilon)}^{d}$.\\
2) If $1 + \gamma_{3} > 0$ and $\gamma_{1} \geq 1 + \gamma_{3}$, then there exists a constant
$C_{2}'>0$ (depending on $\nu,\kappa,\gamma_{1},\gamma_{2},\gamma_{3}$ and $R(X),R_{D}(X)$) such that
\begin{multline}
||\epsilon^{-\gamma_{0}} a_{\gamma_{1},\kappa}(\tau,m) R(im) \tau^{\kappa}
\int_{0}^{\tau^{\kappa}} (\tau^{\kappa} - s)^{\gamma_2} s^{\gamma_3}
f(s^{1/\kappa},m) ds||_{(\nu,\beta,\mu,\chi,\kappa,\epsilon)}\\
\leq C_{2}'|\epsilon|^{\chi \kappa(\gamma_{2}+\gamma_{3}+2)
-\gamma_{0}- \chi \kappa \gamma_{1}}
||f(\tau,m)||_{(\nu,\beta,\mu,\chi,\kappa,\epsilon)} \label{norm_estim_conv_op_2}
\end{multline}
for all $f(\tau,m) \in F_{(\nu,\beta,\mu,\chi,\kappa,\epsilon)}^{d}$.\\
\end{prop}
The forthcoming proposition presents norms estimates for some bilinear convolution operators acting on
the aforementioned Banach spaces.
\begin{prop} Let $R_{D}(X),Q_{1}(X)$ and $Q_{2}(X)$ belonging to $\mathbb{C}[X]$ such that
$R_{D}(im) \neq 0$ for all $m \in \mathbb{R}$. Assume that
$$ \mathrm{deg}(R_{D}) \geq \mathrm{deg}(Q_1), \ \ \mathrm{deg}(R_{D}) \geq \mathrm{deg}(Q_2) $$
and choose the real parameter $\mu$ such that
\begin{equation}
\mu > \max( \mathrm{deg}(Q_1) + 1, \mathrm{deg}(Q_2) + 1). \label{mu>deg_Q1_deg_Q2}
\end{equation}
Let $a(m)$ be a continuous function on $\mathbb{R}$ such that
$$ |a(m)| \leq \frac{1}{|R_{D}(im)|} $$
for all $m \in \mathbb{R}$. Then, there exists a constant $C_{3}>0$ (depending on $Q_{1},Q_{2},R_{D},\mu$ and
$\kappa$) such that
\begin{multline}
|| \tau^{\kappa - 1}a(m) \int_{0}^{\tau^{\kappa}} \int_{-\infty}^{+\infty}
Q_{1}(i(m-m_{1}))f( (\tau^{\kappa} - s')^{1/\kappa},m-m_{1})\\
\times Q_{2}(im_{1})g( (s')^{1/\kappa},m_{1})
\frac{1}{(\tau^{\kappa}-s')s'} ds' dm_{1} ||_{(\nu,\beta,\mu,\chi,\kappa,\epsilon)}\\
\leq \frac{C_3}{|\epsilon|^{\chi}} ||f(\tau,m)||_{(\nu,\beta,\mu,\chi,\kappa,\epsilon)}
||g(\tau,m)||_{(\nu,\beta,\mu,\chi,\kappa,\epsilon)}
\end{multline}
for all $f(\tau,m),g(\tau,m) \in F_{(\nu,\beta,\mu,\chi,\kappa,\epsilon)}^{d}$.
\end{prop}
\begin{proof} We follow similar steps as in the proof of Proposition 3 from \cite{ma2}. By definition
of the norm, we can write
\begin{multline}
B = || \tau^{\kappa - 1}a(m) \int_{0}^{\tau^{\kappa}} \int_{-\infty}^{+\infty}
Q_{1}(i(m-m_{1}))f( (\tau^{\kappa} - s')^{1/\kappa},m-m_{1}) \\
\times Q_{2}(im_{1})g( (s')^{1/\kappa},m_{1})
\frac{1}{(\tau^{\kappa}-s')s'} ds' dm_{1} ||_{(\nu,\beta,\mu,\chi,\kappa,\epsilon)}\\
= \sup_{\tau \in \bar{D}(0,\rho) \cup S_{d},m \in \mathbb{R}}
(1 + |m|)^{\mu} \exp(\beta |m|) \frac{1 + |\frac{\tau}{\epsilon^{\chi}}|^{2\kappa}}{
|\frac{\tau}{\epsilon^{\chi}}|}
\exp( -\nu |\frac{\tau}{\epsilon^{\chi}}|^{\kappa} )|a(m)|\\
\times |\tau^{\kappa-1}\int_{0}^{\tau^{\kappa}}
\int_{-\infty}^{+\infty} \{ (1 + |m-m_{1}|)^{\mu} \exp( \beta |m-m_{1}|)\\
\times
\frac{1 + \frac{|\tau^{\kappa} - s'|^{2}}{|\epsilon|^{\chi 2\kappa}}}{
\frac{|\tau^{\kappa} - s'|^{1/\kappa}}{|\epsilon|^{\chi}}}
\exp( -\nu \frac{|\tau^{\kappa} - s'|}{|\epsilon|^{\chi \kappa}} )
f((\tau^{\kappa} - s')^{1/\kappa},m-m_{1}) \}\\
\times
\{ (1 + |m_{1}|)^{\mu} \exp( \beta |m_{1}| ) \frac{ 1 + \frac{|s'|^{2}}{|\epsilon|^{\chi 2\kappa}} }{
\frac{|s'|^{1/\kappa}}{|\epsilon|^{\chi}} }
\exp(-\nu \frac{|s'|}{|\epsilon|^{\chi \kappa}} ) g((s')^{1/\kappa},m_{1}) \}
\times \mathcal{B}(\tau,s,m,m_{1}) ds' dm_{1} | \label{defin_B}
\end{multline}
where
\begin{multline*}
\mathcal{B}(\tau,s,m,m_{1}) = \frac{\exp( -\beta |m-m_{1}| ) \exp( -\beta |m_{1}| )
|Q_{1}(i(m-m_{1}))| |Q_{2}(im_{1})| }{(1 + |m-m_{1}|)^{\mu}(1 + |m_{1}|)^{\mu}}\\
\times
\frac{ \frac{|s'|^{1/\kappa} |\tau^{\kappa} - s'|^{1/\kappa}}{|\epsilon|^{2\chi}} }{
(1 + \frac{|\tau^{\kappa} - s'|^{2}}{|\epsilon|^{\chi 2\kappa}})
(1 + \frac{|s'|^{2}}{|\epsilon|^{\chi 2\kappa}})} \exp( \nu \frac{|\tau^{\kappa} - s'|}{|\epsilon|^{\chi \kappa}} )
\exp( \nu \frac{|s'|}{|\epsilon|^{\chi \kappa}} )
\frac{1}{(\tau^{\kappa} - s')s'}.
\end{multline*}
By definition of the norms of $f$ and $g$ and according to the triangular inequality
$|m| \leq |m-m_{1}| + |m_{1}|$ for all $m,m_{1} \in \mathbb{R}$, we deduce that
\begin{equation}
B \leq C_{3}(\epsilon) ||f(\tau,m)||_{(\nu,\beta,\mu,\chi,\kappa,\epsilon)}
||g(\tau,m)||_{(\nu,\beta,\mu,\chi,\kappa,\epsilon)} \label{B<=norm_f_times_norm_g} 
\end{equation}
where
\begin{multline*}
C_{3}(\epsilon) = \sup_{\tau \in \bar{D}(0,\rho) \cup S_{d},m \in \mathbb{R}}
(1 + |m|)^{\mu} \frac{1 + |\frac{\tau}{\epsilon^{\chi}}|^{2\kappa}}{
|\frac{\tau}{\epsilon^{\chi}}|}
\exp( -\nu |\frac{\tau}{\epsilon^{\chi}}|^{\kappa} )
|\tau|^{\kappa - 1} |a(m)|\\
\times \int_{0}^{|\tau|^{\kappa}}
\int_{-\infty}^{+\infty}
\frac{|Q_{1}(i(m-m_{1}))| |Q_{2}(im_{1})| }{(1 + |m-m_{1}|)^{\mu}(1 + |m_{1}|)^{\mu}}
\frac{(h')^{1/\kappa} (|\tau|^{\kappa} - h')^{1/\kappa}}{|\epsilon|^{2\chi}}
\frac{1}{(1 + \frac{(|\tau|^{\kappa} - h')^{2}}{|\epsilon|^{\chi 2\kappa}})
(1 + \frac{(h')^{2}}{|\epsilon|^{\chi 2\kappa}})}\\
\times 
\exp( \nu \frac{|\tau|^{\kappa} - h'}{|\epsilon|^{\chi \kappa}} )
\exp( \nu \frac{h'}{|\epsilon|^{\chi \kappa}} )
\frac{1}{(|\tau|^{\kappa} - h')h'} dh' dm_{1}.
\end{multline*}
We provide upper bounds that can be split in two parts,
\begin{equation}
C_{3}(\epsilon) \leq C_{3.1}C_{3.2}(\epsilon) \label{C3<=C31_C32} 
\end{equation}
where
\begin{equation}
C_{3.1} = \sup_{m \in \mathbb{R}} \frac{(1 + |m|)^{\mu}}{|R_{D}(im)|}
\int_{-\infty}^{+\infty} \frac{|Q_{1}(i(m-m_{1}))|
|Q_{2}(im_{1})| }{(1 + |m-m_{1}|)^{\mu}(1 + |m_{1}|)^{\mu}} dm_{1} \label{C31_defin}
\end{equation}
and
$$
C_{3.2}(\epsilon) = \sup_{\tau \in \bar{D}(0,\rho) \cup S_{d}}
\frac{1 + |\frac{\tau}{\epsilon^{\chi}}|^{2\kappa}}{
|\frac{\tau}{\epsilon^{\chi}}|}|\tau|^{\kappa - 1}
\int_{0}^{|\tau|^{\kappa}} \frac{
\frac{(h')^{1/\kappa} (|\tau|^{\kappa} - h')^{1/\kappa}}{|\epsilon|^{2\chi}} }{
(1 + \frac{(|\tau|^{\kappa} - h')^{2}}{|\epsilon|^{\chi 2\kappa}})
(1 + \frac{(h')^{2}}{|\epsilon|^{\chi 2\kappa}})}
\frac{1}{(|\tau|^{\kappa} - h')h'} dh'.
$$
By construction, we can select three constants $\mathfrak{Q}_{1},\mathfrak{Q}_{2},\mathfrak{R}>0$ such that
\begin{multline}
|Q_{1}(i(m-m_{1}))| \leq \mathfrak{Q}_{1} (1 + |m-m_{1}|)^{\mathrm{deg}(Q_1)} \ \ , \ \
|Q_{2}(im_{1})| \leq \mathfrak{Q}_{1} (1 + |m_{1}|)^{\mathrm{deg}(Q_2)},\\
|R_{D}(im)| \geq \mathfrak{R}(1 + |m|)^{\mathrm{deg}(R_{D})}
\end{multline}
for all $m,m_{1} \in \mathbb{R}$. We deduce that
\begin{equation}
C_{3.1} \leq \frac{\mathfrak{Q}_{1} \mathfrak{Q}_{2}}{\mathfrak{R}}
\sup_{m \in \mathbb{R}}(1 + |m|)^{\mu - \mathrm{deg}(R_{D})}\int_{-\infty}^{+\infty}
\frac{1}{(1 + |m-m_{1}|)^{\mu - \mathrm{deg}(Q_1)}(1 + |m_{1}|)^{\mu - \mathrm{deg}(Q_2)}} dm_{1}
\label{bounds_C31}
\end{equation}
which is finite under the condition (\ref{mu>deg_Q1_deg_Q2}) according to Lemma 4 of \cite{ma}.

On the other hand, with the help of the estimates (23) and (24) from \cite{ma2}, we conclude that a constant $C_{3.2}>0$ can be picked
out (depending exclusively on $\kappa$) with
\begin{equation}
C_{3.2}(\epsilon) \leq \frac{C_{3.2}}{|\epsilon|^{\chi}} \label{C32_bounds}
\end{equation}
We finish the proof by collecting (\ref{defin_B}), (\ref{B<=norm_f_times_norm_g}), (\ref{C3<=C31_C32}),
(\ref{C31_defin}), (\ref{bounds_C31}) and (\ref{C32_bounds}) which leads to the statement of Proposition 5.
\end{proof}

\subsection{Construction of formal solutions}
Within this section, we search for time rescaled solutions to (\ref{main_PDE_u}) of the form
$$ u(t,z,\epsilon) = \epsilon^{-m_{0}}U(\epsilon^{\alpha}t,z,\epsilon) $$
where $\alpha \in \mathbb{Q}$. One can check that the expression $U(T,z,\epsilon)$ formally solves the next
nonlinear PDE
\begin{multline}
( \sum_{l=1}^{q}a_{l} \epsilon^{m_{l}-m_{0}-\alpha k_{l}}T^{k_l} + a_{0})Q(\partial_{z})U(T,z,\epsilon)\\+
(\sum_{l=0}^{M}c_{l} \epsilon^{\mu_{l} - 2m_{0} - \alpha h_{l}} T^{h_l})
Q_{1}(\partial_{z})U(T,z,\epsilon)Q_{2}(\partial_{z})U(T,z,\epsilon)\\
= \sum_{j=0}^{Q}b_{j}(z) \epsilon^{n_{j} - \alpha b_{j}} T^{b_{j}} + F^{\theta_{F}}(\epsilon^{-\alpha}T,z,\epsilon)
+ \sum_{l=1}^{D} \epsilon^{\Delta_{l}+ \alpha (\delta_{l}-d_{l}) - m_{0}} T^{d_{l}} R_{l}(\partial_{z})
 \partial_{T}^{\delta_l}U(T,z,\epsilon). \label{main_PDE_U}
\end{multline}
We make the next further assumptions. We choose $\alpha$ such that
\begin{equation}
\Delta_{D} + \alpha(\delta_{D}-d_{D})-m_{0} = 0 \label{cond_deltaD_alpha} 
\end{equation}
We suppose the existence of a positive integer $\kappa \geq 1$ with
\begin{equation}
d_{D} = \delta_{D}(\kappa + 1) \ \ , \ \ d_{l} = \delta_{l}(\kappa + 1) + d_{l,\kappa}  \label{cond_deltal_alpha}
\end{equation}
for some integers $d_{l,\kappa} \geq 1$, for all $1 \leq l \leq D-1$. With the help of the formula (8.7) from
\cite{taya} p. 3630, we can expand the following pieces of (\ref{main_PDE_U}) in term of the irregular operator
$T^{\kappa + 1}\partial_{T}$,
\begin{multline}
T^{d_{D}}\partial_{T}^{\delta_{D}}R_{D}(\partial_{z})U(T,z,\epsilon) =
T^{\delta_{D}(\kappa + 1)}\partial_{T}^{\delta_{D}}R_{D}(\partial_{z})U(T,z,\epsilon)=\\
R_{D}(\partial_{z}) \left( (T^{\kappa+1}\partial_{T})^{\delta_{D}} + 
\sum_{1 \leq p \leq \delta_{D}-1} A_{\delta_{D},p} T^{\kappa(\delta_{D}-p)}(T^{\kappa+1}\partial_{T})^{p} \right)
U(T,z,\epsilon) \label{expand_T_partial_T_delta_D_U}
\end{multline}
for some real numbers $A_{\delta_{D},p}$, $1 \leq p \leq \delta_{D}-1$ and
\begin{multline}
T^{d_{l}}\partial_{T}^{\delta_{l}}R_{l}(\partial_{z})U(T,z,\epsilon) =
T^{d_{l,\kappa}} T^{\delta_{l}(\kappa + 1)} \partial_{T}^{\delta_l}R_{l}(\partial_{z})U(T,z,\epsilon)=\\
R_{l}(\partial_{z})T^{d_{l,\kappa}} \left( (T^{\kappa+1}\partial_{T})^{\delta_l} +
\sum_{1 \leq p \leq \delta_{l}-1} A_{\delta_{l},p} T^{\kappa(\delta_{l}-p)}(T^{\kappa+1}\partial_{T})^{p} \right)
U(T,z,\epsilon) \label{expand_T_partial_T_delta_l_U}
\end{multline}
for well chosen real numbers $A_{\delta_{l},p}$, $1 \leq p \leq \delta_{l}-1$. Notice that, by convention,
the sum $\sum_{1 \leq p \leq \delta_{l}-1}[..]$ appearing in (\ref{expand_T_partial_T_delta_l_U}) is vanishing
provided that $\delta_{l}=1$.

We now furnish the formal Taylor expansion of the part of the forcing term
$F^{\theta_{F}}(\epsilon^{-\alpha}T,z,\epsilon)$ with respect to $T$ at $T=0$. Making use of the convergent
Taylor expansion
of $\exp(-Tu/\epsilon^{\gamma+\alpha}) - 1$ w.r.t $u$ at $u=0$, we can write
\begin{equation}
F^{\theta_{F}}(\epsilon^{-\alpha}T,z,\epsilon) = \sum_{n \geq 1} F_{n}(z,\epsilon) T^{n} \label{Taylor_F_wrt_T}
\end{equation}
where the coefficients $F_{n}(z,\epsilon)$ are expressed as an inverse Fourier transform
$$ F_{n}(z,\epsilon) = \mathcal{F}^{-1}(m \mapsto \psi_{n}(m,\epsilon))(z) $$
\begin{equation}
\psi_{n}(m,\epsilon) =  \epsilon^{n_{F}} \int_{L_{\theta_{F}}}
e^{-K_{F}u} \frac{F_{1}(u)}{F_{2}(u)} \frac{u^{n}}{n!} du C_{F}(m) (-\frac{1}{\epsilon^{\gamma + \alpha}})^{n}
\end{equation}
for all $n \geq 1$. Let us assume now, that the expression $U(T,z,\epsilon)$ has a formal power series expansion
\begin{equation}
U(T,z,\epsilon) = \sum_{n \geq 1} U_{n}(z,\epsilon) T^{n} \label{Taylor_U_wrt_T}
\end{equation}
where each coefficient $U_{n}(z,\epsilon)$ is defined as an inverse Fourier transform
$$ U_{n}(z,\epsilon) = \mathcal{F}^{-1}(m \mapsto \omega_{n}(m,\epsilon) )(z) $$
for some function $m \mapsto \omega_{n}(m,\epsilon)$ belonging to the Banach space $E_{(\beta,\mu)}$
and relying analytically on the parameter $\epsilon$ on some punctured disc $D(0,\epsilon_{0}) \setminus \{ 0 \}$
centered at 0 with radius $\epsilon_{0}>0$. We consider the next formal series
$$ \omega_{\kappa}(\tau,m,\epsilon) = \sum_{n \geq 1} \frac{\omega_{n}(m,\epsilon)}{\Gamma(n/\kappa)} \tau^{n} $$
obtained by formally applying a $m_{\kappa}-$Borel transform w.r.t $T$ and Fourier transform w.r.t $z$ to the
formal series (\ref{Taylor_U_wrt_T}).
We also introduce $\psi_{\kappa}(\tau,m,\epsilon)$ realized as a $m_{\kappa}-$Borel transform w.r.t $T$ and
Fourier transform w.r.t $z$ of the formal series (\ref{Taylor_F_wrt_T}),
$$ \Psi_{\kappa}(\tau,m,\epsilon) = \sum_{n \geq 1} \psi_{n}(m,\epsilon)
\frac{\tau^{n}}{\Gamma(\frac{n}{\kappa})}. $$
Under the restrictions (\ref{cond_deltaD_alpha}) and (\ref{cond_deltal_alpha}), we check that
$\omega_{\kappa}(\tau,m,\epsilon)$ must satisfy some nonlinear integral convolution equation by making use of the
properties of the $m_{\kappa}-$Borel transform listed in Proposition 3 and Fourier inverse transform discussed
in Proposition 2, with the help of the prepared expansions (\ref{expand_T_partial_T_delta_D_U}),
(\ref{expand_T_partial_T_delta_l_U}). Namely, we get the next problem
\begin{multline}
Q(im) \left( \sum_{l=1}^{q}a_{l}\epsilon^{m_{l} - m_{0} - \alpha k_{l}}
\frac{\tau^{\kappa}}{\Gamma(\frac{k_l}{\kappa})}
\int_{0}^{\tau^{\kappa}} (\tau^{\kappa}-s)^{\frac{k_l}{\kappa}-1}\omega_{\kappa}(s^{1/\kappa},m,\epsilon)
\frac{ds}{s} + a_{0}\omega_{\kappa}(\tau,m,\epsilon) \right)\\
+ \sum_{l=0}^{M} c_{l}\epsilon^{\mu_{l} - 2m_{0} - \alpha h_{l}}
\frac{\tau^{\kappa}}{\Gamma(\frac{h_l}{\kappa})} \int_{0}^{\tau^{\kappa}}
(\tau^{\kappa} - s)^{\frac{h_l}{\kappa} - 1}\\
\times \left( s \int_{0}^{s} \frac{1}{(2\pi)^{1/2}}
\int_{-\infty}^{+\infty} Q_{1}(i(m-m_{1}))\omega_{\kappa}((s-s')^{1/\kappa},m-m_{1},\epsilon) \right. \\
\left. \times Q_{2}(im_{1})\omega_{\kappa}((s')^{1/\kappa},m_{1},\epsilon)dm_{1} \frac{1}{(s-s')s'} ds' \right)
\frac{ds}{s} = \sum_{j=0}^{Q} B_{j}(m) \epsilon^{n_{j} - \alpha b_{j}}
\frac{\tau^{b_j}}{\Gamma(\frac{b_j}{\kappa})} + \Psi_{\kappa}(\tau,m,\epsilon)\\
+ R_{D}(im) \left( (\kappa \tau^{\kappa})^{\delta_{D}}\omega_{\kappa}(\tau,m,\epsilon) +
\sum_{1 \leq p \leq \delta_{D}-1} A_{\delta_{D},p}
\frac{\tau^{\kappa}}{\Gamma( \delta_{D}-p )} \right. \\
\times \left.
\int_{0}^{\tau^{\kappa}} (\tau^{\kappa} - s)^{\delta_{D}-p-1} (\kappa s)^{p}
\omega_{\kappa}(s^{1/\kappa},m,\epsilon) \frac{ds}{s} \right) + \sum_{l=1}^{D-1}
\epsilon^{\Delta_{l} + \alpha(\delta_{l} - d_{l}) - m_{0}} R_{l}(im)\\
\times
\left( \frac{\tau^{\kappa}}{\Gamma( \frac{d_{l,\kappa}}{\kappa} )}
\int_{0}^{\tau^{\kappa}}(\tau^{\kappa} - s)^{\frac{d_{l,\kappa}}{\kappa} - 1}
(\kappa s)^{\delta_{l}} \omega_{\kappa}(s^{1/\kappa},m,\epsilon) \frac{ds}{s} +
\sum_{1 \leq p \leq \delta_{l}-1} A_{\delta_{l},p} \right. \\
\left. \times 
\frac{\tau^{\kappa}}{\Gamma( \frac{d_{l,\kappa} + \kappa(\delta_{l}-p)}{\kappa} )}
\int_{0}^{\tau^{\kappa}} (\tau^{\kappa} - s)^{\frac{d_{l,\kappa} + \kappa(\delta_{l}-p)}{\kappa} - 1}
(\kappa s)^{p} \omega_{\kappa}(s^{1/\kappa},m,\epsilon) \frac{ds}{s} \right) \label{main_conv_eq_omega_kappa}
\end{multline}
As above, we assume by convention that the sum $\sum_{1 \leq p \leq \delta_{l}-1} [..]$ appearing in
(\ref{main_conv_eq_omega_kappa}) vanishes whenever $\delta_{l}=1$.

\subsection{Analytic and continuous solutions of a nonlinear convolution equation with complex parameter}

Our principal aim is the construction of a unique solution of the problem
(\ref{main_conv_eq_omega_kappa}) inside the Banach space described in Subsection 3.2.

\noindent We make the following further assumptions. The conditions below are very similar to the ones
proposed in Section 4 of \cite{lama1} and in Section 5 of \cite{ma2}.
Namely, we demand that there exists an unbounded sector
$$ S_{Q,R_{D}} = \{ z \in \mathbb{C} / |z| \geq r_{Q,R_{D}} \ \ , \ \
|\mathrm{arg}(z) - d_{Q,R_{D}}| \leq \eta_{Q,R_{D}} \} $$
with direction $d_{Q,R_{D}} \in \mathbb{R}$, aperture $\eta_{Q,R_{D}}>0$ for
some radius $r_{Q,R_{D}}>0$ such that
\begin{equation}
\frac{Q(im)}{R_{D}(im)} \in S_{Q,R_{D}} \label{quotient_Q_RD_in_S}
\end{equation} 
for all $m \in \mathbb{R}$. The polynomial $P_{m}(\tau) = Q(im)a_{0} -
R_{D}(im)\kappa^{\delta_D}
\tau^{\delta_{D}\kappa}$ can be factorized in the form
\begin{equation}
P_{m}(\tau) =
-R_{D}(im)\kappa^{\delta_D}\Pi_{l=0}^{\delta_{D}\kappa-1} (\tau - q_{l}(m)) \label{factor_P_m}
\end{equation}
where
\begin{equation}
q_{l}(m) = (\frac{|a_{0}Q(im)|}{|R_{D}(im)|\kappa^{\delta_{D}}})^{\frac{1}{\delta_{D}\kappa}}
\exp( \sqrt{-1}( \mathrm{arg}( \frac{a_{0}Q(im)}{R_{D}(im)\kappa^{\delta_{D}}})
\frac{1}{\delta_{D}\kappa} + \frac{2\pi l}{\delta_{D}\kappa} ) ) \label{defin_roots}
\end{equation}
for all $0 \leq l \leq \delta_{D}\kappa-1$, all $m \in \mathbb{R}$.

We select an unbounded sector $S_{d}$ centered at 0, a small closed disc $\bar{D}(0,\rho)$ and we require
the sector $S_{Q,R_{D}}$ to fulfill the next conditions.\medskip

\noindent 1) There exists a constant $M_{1}>0$ such that
\begin{equation}
|\tau - q_{l}(m)| \geq M_{1}(1 + |\tau|) \label{root_cond_1}
\end{equation}
for all $0 \leq l \leq \delta_{D}\kappa-1$, all $m \in \mathbb{R}$, all $\tau \in S_{d} \cup \bar{D}(0,\rho)$.
Indeed,
from (\ref{quotient_Q_RD_in_S}) and the explicit expression (\ref{defin_roots}) of $q_{l}(m)$, we first observe that
$|q_{l}(m)| > 2\rho$ for every $m \in \mathbb{R}$, all $0 \leq l \leq \delta_{D}\kappa-1$ for an appropriate
choice of $r_{Q,R_{D}}$
and of $\rho>0$. We also see that for all $m \in \mathbb{R}$, all $0 \leq l \leq \delta_{D}\kappa-1$, the roots
$q_{l}(m)$ remain in a union $\mathcal{U}$ of unbounded sectors centered at 0 that do not cover a full neighborhood
of
the origin in $\mathbb{C}^{\ast}$ provided that $\eta_{Q,R_{D}}$ is small enough. Therefore,
one can choose an adequate
sector $S_{d}$ such that $S_{d} \cap \mathcal{U} = \emptyset$ with the property that for all
$0 \leq l \leq \delta_{D}\kappa-1$ the quotients $q_{l}(m)/\tau$ lay outside
some small disc centered at 1 in $\mathbb{C}$ for all $\tau \in S_{d}$, all $m \in \mathbb{R}$. This yields
(\ref{root_cond_1})
for some small constant $M_{1}>0$.\medskip

\noindent 2) There exists a constant $M_{2}>0$ such that
\begin{equation}
|\tau - q_{l_0}(m)| \geq M_{2}|q_{l_0}(m)| \label{root_cond_2}
\end{equation}
for some $l_{0} \in \{0,\ldots,\delta_{D}\kappa-1 \}$, all $m \in \mathbb{R}$, all
$\tau \in S_{d} \cup \bar{D}(0,\rho)$. Indeed, for the
sector $S_{d}$ and the disc $\bar{D}(0,\rho)$ chosen as above in 1), we notice that for any fixed
$0 \leq l_{0} \leq \delta_{D}\kappa-1$, the quotient $\tau/q_{l_0}(m)$ stays outside a small disc centered at 1
in $\mathbb{C}$
for all $\tau \in S_{d} \cup \bar{D}(0,\rho)$, all $m \in \mathbb{R}$. Hence (\ref{root_cond_2}) must hold
for some small
constant $M_{2}>0$.\medskip

By construction
of the roots (\ref{defin_roots}) in the factorization (\ref{factor_P_m}) and using the lower bound estimates
(\ref{root_cond_1}), (\ref{root_cond_2}), we get a constant $C_{P}>0$ such that
\begin{multline}
|P_{m}(\tau)| \geq M_{1}^{\delta_{D}\kappa-1}M_{2}|R_{D}(im)\kappa^{\delta_D}
|(\frac{|a_{0}Q(im)|}{|R_{D}(im)|\kappa^{\delta_{D}}})^{\frac{1}{\delta_{D}\kappa}}
(1+|\tau|)^{\delta_{D}\kappa-1}\\
\geq M_{1}^{\delta_{D}\kappa-1}M_{2}
\frac{\kappa^{\delta_D}|a_{0}|^{\frac{1}{\delta_{D}\kappa}}}{(\kappa^{\delta_{D}})^{\frac{1}{\delta_{D}\kappa}}}
(r_{Q,R_{D}})^{\frac{1}{\delta_{D}\kappa}} |R_{D}(im)| \\
\times (\min_{x \geq 0}
\frac{(1+x)^{\delta_{D}\kappa-1}}{(1+x^{\kappa})^{\delta_{D} - \frac{1}{\kappa}}})
(1 + |\tau|^{\kappa})^{\delta_{D} - \frac{1}{\kappa}}\\
= C_{P} (r_{Q,R_{D}})^{\frac{1}{\delta_{D}\kappa}} |R_{D}(im)|
(1+|\tau|^{\kappa})^{\delta_{D} - \frac{1}{\kappa}} \label{low_bounds_P_m}
\end{multline}
for all $\tau \in S_{d} \cup \bar{D}(0,\rho)$, all $m \in \mathbb{R}$.\medskip

In a first step, we show that $\Psi_{\kappa}(\tau,m,\epsilon)$ belongs to
$F_{(\nu,\beta,\mu,\chi,\kappa,\epsilon)}^{d}$, for any sector $S_{d}$, any fixed disc $D(0,\rho)$, for
$\beta>0,\mu > 1$ set above in
(\ref{defin_b_j}), for $\kappa \geq 1$ given in (\ref{cond_deltal_alpha}), for some $\nu>0$ depending
on $\kappa,K_{F}$ and $\theta_{F}$ prescribed in (\ref{defin_F_tzepsilon}) and (\ref{defin_omega_F}), provided that
\begin{equation}
\gamma + \alpha \leq \chi \ \ , \ \ \chi \kappa > \frac{1}{2} \label{cond_gamma_alpha_chi}
\end{equation}
hold. Notice that the second constraint of (\ref{cond_gamma_alpha_chi}) will only be needed later on
in Definition 4. Indeed, since the halfline $L_{\theta_{F}}$ avoids the roots
of $F_{2}(\tau)$, and from the fact that $\mathrm{deg}(F_{1}) \leq \mathrm{deg}(F_{2})$, we get a
constant $C_{F_{1},F_{2}}>0$ such that
\begin{equation}
\left| \frac{F_{1}(u)}{F_{2}(u)} \right| \leq C_{F_{1},F_{2}} 
\end{equation}
for all $u \in L_{\theta_{F}}$. We take a positive real number $\delta_{1}>0$ such that
$\cos(\theta_{F})> \delta_{1}$ and deduce the estimates
\begin{multline}
\left| \int_{L_{\theta_{F}}} e^{-K_{F}u} \frac{F_{1}(u)}{F_{2}(u)} \frac{u^{n}}{n!} du \right|
\leq C_{F_{1},F_{2}} \int_{0}^{+\infty} \exp( -K_{F}r\cos(\theta_{F})) \frac{r^n}{n!} dr\\
\leq C_{F_{1},F_{2}} \int_{0}^{+\infty} \exp( -K_{F}\delta_{1}r) \frac{r^n}{n!} dr
= C_{F_{1},F_{2}} (\frac{1}{K_{F}\delta_{1}})^{n+1}
\end{multline}
by definition of $n! = \int_{0}^{+\infty} e^{-u}u^{n} du$, for $n \geq 1$. We deduce that
\begin{equation}
|\Psi_{\kappa}(\tau,m,\epsilon)| \leq C_{F_{1},F_{2}}|\epsilon|^{n_{F}} |C_{F}(m)| \label{maj_Psi_kappa_C_E}
E(\tau,\epsilon) 
\end{equation}
where
$$ E(\tau,\epsilon) = \sum_{n \geq 1} (\frac{1}{K_{F}\delta_{1}})^{n+1}
\frac{|\frac{\tau}{\epsilon^{\gamma + \alpha}}|^{n}}{\Gamma(n/\kappa)}
= |\frac{\tau}{\epsilon^{\gamma+\alpha}}| \sum_{n \geq 0} (\frac{1}{K_{F}\delta_{1}})^{n+2}
\frac{|\frac{\tau}{\epsilon^{\gamma + \alpha}}|^{n}}{\Gamma(\frac{n+1}{\kappa})}$$
We now provide estimates for the function $E(\tau,\epsilon)$.\\
We first recall that $\Gamma(a+x) \sim x^{a}\Gamma(x)$ as $x \rightarrow +\infty$, for any real
number $a \in \mathbb{R}$, see \cite{ba2}, Appendix B.3. We deduce that
\begin{equation}
\Gamma(\frac{n+1}{\kappa}) \sim (\frac{n}{\kappa} + 1)^{\frac{1}{\kappa} - 1}\Gamma( \frac{n}{\kappa} + 1)
\label{Gamma_n_kappa_estimates}
\end{equation}
as $n \rightarrow +\infty$. Furthermore, we take $b>1$ and a constant $B_{\kappa}>0$ with
\begin{equation}
(\frac{n}{\kappa} + 1)^{1 - \frac{1}{\kappa}} \leq B_{\kappa} b^{n} \label{power_n_kappa_estimates} 
\end{equation}
for all $n \geq 0$. Gathering (\ref{Gamma_n_kappa_estimates}) and (\ref{power_n_kappa_estimates}),
we extract a constant $C_{\kappa}>0$ such that
\begin{equation}
E(\tau,\epsilon) \leq B_{\kappa}C_{\kappa}(\frac{1}{K_{F}\delta_{1}})^{2}
|\frac{\tau}{\epsilon^{\gamma+\alpha}}| \sum_{n \geq 0}
\frac{1}{\Gamma( \frac{n}{\kappa} + 1)} (\frac{b|\tau|}{K_{F}\delta_{1}|\epsilon|^{\gamma + \alpha}})^{n}
\end{equation}
for all $\tau \in \mathbb{C}$, all $\epsilon \in \mathbb{C}^{\ast}$. At this point, we remind that the Mittag-Leffler's functions
$E_{\beta}(x) = \sum_{n \geq 0} x^{n}/\Gamma(1 + \beta n)$ with index $\beta>0$ satisfies the next estimates : there
exists a constant $E_{\beta}>0$ with
$$ E_{\beta}(x) \leq E_{\beta}e^{x^{1/\beta}} $$
for all $x \geq 0$, see \cite{ba2}, Appendix B.4. We deduce the next bounds for $E(\tau,\epsilon)$,
\begin{equation}
E(\tau,\epsilon) \leq B_{\kappa}C_{\kappa}(\frac{1}{K_{F}\delta_{1}})^{2}E_{1/\kappa}
|\frac{\tau}{\epsilon^{\gamma+\alpha}}|
\exp( (\frac{b}{K_{F}\delta_{1}})^{\kappa}( \frac{|\tau|}{|\epsilon|^{\gamma + \alpha}} )^{\kappa} )
\end{equation}
for some constant $E_{1/\kappa}>0$, for all $\tau \in \mathbb{C}$, all $\epsilon \in \mathbb{C}^{\ast}$. Using
(\ref{cond_gamma_alpha_chi}), we get that
\begin{multline}
E(\tau,\epsilon) \leq B_{\kappa}C_{\kappa}(\frac{1}{K_{F}\delta_{1}})^{2}E_{1/\kappa}
\frac{|\tau|}{|\epsilon|^{\chi}}
\exp( 2(\frac{b}{K_{F}\delta_{1}})^{\kappa}( \frac{|\tau|}{|\epsilon|^{\chi}} )^{\kappa} )\\
\times
\frac{1}{1 + |\frac{\tau}{\epsilon^{\chi}}|^{2\kappa}} (1 + |\frac{\tau}{\epsilon^{\chi}}|^{2\kappa})
\exp( -(\frac{b}{K_{F}\delta_{1}})^{\kappa}( \frac{|\tau|}{|\epsilon|^{\chi}} )^{\kappa} ) \\
\leq B_{\kappa}C_{\kappa}(\frac{1}{K_{F}\delta_{1}})^{2}E_{1/\kappa}G_{K_{F},\delta_{1},\kappa}
\frac{|\tau|}{|\epsilon|^{\chi}}\frac{1}{1 + |\frac{\tau}{\epsilon^{\chi}}|^{2\kappa}}
\exp( 2(\frac{b}{K_{F}\delta_{1}})^{\kappa}( \frac{|\tau|}{|\epsilon|^{\chi}} )^{\kappa} ) \label{maj_E_tau_epsilon_exp}
\end{multline}
for all $\tau \in \mathbb{C}$ and $\epsilon \in \mathbb{C}^{\ast}$ with $|\epsilon|<1$, where
$$ G_{K_{F},\delta_{1},\kappa} = \sup_{x \geq 0} (1 + x^{2\kappa})\exp(
-(\frac{b}{K_{F}\delta_{1}})^{\kappa} x^{\kappa} ). $$
Finally, collecting (\ref{maj_Psi_kappa_C_E}) and (\ref{maj_E_tau_epsilon_exp}) yields the fact that
$\Psi_{\kappa}(\tau,m,\epsilon)$ belongs to $F_{(\nu,\beta,\mu,\chi,\kappa,\epsilon)}^{d}$ for all
the parameters specified as above.\medskip

In the next proposition, we disclose sufficient conditions for which the main convolution equation
(\ref{main_conv_eq_omega_kappa}) gets a unique solution in the Banach space
$F_{(\nu,\beta,\mu,\chi,\kappa,\epsilon)}^{d}$ described in Section 3.2, for the parameters chosen as above.

\begin{prop} We take for granted that the next additional assumptions hold,
\begin{equation}
\chi k_{l} + m_{l} - m_{0} - \alpha k_{l} \geq 0 \ \ , \ \ \delta_{D} \geq 1/\kappa \label{constraint_kl_ml} 
\end{equation}
for all $1 \leq l \leq q$,
\begin{equation}
\chi b_{j} + n_{j} - \alpha b_{j} \geq 0 \ \ , \ \ b_{j} \geq 1 \label{constraint_chi_bj_ni_alpha}
\end{equation}
for all $0 \leq j \leq Q$,
\begin{equation}
\chi \kappa ( \frac{d_{l,\kappa}}{\kappa} + \delta_{l} ) + \Delta_{l} + \alpha(\delta_{l} - d_{l}) - m_{0}
-\chi \kappa (\delta_{D} - \frac{1}{\kappa}) \geq 0 \ \ , \ \ \delta_{D} - \frac{1}{\kappa} \geq \delta_{l}
\label{constraint_chi_dlkappa_deltal_Deltal_alpha_m0}
\end{equation}
for $1 \leq l \leq D-1$ and
\begin{equation}
\chi \kappa ( \frac{h_l}{\kappa} + \frac{1}{\kappa}) + \mu_{l} - 2m_{0} - \alpha h_{l}
- \chi \kappa (\delta_{D} - \frac{1}{\kappa}) - \chi \geq 0 \label{constraint_chi_kappa_hl_mul_m0_alpha_deltaD}
\end{equation}
for all $0 \leq l \leq M$. Then, there exists a radius $r_{Q,R_{D}}>0$, $\epsilon_{0}>0$ and a constant $\varpi>0$ such that the equation
(\ref{main_conv_eq_omega_kappa}) has a unique solution $\omega_{\kappa}^{d}(\tau,m,\epsilon)$ in the Banach space
$F_{(\nu,\beta,\mu,\chi,\kappa,\epsilon)}^{d}$ which is subjected to the bounds
$$ ||\omega_{\kappa}^{d}(\tau,m,\epsilon)||_{(\nu,\beta,\mu,\chi,\kappa,\epsilon)} \leq \varpi $$
for all $\epsilon \in D(0,\epsilon_{0}) \setminus \{ 0 \}$, for any directions $d \in \mathbb{R}$ chosen in such
a manner that the sector $S_{d}$ respects the constraints (\ref{root_cond_1}) and (\ref{root_cond_2}) listed above.
\end{prop}

\begin{proof} We enter the proof with a lemma that focuses on some shrinking map upon the Banach spaces
mentioned above and scales down the main convolution problem to the construction of a fixed point for this map. 

\begin{lemma} Under the approval of the constraints (\ref{constraint_kl_ml}), (\ref{constraint_chi_bj_ni_alpha}), 
(\ref{constraint_chi_dlkappa_deltal_Deltal_alpha_m0}), (\ref{constraint_chi_kappa_hl_mul_m0_alpha_deltaD})
above, one can sort the constant $r_{Q,R_{D}}>0$
large enough and a constant $\varpi>0$ small enough such that for all
$\epsilon \in D(0,\epsilon_{0}) \setminus \{ 0 \}$, the map $\mathcal{H}_{\epsilon}$ defined as
\begin{multline}
\mathcal{H}_{\epsilon}(w(\tau,m)) :=  -\sum_{l=1}^{q}a_{l}\epsilon^{m_{l} - m_{0} - \alpha k_{l}}Q(im)
\frac{\tau^{\kappa}}{P_{m}(\tau)\Gamma(\frac{k_l}{\kappa})}
\int_{0}^{\tau^{\kappa}} (\tau^{\kappa}-s)^{\frac{k_l}{\kappa}-1}w(s^{1/\kappa},m) \frac{ds}{s}\\
-\sum_{l=0}^{M} c_{l}\epsilon^{\mu_{l} - 2m_{0} - \alpha h_{l}}
\frac{\tau^{\kappa}}{P_{m}(\tau)\Gamma(\frac{h_l}{\kappa})} \int_{0}^{\tau^{\kappa}}
(\tau^{\kappa} - s)^{\frac{h_l}{\kappa} - 1}\\
\times \left( s \int_{0}^{s} \frac{1}{(2\pi)^{1/2}}
\int_{-\infty}^{+\infty} Q_{1}(i(m-m_{1}))w((s-s')^{1/\kappa},m-m_{1}) \right. \\
\left. \times Q_{2}(im_{1})w((s')^{1/\kappa},m_{1}) dm_{1}\frac{1}{(s-s')s'} ds' \right)
\frac{ds}{s} + \sum_{j=0}^{Q} B_{j}(m) \epsilon^{n_{j} - \alpha b_{j}}
\frac{\tau^{b_j}}{P_{m}(\tau)\Gamma(\frac{b_j}{\kappa})} + \frac{\Psi_{\kappa}(\tau,m,\epsilon)}{P_{m}(\tau)}\\
+ R_{D}(im)\sum_{1 \leq p \leq \delta_{D}-1} A_{\delta_{D},p}
\frac{\tau^{\kappa}}{P_{m}(\tau)\Gamma( \delta_{D}-p )}
\int_{0}^{\tau^{\kappa}} (\tau^{\kappa} - s)^{\delta_{D}-p-1} (\kappa s)^{p}
w(s^{1/\kappa},m) \frac{ds}{s}\\
+ \sum_{l=1}^{D-1} \epsilon^{\Delta_{l} + \alpha(\delta_{l} - d_{l}) - m_{0}} R_{l}(im)\\
\times
\left( \frac{\tau^{\kappa}}{P_{m}(\tau)\Gamma( \frac{d_{l,\kappa}}{\kappa} )}
\int_{0}^{\tau^{\kappa}}(\tau^{\kappa} - s)^{\frac{d_{l,\kappa}}{\kappa} - 1}
(\kappa s)^{\delta_{l}} w(s^{1/\kappa},m) \frac{ds}{s} +
\sum_{1 \leq p \leq \delta_{l}-1} A_{\delta_{l},p} \right. \\
\left. \times 
\frac{\tau^{\kappa}}{P_{m}(\tau)\Gamma( \frac{d_{l,\kappa} + \kappa(\delta_{l}-p)}{\kappa} )}
\int_{0}^{\tau^{\kappa}} (\tau^{\kappa} - s)^{\frac{d_{l,\kappa} + \kappa(\delta_{l}-p)}{\kappa} - 1}
(\kappa s)^{p} w(s^{1/\kappa},m) \frac{ds}{s} \right)
\end{multline}
suffers the next properties.\\
{\bf i)} The following inclusion
\begin{equation}
\mathcal{H}_{\epsilon}(\bar{B}(0,\varpi)) \subset \bar{B}(0,\varpi) \label{H_ball_inside_ball}
\end{equation}
holds, where $\bar{B}(0,\varpi)$ is the closed ball centered at 0 with radius $\varpi$ in
$F_{(\nu,\beta,\mu,\chi,\kappa,\epsilon)}^{d}$, for all $\epsilon \in D(0,\epsilon_{0}) \setminus \{ 0 \}$.\\
{\bf ii)} We observe that
\begin{equation}
||\mathcal{H}_{\epsilon}(w_{1}) - \mathcal{H}_{\epsilon}(w_{2})||_{(\nu,\beta,\mu,\chi,\kappa,\epsilon)}
\leq \frac{1}{2}||w_{1} - w_{2}||_{(\nu,\beta,\mu,\chi,\kappa,\epsilon)} \label{H_shrink}
\end{equation}
for all $w_{1},w_{2} \in \bar{B}(0,\varpi)$, for all $\epsilon \in D(0,\epsilon_{0}) \setminus \{ 0 \}$.
\end{lemma}
\begin{proof} Firstly, we check the property (\ref{H_ball_inside_ball}). Let
$\epsilon \in D(0,\epsilon_{0}) \setminus \{0\}$ and consider
$w(\tau,m) \in F_{(\nu,\beta,\mu,\chi,\kappa,\epsilon)}^{d}$. We select $\varpi>0$ with
$||w(\tau,m)||_{(\nu,\beta,\mu,\chi,\kappa,\epsilon)} \leq \varpi$.\\
By taking a glance at Proposition 4 1), under (\ref{constraint_kl_ml}), we get a constant
$C_{2}>0$ (depending on $\nu,\kappa,Q,R_{D}$ and $k_{l}$ for $1 \leq l \leq q$) such that
\begin{multline}
|| \epsilon^{m_{l} - m_{0} - \alpha k_{l}}Q(im)
\frac{\tau^{\kappa}}{P_{m}(\tau)}
\int_{0}^{\tau^{\kappa}} (\tau^{\kappa}-s)^{\frac{k_l}{\kappa}-1}w(s^{1/\kappa},m) \frac{ds}{s}
||_{(\nu,\beta,\mu,\chi,\kappa,\epsilon)}\\
\leq \frac{C_{2}}{C_{P}(r_{Q,R_{D}})^{\frac{1}{\delta_{D}\kappa}}}
|\epsilon|^{\chi k_{l} + m_{l} - m_{0} - \alpha k_{l}}
||w(\tau,m)||_{(\nu,\beta,\mu,\chi,\kappa,\epsilon)}. \label{1_formula_H_epsilon_ball_in_ball}
\end{multline}
Due to Lemma 2, under (\ref{constraint_chi_bj_ni_alpha}), we get a constant $C_{1}>0$ (depending on 
$\nu,\kappa,\alpha$ and $n_{j},b_{j}$ for $0 \leq j \leq Q$) with
\begin{equation}
|| B_{j}(m) \epsilon^{n_{j} - \alpha b_{j}}
\frac{\tau^{b_j}}{P_{m}(\tau)} ||_{(\nu,\beta,\mu,\chi,\kappa,\epsilon)} \leq
\frac{C_1}{C_{P}(r_{Q,R_{D}})^{\frac{1}{\delta_{D}\kappa}}}
\frac{||B_{j}(m)||_{(\beta,\mu)}}{\mathrm{inf}_{m \in \mathbb{R}} |R_{D}(im)|}
|\epsilon|^{\chi b_{j} + n_{j} - \alpha b_{j}}. \label{2_formula_H_epsilon_ball_in_ball}
\end{equation}
From the estimates (\ref{maj_Psi_kappa_C_E}) and (\ref{maj_E_tau_epsilon_exp}), we get a constant
$C_{\Psi_{\kappa}}$ (depending on $||C_{F}(m)||_{(\beta,\mu)}$,$\kappa,K_{F},\\
\theta_{F},F_{1},F_{2}$) such that
\begin{equation}
|| \frac{\Psi_{\kappa}(\tau,m,\epsilon)}{P_{m}(\tau)} ||_{(\nu,\beta,\mu,\chi,\kappa,\epsilon)} 
\leq \frac{C_{\Psi_{\kappa}}}{C_{P}(r_{Q,R_{D}})^{\frac{1}{\delta_{D}\kappa}}
\mathrm{inf}_{m \in \mathbb{R}} |R_{D}(im)|} |\epsilon|^{n_{F}}. \label{3_formula_H_epsilon_ball_in_ball} 
\end{equation}
Using Proposition 4 2), we can take a constant $C_{2}'>0$ (depending on $\nu,\kappa$ and $\delta_{D}$) with
\begin{multline}
|| \frac{R_{D}(im)}{P_{m}(\tau)}\tau^{\kappa}
\int_{0}^{\tau^{\kappa}} (\tau^{\kappa} - s)^{\delta_{D}-p-1} s^{p-1}
w(s^{1/\kappa},m) ds ||_{(\nu,\beta,\mu,\chi,\kappa,\epsilon)} \\
\leq
\frac{C_{2}'}{C_{P}(r_{Q,R_{D}})^{\frac{1}{\delta_{D}\kappa}}} |\epsilon|^{\chi}
||w(\tau,m)||_{(\nu,\beta,\mu,\chi,\kappa,\epsilon)} \label{4_formula_H_epsilon_ball_in_ball}
\end{multline}
for $1 \leq p \leq \delta_{D}-1$ and in view of (\ref{constraint_chi_dlkappa_deltal_Deltal_alpha_m0}), we deduce
similarly a constant
$C_{2}'>0$ (depending on $\nu,\kappa,\delta_{D},R_{D}$ and $d_{l},\delta_{l},R_{l}$ for $1 \leq l \leq D-1$)
such that
\begin{multline}
||\epsilon^{\Delta_{l} + \alpha(\delta_{l} - d_{l}) - m_{0}} R_{l}(im)
\frac{\tau^{\kappa}}{P_{m}(\tau)}
\int_{0}^{\tau^{\kappa}}(\tau^{\kappa} - s)^{\frac{d_{l,\kappa}}{\kappa} - 1}
s^{\delta_{l}-1} w(s^{1/\kappa},m) ds||_{(\nu,\beta,\mu,\chi,\kappa,\epsilon)}\\
\leq \frac{C_{2}'}{C_{P}(r_{Q,R_{D}})^{\frac{1}{\delta_{D}\kappa}}}
|\epsilon|^{\chi \kappa( \frac{d_{l,\kappa}}{\kappa} + \delta_{l}) + \Delta_{l} +
\alpha(\delta_{l} - d_{l}) - m_{0} - \chi \kappa(\delta_{D} - \frac{1}{\kappa})}
||w(\tau,m)||_{(\nu,\beta,\mu,\chi,\kappa,\epsilon)} \label{5_formula_H_epsilon_ball_in_ball}
\end{multline}
and
\begin{multline}
||\epsilon^{\Delta_{l} + \alpha(\delta_{l} - d_{l}) - m_{0}} R_{l}(im) \frac{\tau^{\kappa}}{P_{m}(\tau)}
\int_{0}^{\tau^{\kappa}} (\tau^{\kappa} - s)^{\frac{d_{l,\kappa} + \kappa(\delta_{l}-p)}{\kappa} - 1}
s^{p-1} w(s^{1/\kappa},m) ds||_{(\nu,\beta,\mu,\chi,\kappa,\epsilon)}\\
\leq \frac{C_{2}'}{C_{P}(r_{Q,R_{D}})^{\frac{1}{\delta_{D}\kappa}}}
|\epsilon|^{\chi \kappa( \frac{d_{l,\kappa}}{\kappa} + \delta_{l}) + \Delta_{l} +
\alpha(\delta_{l} - d_{l}) - m_{0} - \chi \kappa(\delta_{D} - \frac{1}{\kappa})}
||w(\tau,m)||_{(\nu,\beta,\mu,\chi,\kappa,\epsilon)} \label{6_formula_H_epsilon_ball_in_ball}
\end{multline}
for all $1 \leq p \leq \delta_{l}-1$. Hereafter, we focus on estimates for the nonlinear part of
$\mathcal{H}_{\epsilon}$. Namely, we put
\begin{multline*}
h(\tau,m) = \tau^{\kappa-1}\frac{1}{R_{D}(im)}\int_{0}^{\tau^{\kappa}}
\int_{-\infty}^{+\infty}Q_{1}(i(m-m_{1}))
w((\tau^{\kappa} - s')^{1/\kappa},m-m_{1})\\
\times Q_{2}(im_{1})w((s')^{1/\kappa},m_{1})
\frac{1}{(\tau^{\kappa} - s')s'} ds'dm_{1}.
\end{multline*}
A glimpse into Proposition 5, allows us to catch a constant $C_{3}>0$ (depending on $\mu,\kappa,Q_{1},Q_{2},R_{D}$)
with
\begin{equation}
||h(\tau,m)||_{(\nu,\beta,\mu,\chi,\kappa,\epsilon)} \leq \frac{C_{3}}{|\epsilon|^{\chi}}
||w(\tau,m)||_{(\nu,\beta,\mu,\chi,\kappa,\epsilon)}^{2} \label{norm_h}
\end{equation}
On the other hand, bearing in mind Proposition 4 2), it boils down from
(\ref{constraint_chi_kappa_hl_mul_m0_alpha_deltaD}) that there exists a constant $C_{2}'>0$ (depending on
$\nu,\kappa,\delta_{D}$ and $h_{l}$ for $0 \leq l \leq M$)
with
\begin{multline}
||\epsilon^{\mu_{l} - 2m_{0} - \alpha h_{l}}
\frac{\tau^{\kappa}}{P_{m}(\tau)} R_{D}(im) \int_{0}^{\tau^{\kappa}}
(\tau^{\kappa} - s)^{\frac{h_l}{\kappa}-1} s^{\frac{1}{\kappa} - 1} h(s^{1/\kappa},m) ds
||_{(\nu,\beta,\mu,\chi,\kappa,\epsilon)}\\
\leq \frac{C_{2}'}{C_{P}(r_{Q,R_{D}})^{\frac{1}{\delta_{D}\kappa}}}
|\epsilon|^{\chi \kappa ( \frac{h_l}{\kappa} + \frac{1}{\kappa}) + \mu_{l} - 2m_{0} - \alpha h_{l}
- \chi \kappa (\delta_{D} - \frac{1}{\kappa})}
||h(\tau,m)||_{(\nu,\beta,\mu,\chi,\kappa,\epsilon)} \label{norm_conv_s_h_tau_m}
\end{multline}
Clustering (\ref{norm_h}) and (\ref{norm_conv_s_h_tau_m}) yields
\begin{multline}
|| \epsilon^{\mu_{l} - 2m_{0} - \alpha h_{l}}
\frac{\tau^{\kappa}}{P_{m}(\tau)} \int_{0}^{\tau^{\kappa}}
(\tau^{\kappa} - s)^{\frac{h_l}{\kappa} - 1}\\
\times \left( s \int_{0}^{s}
\int_{-\infty}^{+\infty} Q_{1}(i(m-m_{1}))w((s-s')^{1/\kappa},m-m_{1}) \right. \\
\left. \times Q_{2}(im_{1})w((s')^{1/\kappa},m_{1}) dm_{1} \frac{1}{(s-s')s'} ds' \right)
\frac{ds}{s}||_{(\nu,\beta,\mu,\chi,\kappa,\epsilon)} \\
\leq
\frac{C_{2}'C_{3}}{C_{P}(r_{Q,R_{D}})^{\frac{1}{\delta_{D}\kappa}}}
|\epsilon|^{\chi \kappa ( \frac{h_l}{\kappa} + \frac{1}{\kappa}) + \mu_{l} - 2m_{0} - \alpha h_{l}
- \chi \kappa (\delta_{D} - \frac{1}{\kappa}) - \chi}
||w(\tau,m)||_{(\nu,\beta,\mu,\chi,\kappa,\epsilon)}^{2} \label{7_formula_H_epsilon_ball_in_ball}
\end{multline}
Finally, we choose $r_{Q,R_{D}}>0$ and $\varpi>0$ in such a way that
\begin{multline}
\sum_{l=1}^{q} |a_{l}|\frac{C_{2}}{C_{P}(r_{Q,R_{D}})^{\frac{1}{\delta_{D}\kappa}}\Gamma(\frac{k_l}{\kappa})}
\epsilon_{0}^{\chi k_{l} + m_{l} - m_{0} - \alpha k_{l}}\varpi + \sum_{l=0}^{M} |c_{l}|
\frac{C_{2}'C_{3}}{C_{P}(r_{Q,R_{D}})^{\frac{1}{\delta_{D}\kappa}}\Gamma(\frac{h_l}{\kappa})(2\pi)^{1/2}}\\
\times
\epsilon_{0}^{\chi \kappa ( \frac{h_l}{\kappa} + \frac{1}{\kappa}) + \mu_{l} - 2m_{0} - \alpha h_{l}
- \chi \kappa (\delta_{D} - \frac{1}{\kappa}) - \chi}\varpi^{2}
+ \sum_{j=0}^{Q} \frac{C_1}{C_{P}(r_{Q,R_{D}})^{\frac{1}{\delta_{D}\kappa}}\Gamma(\frac{b_j}{\kappa})}\\
\times
\frac{||B_{j}(m)||_{(\beta,\mu)}}{\mathrm{inf}_{m \in \mathbb{R}} |R_{D}(im)|}
\epsilon_{0}^{\chi b_{j} + n_{j} - \alpha b_{j}} +
\frac{C_{\Psi_{\kappa}}}{C_{P}(r_{Q,R_{D}})^{\frac{1}{\delta_{D}\kappa}}
\mathrm{inf}_{m \in \mathbb{R}} |R_{D}(im)|} \epsilon_{0}^{n_{F}} +
\sum_{1 \leq p \leq \delta_{D}-1} |A_{\delta_{D},p}|\\
\times \frac{\kappa^{p}}{\Gamma(\delta_{D}-p)}
\frac{C_{2}'}{C_{P}(r_{Q,R_{D}})^{\frac{1}{\delta_{D}\kappa}}} \epsilon_{0}^{\chi} \varpi +
\sum_{l=1}^{D-1}
\frac{C_{2}' \kappa^{\delta_l} }{C_{P}(r_{Q,R_{D}})^{\frac{1}{\delta_{D}\kappa}}\Gamma(\frac{d_{l,\kappa}}{\kappa})}\\
\times
\epsilon_{0}^{\chi \kappa( \frac{d_{l,\kappa}}{\kappa} + \delta_{l}) + \Delta_{l} +
\alpha(\delta_{l} - d_{l}) - m_{0} - \chi \kappa(\delta_{D} - \frac{1}{\kappa})}\varpi
+ \sum_{1 \leq p \leq \delta_{l}-1} |A_{\delta_{l},p}|
\frac{\kappa^{p}}{\Gamma( \frac{d_{l,\kappa} + \kappa(\delta_{l}-p)}{\kappa} )}
\frac{C_{2}'}{C_{P}(r_{Q,R_{D}})^{\frac{1}{\delta_{D}\kappa}}}\\
\times
\epsilon_{0}^{\chi \kappa( \frac{d_{l,\kappa}}{\kappa} + \delta_{l}) + \Delta_{l} +
\alpha(\delta_{l} - d_{l}) - m_{0} - \chi \kappa(\delta_{D} - \frac{1}{\kappa})} \varpi \leq \varpi
\label{constraints_H_ball_in_ball}
\end{multline}
As an issue of the definition of $\mathcal{H}_{\epsilon}$, by collecting all the bounds
(\ref{1_formula_H_epsilon_ball_in_ball}), (\ref{2_formula_H_epsilon_ball_in_ball}),
(\ref{3_formula_H_epsilon_ball_in_ball}), (\ref{4_formula_H_epsilon_ball_in_ball}),
(\ref{5_formula_H_epsilon_ball_in_ball}), (\ref{6_formula_H_epsilon_ball_in_ball}),
(\ref{7_formula_H_epsilon_ball_in_ball}), we conclude that
\begin{equation}
||\mathcal{H}_{\epsilon}(w(\tau,m))||_{(\nu,\beta,\mu,\chi,\kappa,\epsilon)} \leq \varpi
\end{equation}
and the first claim (\ref{H_ball_inside_ball}) holds.\medskip

In a second part of the proof, we turn our effort to the verification of the affirmation (\ref{H_shrink}). Let
$w_{1}(\tau,m),w_{2}(\tau,m)$ belong to $F_{(\nu,\beta,\mu,\chi,\kappa,\epsilon)}^{d}$ with
$$ ||w_{1}(\tau,m)||_{(\nu,\beta,\mu,\chi,\kappa,\epsilon)} \leq \varpi \ \ , \ \
||w_{2}(\tau,m)||_{(\nu,\beta,\mu,\chi,\kappa,\epsilon)} \leq \varpi $$
Foremost, according to the
estimates (\ref{1_formula_H_epsilon_ball_in_ball}), we obtain a constant
$C_{2}>0$ (depending on $\nu,\kappa,Q,R_{D}$ and $k_{l}$ for $1 \leq l \leq q$) such that
\begin{multline}
|| \epsilon^{m_{l} - m_{0} - \alpha k_{l}}Q(im)
\frac{\tau^{\kappa}}{P_{m}(\tau)}
\int_{0}^{\tau^{\kappa}} (\tau^{\kappa}-s)^{\frac{k_l}{\kappa}-1}
(w_{1}(s^{1/\kappa},m) - w_{2}(s^{1/\kappa},m))  \frac{ds}{s}
||_{(\nu,\beta,\mu,\chi,\kappa,\epsilon)}\\
\leq \frac{C_{2}}{C_{P}(r_{Q,R_{D}})^{\frac{1}{\delta_{D}\kappa}}}
|\epsilon|^{\chi k_{l} + m_{l} - m_{0} - \alpha k_{l}}
||w_{1}(\tau,m) - w_{2}(\tau,m)||_{(\nu,\beta,\mu,\chi,\kappa,\epsilon)}. \label{1_formula_H_epsilon_shrink}
\end{multline} 
Likewise, in agreement with (\ref{4_formula_H_epsilon_ball_in_ball}),
(\ref{5_formula_H_epsilon_ball_in_ball}), (\ref{6_formula_H_epsilon_ball_in_ball}) we can select a constant
$C_{2}'>0$ (depending on $\nu,\kappa,\delta_{D},R_{D}$ and $d_{l},\delta_{l},R_{l}$
for $1 \leq l \leq D-1$) fulfilling
\begin{multline}
|| \frac{R_{D}(im)}{P_{m}(\tau)}\tau^{\kappa}
\int_{0}^{\tau^{\kappa}} (\tau^{\kappa} - s)^{\delta_{D}-p-1} s^{p-1}
(w_{1}(s^{1/\kappa},m) - w_{2}(s^{1/\kappa},m) ds ||_{(\nu,\beta,\mu,\chi,\kappa,\epsilon)} \\
\leq
\frac{C_{2}'}{C_{P}(r_{Q,R_{D}})^{\frac{1}{\delta_{D}\kappa}}} |\epsilon|^{\chi}
||w_{1}(\tau,m) - w_{2}(\tau,m)||_{(\nu,\beta,\mu,\chi,\kappa,\epsilon)} \label{2_formula_H_epsilon_shrink}
\end{multline}
for $1 \leq p \leq \delta_{D}-1$, together with
\begin{multline}
||\epsilon^{\Delta_{l} + \alpha(\delta_{l} - d_{l}) - m_{0}} R_{l}(im)
\frac{\tau^{\kappa}}{P_{m}(\tau)}
\int_{0}^{\tau^{\kappa}}(\tau^{\kappa} - s)^{\frac{d_{l,\kappa}}{\kappa} - 1}
s^{\delta_{l}-1}\\
\times (w_{1}(s^{1/\kappa},m) - w_{2}(s^{1/\kappa},m)) ds||_{(\nu,\beta,\mu,\chi,\kappa,\epsilon)}\\
\leq \frac{C_{2}'}{C_{P}(r_{Q,R_{D}})^{\frac{1}{\delta_{D}\kappa}}}
|\epsilon|^{\chi \kappa( \frac{d_{l,\kappa}}{\kappa} + \delta_{l}) + \Delta_{l} +
\alpha(\delta_{l} - d_{l}) - m_{0} - \chi \kappa(\delta_{D} - \frac{1}{\kappa})}
||w_{1}(\tau,m) - w_{2}(\tau,m)||_{(\nu,\beta,\mu,\chi,\kappa,\epsilon)} \label{3_formula_H_epsilon_shrink}
\end{multline}
and
\begin{multline}
||\epsilon^{\Delta_{l} + \alpha(\delta_{l} - d_{l}) - m_{0}} R_{l}(im) \frac{\tau^{\kappa}}{P_{m}(\tau)}
\int_{0}^{\tau^{\kappa}} (\tau^{\kappa} - s)^{\frac{d_{l,\kappa} + \kappa(\delta_{l}-p)}{\kappa} - 1}
s^{p-1}\\
\times (w_{1}(s^{1/\kappa},m) - w_{2}(s^{1/\kappa},m)) ds||_{(\nu,\beta,\mu,\chi,\kappa,\epsilon)}\\
\leq \frac{C_{2}'}{C_{P}(r_{Q,R_{D}})^{\frac{1}{\delta_{D}\kappa}}}
|\epsilon|^{\chi \kappa( \frac{d_{l,\kappa}}{\kappa} + \delta_{l}) + \Delta_{l} +
\alpha(\delta_{l} - d_{l}) - m_{0} - \chi \kappa(\delta_{D} - \frac{1}{\kappa})}
||w_{1}(\tau,m)-w_{2}(\tau,m)||_{(\nu,\beta,\mu,\chi,\kappa,\epsilon)} \label{4_formula_H_epsilon_shrink}
\end{multline}
for all $1 \leq p \leq \delta_{l}-1$. We turn now to the nonlinear part of $\mathcal{H}_{\epsilon}$. As a
groundwork, let us rewrite
\begin{multline}
Q_{1}(i(m-m_{1}))w_{1}((\tau^{\kappa}-s')^{1/\kappa},m-m_{1})Q_{2}(im_{1})w_{1}((s')^{1/\kappa},m_{1})\\
- Q_{1}(i(m-m_{1}))w_{2}((\tau^{\kappa}-s')^{1/\kappa},m-m_{1})Q_{2}(im_{1})w_{2}((s')^{1/\kappa},m_{1})\\
= Q_{1}(i(m-m_{1}))
\left( w_{1}((\tau^{\kappa}-s')^{1/\kappa},m-m_{1}) - w_{2}((\tau^{\kappa}-s')^{1/\kappa},m-m_{1}) \right)\\
\times
Q_{2}(im_{1})w_{1}((s')^{1/\kappa},m_{1}) + Q_{1}(i(m-m_{1}))w_{2}((\tau^{\kappa}-s')^{1/\kappa},m-m_{1})
Q_{2}(im_{1})\\
\times \left( w_{1}((s')^{1/\kappa},m_{1}) - w_{2}((s')^{1/\kappa},m_{1}) \right)
\label{factor_Q1w1Q2w1_minus_Q1w2Q2w2}
\end{multline}
For $j=1,2$, we put
\begin{multline*}
h_{j}(\tau,m) = \frac{\tau^{\kappa-1}}{R_{D}(im)} \int_{0}^{\tau^{\kappa}}
\int_{-\infty}^{+\infty} Q_{1}(i(m-m_{1}))w_{j}((\tau^{\kappa} - s')^{1/\kappa},m-m_{1})\\
\times
Q_{2}(im_{1})w_{j}((s')^{1/\kappa},m_{1})\frac{1}{(\tau^{\kappa}-s')s'} ds' dm_{1}
\end{multline*}
Focusing both on the factorization (\ref{factor_Q1w1Q2w1_minus_Q1w2Q2w2}) and Proposition 5, we can find
a constant $C_{3}>0$ (depending on $\mu,\kappa,Q_{1},Q_{2},R_{D}$) with
\begin{multline}
||h_{1}(\tau,m) - h_{2}(\tau,m)||_{(\nu,\beta,\mu,\chi,\kappa,\epsilon)} \leq
\frac{C_3}{|\epsilon|^{\chi}}( ||w_{1}(\tau,m)||_{(\nu,\beta,\mu,\chi,\kappa,\epsilon)} +
||w_{2}(\tau,m)||_{(\nu,\beta,\mu,\chi,\kappa,\epsilon)} )\\
\times ||w_{1}(\tau,m) - w_{2}(\tau,m)||_{(\nu,\beta,\mu,\chi,\kappa,\epsilon)} \label{norm_h1_minus_h2}
\end{multline}
From (\ref{norm_conv_s_h_tau_m}) together with (\ref{norm_h1_minus_h2}), we can pick up a constant $C_{2}'>0$
(depending on $\nu,\kappa,\delta_{D}$ and $h_{l}$ for $0 \leq l \leq M$) with
\begin{multline}
||\epsilon^{\mu_{l} - 2m_{0} - \alpha h_{l}}
\frac{\tau^{\kappa}}{P_{m}(\tau)} R_{D}(im) \int_{0}^{\tau^{\kappa}}
(\tau^{\kappa} - s)^{\frac{h_l}{\kappa}-1} s^{\frac{1}{\kappa} - 1}
(h_{1}(s^{1/\kappa},m) - h_{2}(s^{1/\kappa},m)) ds
||_{(\nu,\beta,\mu,\chi,\kappa,\epsilon)}\\
\leq \frac{C_{2}'}{C_{P}(r_{Q,R_{D}})^{\frac{1}{\delta_{D}\kappa}}}
|\epsilon|^{\chi \kappa ( \frac{h_l}{\kappa} + \frac{1}{\kappa}) + \mu_{l} - 2m_{0} - \alpha h_{l}
- \chi \kappa (\delta_{D} - \frac{1}{\kappa})}
||h_{1}(\tau,m) - h_{2}(\tau,m)||_{(\nu,\beta,\mu,\chi,\kappa,\epsilon)}\\
\leq \frac{C_{2}'C_{3}}{C_{P}(r_{Q,R_{D}})^{\frac{1}{\delta_{D}\kappa}}}
|\epsilon|^{\chi \kappa ( \frac{h_l}{\kappa} + \frac{1}{\kappa}) + \mu_{l} - 2m_{0} - \alpha h_{l}
- \chi \kappa (\delta_{D} - \frac{1}{\kappa}) - \chi}\\
\times ( ||w_{1}(\tau,m)||_{(\nu,\beta,\mu,\chi,\kappa,\epsilon)} +
||w_{2}(\tau,m)||_{(\nu,\beta,\mu,\chi,\kappa,\epsilon)} )
||w_{1}(\tau,m) - w_{2}(\tau,m)||_{(\nu,\beta,\mu,\chi,\kappa,\epsilon)} \label{5_formula_H_epsilon_shrink}
\end{multline}
As a result, we choose $r_{Q,R_{D}}>0$ and $\varpi>0$ obeying the next inequality
\begin{multline}
\sum_{l=1}^{q} |a_{l}|\frac{C_{2}}{C_{P}(r_{Q,R_{D}})^{\frac{1}{\delta_{D}\kappa}}\Gamma(\frac{k_l}{\kappa})}
\epsilon_{0}^{\chi k_{l} + m_{l} - m_{0} - \alpha k_{l}} + \sum_{l=0}^{M} |c_{l}|
\frac{C_{2}'C_{3}}{C_{P}(r_{Q,R_{D}})^{\frac{1}{\delta_{D}\kappa}}\Gamma(\frac{h_l}{\kappa})(2\pi)^{1/2}}\\
\times
\epsilon_{0}^{\chi \kappa ( \frac{h_l}{\kappa} + \frac{1}{\kappa}) + \mu_{l} - 2m_{0} - \alpha h_{l}
- \chi \kappa (\delta_{D} - \frac{1}{\kappa}) - \chi} 2\varpi
+ \sum_{1 \leq p \leq \delta_{D}-1} |A_{\delta_{D},p}|\\
\times \frac{\kappa^{p}}{\Gamma(\delta_{D}-p)}
\frac{C_{2}'}{C_{P}(r_{Q,R_{D}})^{\frac{1}{\delta_{D}\kappa}}} \epsilon_{0}^{\chi} +
\sum_{l=1}^{D-1}
\frac{C_{2}' \kappa^{\delta_l} }{C_{P}(r_{Q,R_{D}})^{\frac{1}{\delta_{D}\kappa}}\Gamma(\frac{d_{l,\kappa}}{\kappa})}\\
\times
\epsilon_{0}^{\chi \kappa( \frac{d_{l,\kappa}}{\kappa} + \delta_{l}) + \Delta_{l} +
\alpha(\delta_{l} - d_{l}) - m_{0} - \chi \kappa(\delta_{D} - \frac{1}{\kappa})}
+ \sum_{1 \leq p \leq \delta_{l}-1} |A_{\delta_{l},p}|
\frac{\kappa^{p}}{\Gamma( \frac{d_{l,\kappa} + \kappa(\delta_{l}-p)}{\kappa} )}
\frac{C_{2}'}{C_{P}(r_{Q,R_{D}})^{\frac{1}{\delta_{D}\kappa}}}\\
\times
\epsilon_{0}^{\chi \kappa( \frac{d_{l,\kappa}}{\kappa} + \delta_{l}) + \Delta_{l} +
\alpha(\delta_{l} - d_{l}) - m_{0} - \chi \kappa(\delta_{D} - \frac{1}{\kappa})} \leq \frac{1}{2}
\label{constraints_H_shrink}
\end{multline}
By assembling all the bounds (\ref{1_formula_H_epsilon_shrink}),  (\ref{2_formula_H_epsilon_shrink}),
(\ref{3_formula_H_epsilon_shrink}), (\ref{4_formula_H_epsilon_shrink}), (\ref{5_formula_H_epsilon_shrink}), we
attain the foreseen estimates (\ref{H_shrink}).\medskip

At the very end of the proof, we now take for granted that the two conditions
(\ref{constraints_H_ball_in_ball}) and (\ref{constraints_H_shrink}) hold conjointly
for the radii $r_{Q,R_{D}}$ and $\varpi$. Then both (\ref{H_ball_inside_ball}) and
(\ref{H_shrink}) hold at the same time and the Lemma 3 is shown.
\end{proof}
We consider the ball $\bar{B}(0,\varpi) \subset F_{(\nu,\beta,\mu,\chi,\kappa,\epsilon)}^{d}$ just
built above in Lemma 3 which is actually a complete
metric space equipped with the norm $||.||_{(\nu,\beta,\mu,\chi,\kappa,\epsilon)}$. From the lemma above, we get
that $\mathcal{H}_{\epsilon}$ is a
contractive application from $\bar{B}(0,\varpi)$ into itself. Due to the classical contractive mapping theorem, we
deduce that
the map $\mathcal{H}_{\epsilon}$ has a unique fixed point denoted $\omega_{\kappa}^{d}(\tau,m,\epsilon)$
in the ball $\bar{B}(0,\varpi)$, meaning that
\begin{equation}
\mathcal{H}_{\epsilon}(\omega_{\kappa}^{d}(\tau,m,\epsilon)) = \omega_{\kappa}^{d}(\tau,m,\epsilon)
\label{H_epsilon_fix_point_eq}
\end{equation}
for a unique $\omega_{\kappa}^{d}(\tau,m,\epsilon) \in F_{(\nu,\beta,\mu,\chi,\kappa,\epsilon)}^{d}$ such
that  
$||\omega_{\kappa}^{d}(\tau,m,\epsilon)||_{(\nu,\beta,\mu,\chi,\kappa,\epsilon)} \leq \varpi$,
for all $\epsilon \in D(0,\epsilon_{0}) \setminus \{ 0 \}$. Moreover, the function
$\omega_{\kappa}^{d}(\tau,m,\epsilon)$ depends holomorphically on $\epsilon$ in
$D(0,\epsilon_{0}) \setminus \{ 0 \}$.

Now, if one sets apart the terms $Q(im)a_{0}\omega_{\kappa}(\tau,m,\epsilon)$ in the left handside of
(\ref{main_conv_eq_omega_kappa}) and \\
$R_{D}(im)(\kappa \tau^{\kappa})^{\delta_D}\omega_{\kappa}(\tau,m,\epsilon)$
in the right handside of (\ref{main_conv_eq_omega_kappa}), we recognize by dividing with the 
polynomial $P_{m}(\tau)$ given in (\ref{factor_P_m}) that (\ref{main_conv_eq_omega_kappa}) can be exactly
rewritten as the equation (\ref{H_epsilon_fix_point_eq}) above. Therefore, the unique fixed point
$\omega_{\kappa}^{d}(\tau,m,\epsilon)$ of $\mathcal{H}_{\epsilon}$ in $\bar{B}(0,\varpi)$ precisely solves the
problem (\ref{main_conv_eq_omega_kappa}). This yields the proposition.
\end{proof}

\subsection{Analytic solutions to the main problem on boundary layers $\epsilon-$depending domains in time near
the origin}

We return to the formal construction of time rescaled solutions to the main equation (\ref{main_PDE_u}) under
the new insight on the main associated convolution equation (\ref{main_conv_eq_omega_kappa})
reached in the previous subsection.\medskip

\noindent We first recall the definitions of a good covering as introduced in \cite{lama1}.

\begin{defin} Let $\varsigma \geq 2$ be an integer. For all $0 \leq p \leq \varsigma-1$, we consider open sectors
$\mathcal{E}_{p}$ centered at $0$, with radius $\epsilon_{0}>0$ and opening
$\frac{\pi}{\chi \kappa} + \xi_{p} < 2\pi$ with $\xi_{p}>0$ small enough such that
$\mathcal{E}_{p} \cap \mathcal{E}_{p+1} \neq \emptyset$, for all
$0 \leq p \leq \varsigma-1$ (with the convention that $\mathcal{E}_{\varsigma} = \mathcal{E}_{0})$. Moreover, we
assume that
the intersection of any three different elements in $\{ \mathcal{E}_{p} \}_{0 \leq p \leq \varsigma-1}$ is empty
and that
$\cup_{p=0}^{\varsigma - 1} \mathcal{E}_{p} = \mathcal{U} \setminus \{ 0 \}$,
where $\mathcal{U}$ is some neighborhood of 0 in $\mathbb{C}$. Such a set of sectors
$\{ \mathcal{E}_{p} \}_{0 \leq p \leq \varsigma - 1}$ is called a good covering in $\mathbb{C}^{\ast}$ with aperture
$\frac{\pi}{\chi \kappa}$.
\end{defin}

\noindent We now give a definition for a set of $\epsilon-$depending sectors associated to a good covering.

\begin{defin} Let $\{ \mathcal{E}_{p} \}_{0 \leq p \leq \varsigma - 1}$ be a good covering with aperture
$\frac{\pi}{\chi \kappa}$. Let $\alpha \in \mathbb{Q}$ with $\alpha < \chi$. We choose a fixed open sector
$X$ centered at 0 with radius $\varrho_{X}>0$ and for each
$\epsilon \in D(0,\epsilon_{0}) \setminus \{ 0 \}$, we define the sector
$$ \mathcal{T}_{\epsilon,\chi - \alpha} =
\{ x\epsilon^{\chi - \alpha} / x \in X \} $$
with radius $\varrho_{X}|\epsilon|^{\chi - \alpha}$. For each $\epsilon \in
D(0,\epsilon_{0}) \setminus \{ 0 \}$, we consider also
a family of sectors
$$ S_{\mathfrak{d}_{p},\theta,\varrho_{X}|\epsilon|^{\chi}} =
\{ T \in \mathbb{C}^{\ast} / |T| \leq \varrho_{X} |\epsilon|^{\chi} \ \ , \ \
|\mathfrak{d}_{p} - \mathrm{arg}(T)| < \frac{\theta}{2} \} $$
for some aperture $\theta > \frac{\pi}{\kappa}$, where $\mathfrak{d}_{p} \in \mathbb{R}$, for
$0 \leq p \leq \varsigma - 1$
are directions which satisfy the next constraints described below.\\
Let $q_{l}(m)$ be the roots of $P_{m}(\tau)$ defined in (\ref{defin_roots}) and $S_{\mathfrak{d}_{p}}$,
$0 \leq p \leq \varsigma - 1$ be unbounded sectors centered at 0 with direction $\mathfrak{d}_{p}$ and small
aperture. We assume
that\\
1) There exists a constant $M_{1}>0$ such that
\begin{equation}
|\tau - q_{l}(m)| \geq M_{1}(1 + |\tau|) \label{root_cond_1_in_defin}
\end{equation}
for all $0 \leq l \leq \delta_{D}\kappa-1$, all $m \in \mathbb{R}$, all
$\tau \in S_{\mathfrak{d}_p} \cup \bar{D}(0,\rho)$, for all
$0 \leq p \leq \varsigma-1$.\\
2) There exists a constant $M_{2}>0$ such that
\begin{equation}
|\tau - q_{l_0}(m)| \geq M_{2}|q_{l_0}(m)| \label{root_cond_2_in_defin}
\end{equation}
for some $l_{0} \in \{0,\ldots,\delta_{D}\kappa-1 \}$, all $m \in \mathbb{R}$, all
$\tau \in S_{\mathfrak{d}_p} \cup \bar{D}(0,\rho)$, for
all $0 \leq p \leq \varsigma - 1$.\\
3) For all $0 \leq p \leq \varsigma - 1$, all
$\epsilon \in \mathcal{E}_{p}$ and $t \in \mathcal{T}_{\epsilon,\chi - \alpha}$, we have that
$$ \epsilon^{\alpha}t \in S_{\mathfrak{d}_{p},\theta,\varrho_{X}|\epsilon|^{\chi}}.$$

\noindent We say that the family of sectors $\{
(S_{\mathfrak{d}_{p},\theta,\varrho_{X}|\epsilon|^{\chi}})_{0 \leq p \leq \varsigma - 1},
\mathcal{T}_{\epsilon,\chi - \alpha} \} $ is associated to the good covering
$\{ \mathcal{E}_{p} \}_{0 \leq p \leq \varsigma - 1}$.
\end{defin}

In the next main first outcome, we build a family of actual holomorphic solutions to the principal
equation (\ref{main_PDE_u}) that we call \emph{inner solutions}. These solutions
are defined on the sectors $\mathcal{E}_{p}$ of a good covering w.r.t $\epsilon$, on $\epsilon-$depending
associated sectors $\mathcal{T}_{\epsilon,\chi - \alpha}$ w.r.t $t$ and on some horizontal strip $H_{\beta}$
w.r.t $z$. Furthermore, we can oversee the difference between any two neighboring solutions on the
intersections $\mathcal{E}_{p} \cap \mathcal{E}_{p+1}$ and ascertain that it is exponentially flat of order
at most $\chi \kappa$ w.r.t $\epsilon$.

\begin{theo} We look at the singularly perturbed PDE (\ref{main_PDE_u}) and we assume that all the prior
constraints (\ref{constraints_degree_coeff_Q_Rl}), (\ref{constraint_m0_ml}), (\ref{defin_b_j}),
(\ref{defin_F_tzepsilon}), (\ref{defin_omega_F}), (\ref{cond_deltaD_alpha}), (\ref{cond_deltal_alpha}),
(\ref{quotient_Q_RD_in_S}), (\ref{cond_gamma_alpha_chi}), (\ref{constraint_kl_ml}),
(\ref{constraint_chi_bj_ni_alpha}), (\ref{constraint_chi_dlkappa_deltal_Deltal_alpha_m0}) and
(\ref{constraint_chi_kappa_hl_mul_m0_alpha_deltaD}) hold. Let $\{ \mathcal{E}_{p} \}_{0 \leq p \leq \varsigma-1}$
a good covering in $\mathbb{C}^{\ast}$ with aperture $\frac{\pi}{\chi \kappa}$ be given, for which a 
family of open sectors $\{ (S_{\mathfrak{d}_{p},\theta,\varrho_{X}|\epsilon|^{\chi}})_{0 \leq p \leq \varsigma-1},
\mathcal{T}_{\epsilon,\chi - \alpha} \}$ associated to this good covering can be distinguished.

Then, there exist a radius $r_{Q,R_{D}}>0$ large enough and $\epsilon_{0}>0$ small enough, for which
a family $\{ u^{\mathfrak{d}_{p}}(t,z,\epsilon) \}_{0 \leq p \leq \varsigma - 1}$ of actual solutions
of (\ref{main_PDE_u}) are built up. Furthermore, for each $\epsilon \in \mathcal{E}_{p}$, the
function $(t,z) \mapsto u^{\mathfrak{d}_{p}}(t,z,\epsilon)$ defines a bounded holomorphic function
on $\mathcal{T}_{\epsilon,\chi - \alpha} \times H_{\beta'}$ for any given $0 < \beta' < \beta$,
for all $0 \leq p \leq \varsigma - 1$. Moreover, the functions
$(x,z,\epsilon) \mapsto \epsilon^{m_{0}}u^{\mathfrak{d}_{p}}(x \epsilon^{\chi - \alpha},z,\epsilon)$
are bounded and holomorphic on $X \times H_{\beta'} \times \mathcal{E}_{p}$ for any
given $0 < \beta' < \beta$, $0 \leq p \leq \varsigma -1$ and are submitted to the next bounds : there exist
constants $K_{p},M_{p}>0$ and $\sigma>0$ (independent of $\epsilon$) such that
\begin{equation}
\sup_{x \in X \cap D(0,\sigma),z \in H_{\beta'}} | \epsilon^{m_0}u^{\mathfrak{d}_{p+1}}(x \epsilon^{\chi - \alpha},z,\epsilon) -
\epsilon^{m_0}u^{\mathfrak{d}_{p}}(x \epsilon^{\chi - \alpha},z,\epsilon) | \leq
K_{p}\exp( -\frac{M_p}{|\epsilon|^{\chi \kappa}} ) \label{difference_u_dp_exp_small_epsilon}
\end{equation}
for all $\epsilon \in \mathcal{E}_{p+1} \cap \mathcal{E}_{p}$, all $0 \leq p \leq \varsigma - 1$
(where by convention $u^{\mathfrak{d}_{\varsigma}} = u^{\mathfrak{d}_{0}}$).
\end{theo}
\begin{proof} As shown above in Section 3.4, the series
$$ \Psi_{\kappa}(\tau,m,\epsilon) = \sum_{n \geq 1} \psi_{n}(m,\epsilon)
\frac{\tau^{n}}{\Gamma(n/\kappa)} \in E_{(\beta,\mu)}[[ \tau ]] $$
is convergent for all $\tau$ in $\mathbb{C}$, for all $\epsilon \in D(0,\epsilon_{0}) \setminus \{ 0 \}$.
Moreover, it is subjected to the next bounds
$$ |\Psi_{\kappa}(\tau,m,\epsilon)| \leq \Psi |\epsilon|^{n_F} (1 + |m|)^{-\mu} \exp( -\beta |m|)
\frac{|\frac{\tau}{\epsilon^{\chi}}|}{1 + |\frac{\tau}{\epsilon^{\chi}}|^{2\kappa}}
\exp( \nu |\frac{\tau}{\epsilon^{\chi}}|^{\kappa} )$$
for some constant $\Psi>0$ (independent of $\epsilon$), for all $\tau \in \mathbb{C}$, all
$\epsilon \in D(0,\epsilon_{0}) \setminus \{ 0 \}$. We deduce that the formal series
$$ \Phi_{\kappa}(T,m,\epsilon) = \sum_{n \geq 1} \psi_{n}(m,\epsilon) T^{n} \in E_{(\beta,\mu)}[[T]] $$
is $m_{\kappa}-$summable in all directions $d \in \mathbb{R}$ according to Definition 2 and hence defines a convergent series near $T=0$.
In order to get its radius of convergence, we can express it as a $m_{\kappa}-$sum
$$ \Phi_{\kappa}(T,m,\epsilon) = \kappa \int_{L_{d}} \Psi_{\kappa}(u,m,\epsilon)
\exp( -(\frac{u}{T})^{\kappa} ) \frac{du}{u}
$$
for any halfline $L_{d} = \mathbb{R}_{+}e^{\sqrt{-1}d}$, with direction $d \in \mathbb{R}$. From Definition 2,
we can check that $T \mapsto \Phi_{\kappa}(T,m,\epsilon)$ defines a $E_{(\beta,\mu)}-$valued holomorphic
function on a disc $D(0,\Delta|\epsilon|^{\chi})$, for some $\Delta>0$ (independent of $\epsilon$), for all
$\epsilon \in D(0,\epsilon_{0}) \setminus \{ 0 \}$.

Now, we select a good covering $\{ \mathcal{E}_{p} \}_{0 \leq p \leq \varsigma - 1}$ in
$\mathbb{C}^{\ast}$ with aperture $\frac{\pi}{\chi \kappa}$ and a family of sectors
$\{ (S_{\mathfrak{d}_{p},\theta,\varrho_{X}|\epsilon|^{\chi}})_{0 \leq p \leq \varsigma-1},
\mathcal{T}_{\epsilon,\chi - \alpha} \}$ associated to this covering according to Definition 5. Proposition 6
allows us to choose for each direction $\mathfrak{d}_{p}$, a solution
$\omega_{\kappa}^{\mathfrak{d}_{p}}\tau,m,\epsilon)$ of the convolution equation (\ref{main_conv_eq_omega_kappa}) which
is located in the Banach space $F_{(\nu,\beta,\mu,\chi,\kappa,\epsilon)}^{\mathfrak{d}_{p}}$ and thus suffering the
next bounds
\begin{equation}
|\omega_{\kappa}^{\mathfrak{d}_{p}}(\tau,m,\epsilon)| \leq \varpi
(1 + |m|)^{-\mu} \exp( -\beta |m|)
\frac{|\frac{\tau}{\epsilon^{\chi}}|}{1 + |\frac{\tau}{\epsilon^{\chi}}|^{2\kappa}}
\exp( \nu |\frac{\tau}{\epsilon^{\chi}}|^{\kappa} ) \label{omega_kappa_dp_in_F}
\end{equation}
for all $\tau \in \bar{D}(0,\rho) \cup S_{\mathfrak{d}_{p}}$, $m \in \mathbb{R}$ and
$\epsilon \in D(0,\epsilon_{0}) \setminus \{ 0 \}$, for some suitable $\varpi > 0$. In particular, these
functions $\omega_{\kappa}^{\mathfrak{d}_{p}}(\tau,m,\epsilon)$ are analytic continuation w.r.t $\tau$ of a common
convergent series
$$ \omega_{\kappa}(\tau,m,\epsilon) = \sum_{n \geq 1} \frac{\omega_{n}(m,\epsilon)}{\Gamma(n/\kappa)} \tau^{n}
\in E_{(\beta,\mu)}\{\tau \} $$
which defines a solution of (\ref{main_conv_eq_omega_kappa}) for $\tau \in D(0,\rho)$. As a result,
the formal series
$$\Omega_{\kappa}(T,m,\epsilon) = \sum_{n \geq 1} \omega_{n}(m,\epsilon)T^{n} \in E_{(\beta,\mu)}[[T]]$$
turns out to be $m_{\kappa}-$summable in direction $\mathfrak{d}_{p}$ according to Definition 2, for all
$\epsilon \in D(0,\epsilon_{0}) \setminus \{ 0 \}$. We set
\begin{equation}
\Omega_{\kappa}^{\mathfrak{d}_{p}}(T,m,\epsilon) = \kappa \int_{L_{\gamma}}
\omega_{\kappa}^{\mathfrak{d}_{p}}(u,m,\epsilon) \exp( -(\frac{u}{T})^{\kappa} ) \frac{du}{u}
\end{equation}
as the $m_{\kappa}-$sum of $\Omega_{\kappa}(T,m,\epsilon)$ in direction $\mathfrak{d}_{p}$, with
$L_{\gamma} = \mathbb{R}_{+}e^{\sqrt{-1}\gamma} \subset S_{\mathfrak{d}_{p}}$. This map defines a
$E_{(\beta,\mu)}-$valued holomorphic function w.r.t $T$ on a sector
$$ S_{\mathfrak{d}_{p},\theta,\Delta|\epsilon|^{\chi}} =
\{ T \in \mathbb{C}^{\ast} : |T|<\Delta|\epsilon|^{\chi} \ \ , \ \ |\mathfrak{d}_{p} - \mathrm{arg}(T)|<\theta/2 \}
$$
for some $\theta \in (\frac{\pi}{\kappa},\frac{\pi}{\kappa} + \mathrm{Ap}(S_{\mathfrak{d}_{p}}))$
(where $\mathrm{Ap}(S_{\mathfrak{d}_{p}})$ stands for the aperture of the sector $S_{\mathfrak{d}_{p}}$)
and some $\Delta>0$ independent of $\epsilon$, for all $\epsilon \in D(0,\epsilon_{0}) \setminus \{ 0 \}$.

According to the formal identities enounced in Proposition 3 and by virtue of the properties of the
$m_{\kappa}-$sums with respect to derivatives and product (within the Banach space $E_{(\beta,\mu)}$ endowed
with the convolution product $\star$ described in Proposition 1) we notice that the function
$\Omega_{\kappa}^{\mathfrak{d}_{p}}(T,m,\epsilon)$ must solve the next equation
\begin{multline}
Q(im) ( \sum_{l=1}^{q} a_{l} \epsilon^{m_{l} - m_{0} - \alpha k_{l}} T^{k_l} + a_{0})
\Omega_{\kappa}^{\mathfrak{d}_{p}}(T,m,\epsilon) + (\sum_{l=0}^{M} c_{l}
\epsilon^{\mu_{l} - 2m_{0} - \alpha h_{l}} T^{h_l})\frac{1}{(2\pi)^{1/2}}\\
\times
\int_{-\infty}^{+\infty} Q_{1}(i(m-m_{1}))\Omega_{\kappa}^{\mathfrak{d}_{p}}(T,m-m_{1},\epsilon)
Q_{2}(im_{1})\Omega_{\kappa}^{\mathfrak{d}_{p}}(T,m_{1},\epsilon) dm_{1}=
\sum_{j=0}^{Q} B_{j}(m) \epsilon^{n_{j} - \alpha b_{j}} T^{b_j} \\
+ \Phi_{\kappa}(T,m,\epsilon)
+ R_{D}(im) \left( (T^{\kappa+1}\partial_{T})^{\delta_{D}} +
\sum_{1 \leq p \leq \delta_{D}-1} A_{\delta_{D},p} T^{\kappa(\delta_{D}-p)}(T^{\kappa+1}\partial_{T})^{p} \right)
\Omega_{\kappa}^{\mathfrak{d}_{p}}(T,m,\epsilon)\\
+ \sum_{l=1}^{D-1} \epsilon^{\Delta_{l} + \alpha(\delta_{l} - d_{l}) - m_{0}} R_{l}(im) T^{d_{l,\kappa}}
\left( (T^{\kappa+1}\partial_{T})^{\delta_l} \right. \\
\left. + \sum_{1 \leq p \leq \delta_{l}-1} A_{\delta_{l},p}
T^{\kappa(\delta_{l}-p)} (T^{\kappa+1}\partial_{T})^{p} \right) \Omega_{\kappa}^{\mathfrak{d}_{p}}(T,m,\epsilon)
\label{main_equation_Omega_kappa_dp}
\end{multline}
In the next step, we introduce
$$ U^{\mathfrak{d}_{p}}(T,z,\epsilon) =
\mathcal{F}^{-1}(m \mapsto \Omega_{\kappa}^{\mathfrak{d}_{p}}(T,m,\epsilon) )(z) $$
which defines a bounded holomorphic function w.r.t $T$ on $S_{\mathfrak{d}_{p},\theta,\Delta|\epsilon|^{\chi}}$,
w.r.t $z$ on $H_{\beta'}$ for any $0 < \beta' < \beta$ and all $\epsilon \in D(0,\epsilon_{0}) \setminus \{ 0 \}$.
Besides, by construction, we observe that the function $F^{\theta_{F}}$ defined in (\ref{defin_F_tzepsilon}) can be expressed as
$$ F^{\theta_{F}}(\epsilon^{-\alpha}T,z,\epsilon) = \mathcal{F}^{-1}(m \mapsto \Phi_{\kappa}(T,m,\epsilon) ) $$
which represents a bounded holomorphic function w.r.t $T$ on the disc $D(0,\Delta |\epsilon|^{\chi})$,
w.r.t $z$ on $H_{\beta'}$ for any $0 < \beta' < \beta$ and all $\epsilon \in D(0,\epsilon_{0}) \setminus \{ 0 \}$.

Bearing in mind the basic properties of the Fourier inverse transform described in Proposition 2
and taking notice of the expansions (\ref{expand_T_partial_T_delta_D_U}) and
(\ref{expand_T_partial_T_delta_l_U}), we extract from the latter equality
(\ref{main_equation_Omega_kappa_dp}) the next problem satisfied by $U^{\mathfrak{d}_{p}}(T,z,\epsilon)$, namely
\begin{multline}
( \sum_{l=1}^{q}a_{l} \epsilon^{m_{l}-m_{0}-\alpha k_{l}}T^{k_l} + a_{0})
Q(\partial_{z})U^{\mathfrak{d}_{p}}(T,z,\epsilon)\\+
(\sum_{l=0}^{M}c_{l} \epsilon^{\mu_{l} - 2m_{0} - \alpha h_{l}} T^{h_l})
Q_{1}(\partial_{z})U^{\mathfrak{d}_{p}}(T,z,\epsilon)Q_{2}(\partial_{z})U^{\mathfrak{d}_{p}}(T,z,\epsilon)\\
= \sum_{j=0}^{Q}b_{j}(z) \epsilon^{n_{j} - \alpha b_{j}} T^{b_{j}} + F^{\theta_{F}}(\epsilon^{-\alpha}T,z,\epsilon)
+ \sum_{l=1}^{D} \epsilon^{\Delta_{l}+ \alpha (\delta_{l}-d_{l}) - m_{0}} T^{d_{l}} R_{l}(\partial_{z})
 \partial_{T}^{\delta_l}U^{\mathfrak{d}_{p}}(T,z,\epsilon) \label{main_PDE_U_dp}
\end{multline}
Finally, we put
\begin{equation}
u^{\mathfrak{d}_{p}}(t,z,\epsilon) = \epsilon^{-m_{0}}U^{\mathfrak{d}_{p}}(\epsilon^{\alpha}t,z,\epsilon)
\end{equation}
that constitutes a bounded holomorphic function w.r.t $t$ on $\mathcal{T}_{\epsilon,\chi - \alpha}$ and
w.r.t $z$ on $H_{\beta'}$ for any $0 < \beta' < \beta$, for each fixed $\epsilon \in \mathcal{E}_{p}$,
according to Definition 5. Furthermore, by direct inspection, one can check that the function
$(x,z,\epsilon) \mapsto \epsilon^{m_{0}}u^{\mathfrak{d}_{p}}(x \epsilon^{\chi - \alpha},z,\epsilon)$
is bounded and holomorphic on $X \times \mathcal{E}_{p} \times H_{\beta'}$ for any
given $0 < \beta' < \beta$ and fixed $0 \leq p \leq \varsigma -1$. Moreover,
$u^{\mathfrak{d}_{p}}(t,z,\epsilon)$ solves the main equation (\ref{main_PDE_u}) where the piece of
forcing term $(t,z) \mapsto F^{\theta_{F}}(t,z,\epsilon)$ represents in particular a bounded holomorphic function
w.r.t $t$ on the disc
$D(0,\Delta|\epsilon|^{\chi - \alpha})$, w.r.t $z$ on the strip $H_{\beta'}$ for any given $0 < \beta' < \beta$
and fixed $\epsilon \in D(0,\epsilon_{0}) \setminus \{ 0 \}$.

In the last part of the proof, it remains to justify the bounds
(\ref{difference_u_dp_exp_small_epsilon}). According to the construction given above, we observe that
for each $0 \leq p \leq \varsigma - 1$, 
the function $\epsilon^{m_0}u^{\mathfrak{d}_{p}}(x \epsilon^{\chi - \alpha},z,\epsilon)$ can be written as
a $m_{\kappa}-$Laplace and Fourier inverse transform
\begin{equation}
\epsilon^{m_0}u^{\mathfrak{d}_{p}}(x \epsilon^{\chi - \alpha},z,\epsilon) = 
\frac{\kappa}{(2\pi)^{1/2}} \int_{-\infty}^{+\infty} \int_{L_{\gamma_{p}}}
\omega_{\kappa}^{\mathfrak{d}_{p}}(u,m,\epsilon) \exp( -(\frac{u}{x \epsilon^{\chi}})^{\kappa} )
e^{izm} \frac{du}{u} dm 
\end{equation}
where $L_{\gamma_{p}} = \mathbb{R}_{+}e^{\sqrt{-1}\gamma_{p}} \subset S_{\mathfrak{d}_{p}}$. The steps
of the verification are similar to the arguments disclosed in Theorem 1 of \cite{lama1} but we still decide to
present the details for the benefit of clarity. Namely, using the fact that
the function $u \mapsto \omega_{\kappa}(u,m,\epsilon) \exp( -(\frac{u}{\epsilon^{\chi} x})^{\kappa} )/u$
is holomorphic on $D(0,\rho)$ for all
$(m,\epsilon) \in \mathbb{R} \times (D(0,\epsilon_{0}) \setminus \{ 0 \})$, its integral along the union of a segment starting from
0 to $(\rho/2)e^{i\gamma_{p+1}}$, an arc of circle with radius $\rho/2$ which connects
$(\rho/2)e^{i\gamma_{p+1}}$ and $(\rho/2)e^{i\gamma_{p}}$ and a segment starting from
$(\rho/2)e^{i\gamma_{p}}$ to 0, is vanishing. Therefore, we can write the difference
$\epsilon^{m_0}u^{\mathfrak{d}_{p+1}} - \epsilon^{m_0}u^{\mathfrak{d}_{p}}$ as a sum of three integrals,
\begin{multline}
\epsilon^{m_0}u^{\mathfrak{d}_{p+1}}(x\epsilon^{\chi - \alpha},z,\epsilon) -
\epsilon^{m_0}u^{\mathfrak{d}_{p}}(x\epsilon^{\chi - \alpha},z,\epsilon) \\
=
\frac{\kappa}{(2\pi)^{1/2}}\int_{-\infty}^{+\infty}
\int_{L_{\rho/2,\gamma_{p+1}}}
\omega_{\kappa}^{\mathfrak{d}_{p+1}}(u,m,\epsilon) e^{-(\frac{u}{\epsilon^{\chi} x})^{\kappa}}
e^{izm} \frac{du}{u} dm\\ -
\frac{\kappa}{(2\pi)^{1/2}}\int_{-\infty}^{+\infty}
\int_{L_{\rho/2,\gamma_{p}}}
\omega_{\kappa}^{\mathfrak{d}_p}(u,m,\epsilon) e^{-(\frac{u}{\epsilon^{\chi} x})^{\kappa}}
e^{izm} \frac{du}{u} dm\\
+ \frac{\kappa}{(2\pi)^{1/2}}\int_{-\infty}^{+\infty}
\int_{C_{\rho/2,\gamma_{p},\gamma_{p+1}}}
\omega_{\kappa}(u,m,\epsilon) e^{-(\frac{u}{\epsilon^{\chi} x})^{\kappa}}
e^{izm} \frac{du}{u} dm \label{difference_epsilon_m0_u_dp_decomposition}
\end{multline}
where $L_{\rho/2,\gamma_{p+1}} = [\rho/2,+\infty)e^{i\gamma_{p+1}}$,
$L_{\rho/2,\gamma_{p}} = [\rho/2,+\infty)e^{i\gamma_{p}}$ and
$C_{\rho/2,\gamma_{p},\gamma_{p+1}}$ is an arc of circle with radius connecting
$(\rho/2)e^{i\gamma_{p}}$ and $(\rho/2)e^{i\gamma_{p+1}}$ with a well chosen orientation.\medskip

We give estimates for the quantity
$$ I_{1} = \left| \frac{\kappa}{(2\pi)^{1/2}}\int_{-\infty}^{+\infty}
\int_{L_{\rho/2,\gamma_{p+1}}}
\omega_{\kappa}^{\mathfrak{d}_{p+1}}(u,m,\epsilon) e^{-(\frac{u}{\epsilon^{\chi} x})^{\kappa}}
e^{izm} \frac{du}{u} dm \right|.
$$
By construction, the direction $\gamma_{p+1}$ (which depends on $\epsilon^{\chi} x$) is chosen in such
a way that
$\cos( \kappa( \gamma_{p+1} - \mathrm{arg}(\epsilon^{\chi} x) )) \geq \delta_{1}$, for all
$\epsilon \in \mathcal{E}_{p} \cap \mathcal{E}_{p+1}$, all $x \in X$, for some fixed $\delta_{1} > 0$.
From the estimates (\ref{omega_kappa_dp_in_F}), we get that
\begin{multline}
I_{1} \leq \frac{\kappa}{(2\pi)^{1/2}} \int_{-\infty}^{+\infty} \int_{\rho/2}^{+\infty}
\varpi (1+|m|)^{-\mu} e^{-\beta|m|}
\frac{ \frac{r}{|\epsilon|^{\chi}}}{1 + (\frac{r}{|\epsilon|^{\chi}})^{2\kappa} } \\
\times \exp( \nu (\frac{r}{|\epsilon|^{\chi}})^{\kappa} )
\exp(-\frac{\cos(\kappa(\gamma_{p+1} -
\mathrm{arg}(\epsilon^{\chi} x)))}{|\epsilon^{\chi} x|^{\kappa}}
r^{\kappa}) e^{-m\mathrm{Im}(z)} \frac{dr}{r} dm\\
\leq \frac{\kappa \varpi}{(2\pi)^{1/2}} \int_{-\infty}^{+\infty} e^{-(\beta - \beta')|m|} dm
\int_{\rho/2}^{+\infty}\frac{1}{|\epsilon|^{\chi}}
\exp( -(\frac{\delta_{1}}{|x|^{\kappa}} - \nu)(\frac{r}{|\epsilon|^{\chi}})^{\kappa} ) dr\\
\leq  \frac{2\kappa \varpi}{(2\pi)^{1/2}} \int_{0}^{+\infty} e^{-(\beta - \beta')m} dm
\int_{\rho/2}^{+\infty} \frac{|\epsilon|^{\chi (\kappa-1)}}{(\frac{\delta_1}{|x|^{\kappa}} - \nu)\kappa
(\frac{\rho}{2})^{\kappa-1}}
\times \frac{ (\frac{\delta_1}{|x|^{\kappa}} - \nu)\kappa r^{\kappa-1} }{|\epsilon|^{\chi \kappa}}
\exp( -(\frac{\delta_{1}}{|x|^{\kappa}} - \nu)(\frac{r}{|\epsilon|^{\chi}})^{\kappa} ) dr\\
\leq
\frac{2 \kappa \varpi}{(2\pi)^{1/2}} \frac{|\epsilon|^{\chi (\kappa-1)}}{(\beta - \beta')
(\frac{\delta_{1}}{|x|^{\kappa}} - \nu) \kappa (\frac{\rho}{2})^{\kappa-1}}
\exp( -(\frac{\delta_1}{|x|^{\kappa}} - \nu) \frac{(\rho/2)^{\kappa}}{|\epsilon|^{\chi \kappa}} )\\
\leq \frac{2 \kappa \varpi}{(2\pi)^{1/2}} \frac{|\epsilon|^{\chi (\kappa-1)}}{(\beta - \beta')
\delta_{2}\kappa(\frac{\rho}{2})^{\kappa-1}} \exp( -\delta_{2}
\frac{(\rho/2)^{\kappa}}{|\epsilon|^{\chi \kappa}} ) \label{I_1_exp_small_order_chi_alpha_kappa}
\end{multline}
for all $x \in X$ and $|\mathrm{Im}(z)| \leq \beta'$ with
$|x| < (\frac{\delta_{1}}{\delta_{2} + \nu})^{1/\kappa}$, for some $\delta_{2}>0$, for all
$\epsilon \in \mathcal{E}_{p} \cap \mathcal{E}_{p+1}$.\medskip

In the same way, we also give estimates for the integral
$$ I_{2} = \left| \frac{\kappa}{(2\pi)^{1/2}}\int_{-\infty}^{+\infty}
\int_{L_{\rho/2,\gamma_{p}}}
\omega_{\kappa}^{\mathfrak{d}_{p}}(u,m,\epsilon) e^{-(\frac{u}{\epsilon^{\chi} x})^{\kappa}}
e^{izm} \frac{du}{u} dm \right|.
$$
Namely, the direction $\gamma_{p}$ (which depends on $\epsilon^{\chi} x$) is chosen in such a way that
$\cos( \kappa( \gamma_{p} - \mathrm{arg}(\epsilon^{\chi} x) )) \geq \delta_{1}$, for all
$\epsilon \in \mathcal{E}_{p} \cap \mathcal{E}_{p+1}$, all $x \in X$, for some fixed $\delta_{1} > 0$.
Again from the estimates (\ref{omega_kappa_dp_in_F}) and following the same steps as in
(\ref{I_1_exp_small_order_chi_alpha_kappa}), we deduce that
\begin{equation}
I_{2} \leq \frac{2 \kappa \varpi}{(2\pi)^{1/2}} \frac{|\epsilon|^{\chi (\kappa-1)}}{(\beta - \beta')
\delta_{2}\kappa(\frac{\rho}{2})^{\kappa-1}}
\exp( -\delta_{2} \frac{(\rho/2)^{\kappa}}{|\epsilon|^{\chi \kappa}} )
\label{I_2_exp_small_order_chi_alpha_kappa}
\end{equation}
for all $x \in X$ and $|\mathrm{Im}(z)| \leq \beta'$ with
$|x| < (\frac{\delta_{1}}{\delta_{2} + \nu})^{1/\kappa}$, for some $\delta_{2}>0$, for all
$\epsilon \in \mathcal{E}_{p} \cap \mathcal{E}_{p+1}$.\medskip

Finally, we give upper bound estimates for the integral
$$
I_{3} = \left| \frac{\kappa}{(2\pi)^{1/2}}\int_{-\infty}^{+\infty}
\int_{C_{\rho/2,\gamma_{p},\gamma_{p+1}}}
\omega_{\kappa}(u,m,\epsilon) e^{-(\frac{u}{\epsilon^{\chi} x})^{\kappa}} e^{izm} \frac{du}{u} dm \right|.
$$
By construction, the arc of circle $C_{\rho/2,\gamma_{p},\gamma_{p+1}}$ is chosen in such a way that
$\cos(\kappa(\theta - \mathrm{arg}(\epsilon^{\chi} x))) \geq \delta_{1}$, for all
$\theta \in [\gamma_{p},\gamma_{p+1}]$ (if
$\gamma_{p} < \gamma_{p+1}$), $\theta \in [\gamma_{p+1},\gamma_{p}]$ (if
$\gamma_{p+1} < \gamma_{p}$), for all $x \in X$, all $\epsilon \in \mathcal{E}_{p} \cap \mathcal{E}_{p+1}$, for some
fixed $\delta_{1}>0$. Bearing in mind (\ref{omega_kappa_dp_in_F}) and the classical estimates
$$ \sup_{s \geq 0} s^{m_1} \exp( -m_{2} s ) = (\frac{m_1}{m_2})^{m_1}e^{-m_1} $$
for any $m_{1} \geq 0$, $m_{2}>0$, we get that
\begin{multline}
I_{3} \leq \frac{\kappa}{(2\pi)^{1/2}} \int_{-\infty}^{+\infty}  \left| \int_{\gamma_{p}}^{\gamma_{p+1}} \right.
\varpi (1+|m|)^{-\mu} e^{-\beta|m|}
\frac{ \frac{\rho/2}{|\epsilon|^{\chi}}}{1 + (\frac{\rho/2}{|\epsilon|^{\chi}})^{2\kappa} } \\
\times \exp( \nu (\frac{\rho/2}{|\epsilon|^{\chi}})^{\kappa} )
\exp( -\frac{\cos(\kappa(\theta - \mathrm{arg}(\epsilon^{\chi} x)))}{|\epsilon^{\chi} x|^{\kappa}}
(\frac{\rho}{2})^{\kappa})
\left. e^{-m\mathrm{Im}(z)} d\theta \right| dm\\
\leq \frac{\kappa \varpi }{(2\pi)^{1/2}} \int_{-\infty}^{+\infty}
e^{-(\beta - \beta')|m|} dm \times
|\gamma_{p} - \gamma_{p+1}| \frac{\rho/2}{|\epsilon|^{\chi}}
\exp( -\frac{( \frac{\delta_1}{|x|^{\kappa}} - \nu)}{2} (\frac{\rho/2}{|\epsilon|^{\chi}})^{\kappa}) \\
 \times \exp( -\frac{( \frac{\delta_1}{|x|^{\kappa}} - \nu)}{2}
 (\frac{\rho/2}{|\epsilon|^{\chi}})^{\kappa})\\
\leq \frac{ 2 \kappa \varpi |\gamma_{p} - \gamma_{p+1}|}{(2\pi)^{1/2}(\beta - \beta')}
\sup_{s \geq 0} s^{1/\kappa}e^{-\frac{1}{2}(\frac{\delta_1}{|x|^{\kappa}} - \nu)s} \times
\exp( -\frac{( \frac{\delta_1}{|x|^{\kappa}} - \nu)}{2} (\frac{\rho/2}{|\epsilon|^{\chi}})^{\kappa})\\
\leq \frac{2 \kappa
\varpi |\gamma_{p} - \gamma_{p+1}|}{(2\pi)^{1/2}(\beta - \beta')} (\frac{2/\kappa}{\delta_2})^{1/\kappa}
e^{-1/\kappa} \exp( -\frac{\delta_{2}}{2} (\frac{\rho/2}{|\epsilon|^{\chi}})^{\kappa})
\label{I_3_exp_small_order_chi_alpha_kappa}
\end{multline}
for all $x \in X$ and $|\mathrm{Im}(z)| \leq \beta'$ with
$|x| < (\frac{\delta_{1}}{\delta_{2} + \nu})^{1/\kappa}$, for some $\delta_{2}>0$, for all
$\epsilon \in \mathcal{E}_{p} \cap \mathcal{E}_{p+1}$.\medskip

Finally, gathering the three above inequalities (\ref{I_1_exp_small_order_chi_alpha_kappa}),
(\ref{I_2_exp_small_order_chi_alpha_kappa}) and (\ref{I_3_exp_small_order_chi_alpha_kappa}), we deduce
from the decomposition (\ref{difference_epsilon_m0_u_dp_decomposition}) that
\begin{multline*}
|\epsilon^{m_0}u^{\mathfrak{d}_{p+1}}(x\epsilon^{\chi - \alpha},z,\epsilon) -
\epsilon^{m_0}u^{\mathfrak{d}_{p}}(x\epsilon^{\chi - \alpha},z,\epsilon)| \leq
 \frac{4 \kappa \varpi}{(2\pi)^{1/2}} \frac{|\epsilon|^{\chi (\kappa-1)}}{(\beta - \beta')
\delta_{2}\kappa(\frac{\rho}{2})^{\kappa-1}}
\exp( -\delta_{2} \frac{(\rho/2)^{\kappa}}{|\epsilon|^{\chi \kappa}})\\
+  \frac{2 \kappa \varpi |\gamma_{p} - \gamma_{p+1}|}{(2\pi)^{1/2}(\beta - \beta')}
(\frac{2/\kappa}{\delta_2})^{1/\kappa}
e^{-1/\kappa} \exp( -\frac{\delta_{2}}{2} (\frac{\rho/2}{|\epsilon|^{\chi}})^{\kappa})
\end{multline*}
for all $x \in X$ and $|\mathrm{Im}(z)| \leq \beta'$ with
$|x| < (\frac{\delta_{1}}{\delta_{2} + \nu})^{1/k}$, for some $\delta_{2}>0$, for all
$\epsilon \in \mathcal{E}_{p} \cap \mathcal{E}_{p+1}$. Therefore, the inequality
(\ref{difference_u_dp_exp_small_epsilon}) holds.
\end{proof}

\section{Construction of outer solutions to the main problem}

In this section, we construct solutions of the main equation (\ref{main_PDE_u}) for $t$ in a large sectorial
domain outside the origin and we provide constraints under which their domain of holomorphy in time
can be extended to some $\epsilon-$depending domains in the vicinity of the origin.

\subsection{Classical Laplace transforms}

In this little subsection, we report some identities for the usual Laplace transform of holomorphic functions on
unbounded sectors involving convolution products and derivations. The next lemma has already appeared in our
previous work \cite{ma1} and is classical in reference textbooks such as \cite{ba}.

\begin{lemma} Let $m \geq 0$ be an integer. Let $w_{1}(\tau)$, $w_{2}(\tau)$ be holomorphic functions on
an unbounded open
sector $U_{d}$ centered at 0 with bisecting direction $d \in \mathbb{R}$ such that there exist $C,K>0$ with
$$ |w_{j}(\tau)| \leq C \exp(K|\tau|) \ \ , \ \ j=1,2$$
for all $\tau \in U_{d}$. We denote
$$ w_{1} \ast w_{2}(\tau) = \int_{0}^{\tau} w_{1}(\tau - s)w_{2}(s) ds $$ 
their convolution product on $U_{d}$. We pick up an unbounded sector $\mathcal{D}$ centered at 0 for which there exists
$\delta_{1}>0$ with
$$ d + \mathrm{arg}(t) \in (-\pi/2,\pi/2) \ \ , \ \ \cos(d + \mathrm{arg}(t)) \geq \delta_{1},$$
for all $t \in \mathcal{D}$. Then the following identities hold for the Laplace transforms
\begin{multline*}
\int_{L_{d}} \tau^{m} \exp(-t \tau) d\tau = \frac{m!}{t^{m+1}} \ \ , \ \
\partial_{t}(\int_{L_{d}}w_{1}(\tau) \exp(-t\tau) d\tau) = \int_{L_{d}}(-\tau)w_{1}(\tau) \exp(-t\tau) d\tau,\\
 \int_{L_{d}} w_{1} \ast w_{2}(\tau) \exp(-t \tau) d\tau = (\int_{L_{d}}w_{1}(\tau) \exp(-t\tau) d\tau)
(\int_{L_{d}}w_{2}(\tau) \exp(-t\tau) d\tau)
\end{multline*}
where $L_{d} = \mathbb{R}_{+}e^{id} \subset U_{d} \cup \{0 \}$, for all
$t \in \mathcal{D} \cap \{ |t| > K/\delta_{1} \}$.
\end{lemma}

\subsection{Sets of Banach spaces with exponential growth and decay of order 1}

In this subsection, we study a slightly modified version of the Banach spaces mentioned in subsection 3.2 of this
work in the particular situation of functions with exponential growth of order 1 on unbounded sectors in
$\mathbb{C}$ and exponential decay on $\mathbb{R}$. Although the proofs of the next lemma are proximate to
the ones of the statements disclosed in Subsection 3.2, we decide to present them for the sake of clarity and
convenience for the reader.

\begin{defin} Let $U_{d}$ be an open unbounded sector with bisecting direction $d \in \mathbb{R}$ and
$\mathcal{E}$ be an open sector with finite radius $r_{\mathcal{E}}$, both centered at $0$ in $\mathbb{C}$. Let
$\nu,\rho>0$ and $\beta>0,\Gamma \geq 0,\mu>1$ be real numbers and let $\epsilon \in \mathcal{E}$. We define
$E_{(\nu,\beta,\mu,\Gamma,\epsilon)}^{d}$ as the space of continuous functions $(\tau,m) \mapsto f(\tau,m)$
on $(\bar{D}(0,\rho) \cup U_{d}) \times \mathbb{R}$ with values in $\mathbb{C}$, holomorphic w.r.t $\tau$
on $D(0,\rho) \cup U_{d}$, with
$$ ||f(\tau,m)||_{(\nu,\beta,\mu,\Gamma,\epsilon)} = \sup_{\tau \in \bar{D}(0,\rho) \cup U_{d}, m \in \mathbb{R}}
(1 + |m|)^{\mu} e^{\beta |m|} (1 + |\frac{\tau}{\epsilon^{\Gamma}}|^{2})
\exp( -\nu |\frac{\tau}{\epsilon^{\Gamma}}| ) |f(\tau,m)|
$$
is finite. It turns out that the normed space
$(E_{(\nu,\beta,\mu,\Gamma,\epsilon)}^{d},||.||_{(\nu,\beta,\mu,\Gamma,\epsilon)})$ is a Banach space. 
\end{defin}
{\bf Remark:} Compared to the space $F_{(\nu,\beta,\mu,\chi,1,\epsilon)}^{d}$
mentioned above, the functions from $E_{(\nu,\beta,\mu,\Gamma,\epsilon)}^{d}$
do not need to vanish at $\tau=0$.

\begin{lemma}
Let $\gamma_{2} \geq 0$ be an integer. Take $B(m) \in E_{(\beta,\mu)}$ for some real numbers
$\beta>0$ and $\mu>1$. Then,
$\tau^{\gamma_2}B(m)$ belongs to $E_{(\nu,\beta,\mu,\Gamma,\epsilon)}^{d}$ for any real numbers
$\nu>0$,$\Gamma \geq 0$ and $\epsilon \in \mathcal{E}$. Moreover, there exists a constant $B_{1}>0$ (depending on
$\gamma_{2},\nu$) such that
\begin{equation}
||\tau^{\gamma_2}B(m)||_{(\nu,\beta,\mu,\Gamma,\epsilon)} \leq B_{1} ||B(m)||_{(\beta,\mu)}
|\epsilon|^{\Gamma \gamma_{2}}.
\end{equation}
\end{lemma}
\begin{proof} By definition, we can write
\begin{multline}
||\tau^{\gamma_2}B(m)||_{(\nu,\beta,\mu,\Gamma,\epsilon)} =
\sup_{\tau \in \bar{D}(0,\rho) \cup U_{d}, m \in \mathbb{R}}
(1 + |m|)^{\mu} e^{\beta |m|} |B(m)| (1 + |\frac{\tau}{\epsilon^{\Gamma}}|^{2})\\
\times \exp( -\nu |\frac{\tau}{\epsilon^{\Gamma}}| )
|\frac{\tau}{\epsilon^{\Gamma}}|^{\gamma_2}|\epsilon|^{\Gamma \gamma_{2}} \leq
||B(m)||_{(\beta,\mu)} (\sup_{x \geq 0} (1 + x^2)e^{-\nu x} x^{\gamma_2}) |\epsilon|^{\Gamma \gamma_{2}}
\end{multline}
from which the lemma follows owing to the fact that an exponential function grows faster than any polynomial.
\end{proof}

\begin{lemma} Let $\gamma_{1},\gamma_{2},\gamma_{3} \geq 0$ be real numbers. We assume that
\begin{equation}
\gamma_{1} \leq \gamma_{2} + \gamma_{3} + 1 \ \ , \ \ \gamma_{1} \geq \gamma_{3}.
\label{constraints_Gamma_conv_s_gamma_f} 
\end{equation}
Then, there exists a constant $B_{2}>0$ (depending on $\gamma_{1},\gamma_{2},\gamma_{3},\nu$) such that
\begin{equation}
|| \frac{1}{\tau^{\gamma_1}} \int_{0}^{\tau} (\tau - s)^{\gamma_2} s^{\gamma_3} f(s,m) ds
||_{(\nu,\beta,\mu,\Gamma,\epsilon)} \leq B_{2}
||f(\tau,m)||_{(\nu,\beta,\mu,\Gamma,\epsilon)}
|\epsilon|^{\Gamma( \gamma_{2} + \gamma_{3} + 1) - \Gamma\gamma_{1}} \label{norm_Gamma_conv_s_gamma_f}
\end{equation}
for all $f(\tau,m) \in E_{(\nu,\beta,\mu,\Gamma,\epsilon)}^{d}$.
\end{lemma}
\begin{proof} By factoring out the pieces composing the norm of $f(\tau,m)$, we can rewrite the left handside of
(\ref{norm_Gamma_conv_s_gamma_f}) as
\begin{multline}
A = || \frac{1}{\tau^{\gamma_1}} \int_{0}^{\tau} (\tau - s)^{\gamma_2} s^{\gamma_3} f(s,m) ds
||_{(\nu,\beta,\mu,\Gamma,\epsilon)}\\
= \sup_{\tau \in \bar{D}(0,\rho) \cup U_{d},m \in \mathbb{R}}
(1 + |m|)^{\mu} e^{\beta |m|} (1 + |\frac{\tau}{\epsilon^{\Gamma}}|^{2})
\exp( -\nu |\frac{\tau}{\epsilon^{\Gamma}}| )\\
\times \left|\frac{1}{\tau^{\gamma_1}}\int_{0}^{\tau} \left\{ (1 + |m|)^{\mu} e^{\beta |m|}
(1 + |\frac{s}{\epsilon^{\Gamma}}|^{2}) \exp( - \nu |\frac{s}{\epsilon^{\Gamma}}| ) f(s,m) \right\}
\mathcal{A}(\tau,s,m,\epsilon) ds \right| \label{A=}
\end{multline}
where
$$ \mathcal{A}(\tau,s,m,\epsilon) = \frac{1}{(1 + |m|)^{\mu}e^{\beta|m|}}
\frac{\exp( \nu |\frac{s}{\epsilon^{\Gamma}}| )}{1 + |\frac{s}{\epsilon^{\Gamma}}|^{2}}
(\tau - s)^{\gamma_2}s^{\gamma_3}. $$
As a result, we obtain
\begin{equation}
A \leq B_{2.1}(\epsilon) ||f(\tau,m)||_{(\nu,\beta,\mu,\Gamma,\epsilon)} \label{A_<_norm_f}
\end{equation}
where
$$ B_{2.1}(\epsilon) = \sup_{\tau \in \bar{D}(0,\rho) \cup U_{d}}
(1 + |\frac{\tau}{\epsilon^{\Gamma}}|^{2}) \exp( -\nu |\frac{\tau}{\epsilon^{\Gamma}}| )
\frac{1}{|\tau|^{\gamma_1}} \int_{0}^{|\tau|}
\frac{\exp( \nu \frac{h}{|\epsilon|^{\Gamma}} )}{ 1 + (\frac{h}{|\epsilon|^{\Gamma}})^{2} }
(|\tau| - h)^{\gamma_{2}} h^{\gamma_{3}} dh .$$
We perform the change of variable $h = |\epsilon|^{\Gamma}h'$ inside the integral part of $B_{2.1}(\epsilon)$ and
get the next bounds
\begin{multline}
B_{2.1}(\epsilon) = \sup_{\tau \in \bar{D}(0,\rho) \cup U_{d}}
(1 + |\frac{\tau}{\epsilon^{\Gamma}}|^{2}) \exp( -\nu |\frac{\tau}{\epsilon^{\Gamma}}| )
\frac{1}{(\frac{|\tau|}{|\epsilon|^{\Gamma}} |\epsilon|^{\Gamma})^{\gamma_1}}
|\epsilon|^{\Gamma(\gamma_{2}+\gamma_{3}+1)} \\
\times \int_{0}^{\frac{|\tau|}{|\epsilon|^{\Gamma}}} \frac{e^{\nu h'}}{1 + (h')^{2}}
(\frac{|\tau|}{|\epsilon|^{\Gamma}} - h')^{\gamma_2} (h')^{\gamma_{3}} dh'
\leq |\epsilon|^{\Gamma (\gamma_{2}+\gamma_{3}+1) - \Gamma \gamma_{1}}
\sup_{x \geq 0} (1 + x^2)e^{-\nu x}
\frac{1}{x^{\gamma_1}} G(x) \label{B21_estimates}
\end{multline}
where
$$ G(x) = \int_{0}^{x} \frac{e^{\nu h'}}{1 + (h')^{2}} (x-h')^{\gamma_{2}} (h')^{\gamma_3} dh'. $$
In the last part of the proof, we need to study the function $G(x)$ near 0 and $+\infty$. In order to investigate its
behaviour in the vicinity of the origin, we make the change of variable $h'=xu$ inside $G(x)$, getting
\begin{equation}
G(x) = x^{\gamma_{2} + \gamma_{3} + 1} \int_{0}^{1} \frac{e^{\nu xu}}{1 + (xu)^{2}}
(1 - u)^{\gamma_2}u^{\gamma_3} du. \label{G_near_0}
\end{equation}
From the first constraint in (\ref{constraints_Gamma_conv_s_gamma_f}), we deduce that $G(x)/x^{\gamma_1}$ is
bounded near 0. For large values of $x$, we proceed as in Proposition 1 of \cite{lama2} and split $G(x)$ into two
pieces
$$ G(x) = G_{1}(x) + G_{2}(x) $$
where
$$ G_{1}(x) = \int_{0}^{x/2} \frac{e^{\nu h'}}{1 + (h')^{2}} (x-h')^{\gamma_{2}} (h')^{\gamma_3} dh' \ \ , \ \
G_{2}(x) = \int_{x/2}^{x} \frac{e^{\nu h'}}{1 + (h')^{2}} (x-h')^{\gamma_{2}} (h')^{\gamma_3} dh' $$
Since $\gamma_{2} \geq 0$, we notice that $(x-h')^{\gamma_2} \leq x^{\gamma_2}$ for $0 \leq h' \leq x/2$. Therefore,
$$ G_{1}(x) \leq x^{\gamma_2}e^{\nu x/2}\int_{0}^{x/2} (h')^{\gamma_3} dh' =
x^{\gamma_2}e^{\nu x/2} \frac{(x/2)^{\gamma_{3}+1}}{\gamma_{3}+1}. $$
Accordingly, we get that
\begin{equation}
\sup_{x \geq 1} (1+x^{2})e^{-\nu x} \frac{1}{x^{\gamma_1}} G_{1}(x) \label{G1_near_infty} 
\end{equation}
is finite. On the other hand, we check that $1 + (h')^{2} \geq 1 + (x/2)^{2}$ for $x/2 \leq h' \leq x$. Hence,
$$ G_{2}(x) \leq \frac{1}{1 + (x/2)^{2}} G_{2.1}(x) $$
where
$$ G_{2.1}(x) = \int_{x/2}^{x} e^{\nu h'} (h')^{\gamma_3} (x-h')^{\gamma_2} dh' $$
Bestowing the estimates (18) in \cite{lama2}, we get a constant $K_{2.1}>0$ (depending on
$\nu,\gamma_{2},\gamma_{3}$) such that
$$ G_{2.1}(x) \leq K_{2.1}x^{\gamma_3}e^{\nu x} $$
for all $x \geq 1$. It follows that
\begin{equation}
\sup_{x \geq 1} (1+x^{2})e^{-\nu x} \frac{1}{x^{\gamma_1}} G_{2}(x) \label{G2_near_infty} 
\end{equation}
is finite provided that the second constraint from (\ref{constraints_Gamma_conv_s_gamma_f}) holds.
Finally, collecting (\ref{A=}), (\ref{A_<_norm_f}), (\ref{B21_estimates}),
(\ref{G_near_0}), (\ref{G1_near_infty}) and (\ref{G2_near_infty})
yields the estimates (\ref{norm_Gamma_conv_s_gamma_f}).
\end{proof}

\begin{lemma} Let $Q_{1}(X),Q_{2}(X)$ and $R_{D}(X)$ belonging to $\mathbb{C}[X]$ with $R_{D}(im) \neq 0$ for all
$m \in \mathbb{R}$ and
\begin{equation}
\mathrm{deg}(R_{D}) \geq \mathrm{deg}(Q_{1}) \ \ , \ \ \mathrm{deg}(R_{D}) \geq \mathrm{deg}(Q_{2}). 
\end{equation}
Besides, we choose the real parameter $\mu > 1$ with
$\mu > \max( \mathrm{deg}(Q_{1}) + 1, \mathrm{deg}(Q_{2}) + 1)$. Then, there exists a constant $B_{3}>0$
(depending on $\mu,Q_{1},Q_{2},R_{D}$) such that
\begin{multline}
||\frac{1}{R_{D}(im)} \int_{0}^{\tau} \int_{-\infty}^{+\infty} Q_{1}(i(m-m_{1}))f(\tau - s,m-m_{1})
Q_{2}(im_{1})g(s,m_{1}) ds dm_{1} ||_{(\nu,\beta,\mu,\Gamma,\epsilon)}\\
\leq B_{3}|\epsilon|^{\Gamma}||f(\tau,m)||_{(\nu,\beta,\mu,\Gamma,\epsilon)}
||g(\tau,m)||_{(\nu,\beta,\mu,\Gamma,\epsilon)} \label{norm_double_conv_f_g_Gamma}
\end{multline}
for all $f(\tau,m),g(\tau,m) \in E_{(\nu,\beta,\mu,\Gamma,\epsilon)}^{d}$.
\end{lemma}
\begin{proof} As above, by setting apart the terms stemming from the norms of $f$ and $g$, we can reorganize
the left handside of (\ref{norm_double_conv_f_g_Gamma}) as follows
\begin{multline}
K =  ||\frac{1}{R_{D}(im)} \int_{0}^{\tau} \int_{-\infty}^{+\infty} Q_{1}(i(m-m_{1}))f(\tau - s,m-m_{1})\\
\times
Q_{2}(im_{1})g(s,m_{1}) ds dm_{1} ||_{(\nu,\beta,\mu,\Gamma,\epsilon)} =
\sup_{\tau \in \bar{D}(0,\rho) \cup U_{d}, m \in \mathbb{R}}
(1 + |m|)^{\mu}e^{\beta|m|}(1 + |\frac{\tau}{\epsilon^{\Gamma}}|^{2})
\exp(-\nu |\frac{\tau}{\epsilon^{\Gamma}}| ) \\
\times |\frac{1}{R_{D}(im)} \int_{0}^{\tau}
\int_{-\infty}^{+\infty} \{ (1 + |m-m_{1}|)^{\mu} e^{\beta|m-m_{1}|}
(1 + (\frac{|\tau - s|}{|\epsilon|^{\Gamma}})^{2}) \exp(-\nu |\frac{\tau-s}{\epsilon^{\Gamma}}| )
f(\tau-s,m-m_{1}) \}\\
\times \{ (1 + |m_{1}|)^{\mu}e^{\beta|m_{1}|}(1 + |\frac{s}{\epsilon^{\Gamma}}|^{2})
\exp(-\nu |\frac{s}{\epsilon^{\Gamma}}|) g(s,m_{1}) \} \mathcal{K}(\tau,s,m,m_{1}) ds dm_{1}|
\label{defin_K}
\end{multline}
where
\begin{multline*}
\mathcal{K}(\tau,s,m,m_{1}) = \\
\frac{e^{-\beta|m-m_{1}|}e^{-\beta|m_{1}|}}{(1 + |m-m_{1}|)^{\mu}
(1 + |m_{1}|)^{\mu}} Q_{1}(i(m-m_{1}))Q_{2}(im_{1})
\frac{\exp( \nu |\frac{\tau - s}{\epsilon^{\Gamma}}| ) \exp( \nu|\frac{s}{\epsilon^{\Gamma}}| )}{
(1 + (\frac{|\tau - s|}{|\epsilon|^{\Gamma}})^{2})(1 + |\frac{s}{\epsilon^{\Gamma}}|^{2})}
\end{multline*}
According to the triangular inequality, $|m| \leq |m-m_{1}| + |m_{1}|$ for all $m,m_{1} \in \mathbb{R}$, we get that
\begin{equation}
K \leq B_{3.1} B_{3.2}(\epsilon)
||f(\tau,m)||_{(\nu,\beta,\mu,\Gamma,\epsilon)}||g(\tau,m)||_{(\nu,\beta,\mu,\Gamma,\epsilon)} \label{norm_K}
\end{equation}
where
$$ B_{3.1} = \sup_{m \in \mathbb{R}} \frac{(1 + |m|)^{\mu}}{|R_{D}(im)|}
\int_{-\infty}^{+\infty} \frac{ |Q_{1}(i(m-m_{1}))||Q_{2}(im_{1})| }{(1 + |m-m_{1}|)^{\mu}(1 + |m_{1}|)^{\mu}}
dm_{1} $$
and
$$B_{3.2}(\epsilon) = \sup_{\tau \in \bar{D}(0,\rho) \cup U_{d}}
(1 + |\frac{\tau}{\epsilon^{\Gamma}}|^{2})\int_{0}^{|\tau|}
\frac{1}{1 + \frac{(|\tau|-h')^{2}}{|\epsilon|^{2\Gamma}}}
\frac{1}{1 + \frac{(h')^{2}}{|\epsilon|^{2\Gamma}}} dh'.$$
Since $B_{3.1}=C_{3.1}$ in formula (\ref{C31_defin}), we deduce from the bounds (\ref{bounds_C31}), that
$B_{3.1}$ is finite. Besides, by operating the change of variable $h' = |\epsilon|^{\Gamma}h$ inside the
integral piece of $B_{3.2}(\epsilon)$, we observe that
\begin{equation}
B_{3.2}(\epsilon) = \sup_{\tau \in \bar{D}(0,\rho) \cup U_{d}}
(1 + |\frac{\tau}{\epsilon^{\Gamma}}|^{2}) |\epsilon|^{\Gamma}
\int_{0}^{\frac{|\tau|}{|\epsilon|^{\Gamma}}} \frac{1}{1 + (\frac{|\tau|}{|\epsilon|^{\Gamma}} - h)^{2}}
\frac{1}{1 + h^{2}} dh\\
\leq |\epsilon|^{\Gamma} \sup_{x \geq 0} \tilde{B}_{3.2}(x) \label{bounds_B32_epsilon}
\end{equation}
where
$$\tilde{B}_{3.2}(x) = (1 + x^{2})\int_{0}^{x} \frac{1}{1 + (x-h)^{2}} \frac{1}{1 + h^{2}} dh.$$
In accordance with Corollary 4.9 of \cite{cota}, we get that $\sup_{x \geq 0} \tilde{B}_{3.2}(x)$ is finite.
Gathering (\ref{defin_K}), (\ref{norm_K}) and (\ref{bounds_B32_epsilon}) furnishes the result.
\end{proof}

\subsection{Construction of formal expressions solutions of the main equation as classical Laplace
and Fourier inverse transforms}

Within this subsection, we search for solutions of the main equation (\ref{main_PDE_u}) expressed as
integral representations through classical Laplace and Fourier inverse transforms
\begin{equation}
v(t,z,\epsilon) = \frac{\epsilon^{\gamma_{0}}}{(2\pi)^{1/2}}
\int_{-\infty}^{+\infty} \int_{L_{\mathfrak{u}}} W(u,m,\epsilon) \exp( -(\frac{t}{\epsilon^{\gamma}})u )
e^{izm} du dm \label{defin_v_Laplace_Fourier}
\end{equation}
for some real number $\gamma_{0} \in \mathbb{R}$, where $\gamma>1/2$ is the positive real number introduced in
formula (\ref{defin_F_tzepsilon}) and $L_{\mathfrak{u}} = \mathbb{R}_{+}e^{\sqrt{-1}\mathfrak{u}}$ is a
halfline with direction
$\mathfrak{u} \in \mathbb{R}$. Our prominent goal is the presentation of a related problem satisfied by the
expression $W(u,m,\epsilon)$ that is planned to be solved in the next subsection among the Banach spaces introduced
in the previous subsection.

Overall this subsection, let us assume that the function $W(\tau,m,\epsilon)$ belongs to the Banach space
$E_{(\nu,\beta,\mu,\Gamma,\epsilon)}^{d}$ for some positive real numbers $\nu,\beta>0$, $\mu > 1$ and
$0 \leq \Gamma < \gamma$, with $\epsilon$ belonging to some punctured disc $D(0,\epsilon_{0}) \setminus \{ 0 \}$.
The unbounded sector $U_{d}$ is properly chosen in a way that it avoids the roots
of the polynomial $F_{2}(\tau)$ introduced in the expression (\ref{defin_omega_F}). According to Lemma 4 and
Proposition 2, we can check that the expression $v(t,z,\epsilon)$ given in (\ref{defin_v_Laplace_Fourier}) is
well defined for all $t \in \mathbb{C}$, $\epsilon \in D(0,\epsilon_{0}) \setminus \{ 0 \}$ and
$\mathfrak{u} \in \mathbb{R}$ such that
$$ \mathfrak{u} + \mathrm{arg}(t/\epsilon^{\gamma}) \in (-\pi/2,\pi/2) \ \ , \ \
\mathrm{cos}( \mathfrak{u} + \mathrm{arg}(t/\epsilon^{\gamma}) ) \geq \delta_{1} $$
for some $\delta_{1}>0$, provided that $|t| > \frac{\nu}{\delta_{1}}|\epsilon|^{\gamma - \Gamma}$ and
$z \in H_{\beta}$.

We make the following assumption
\begin{equation}
d_{D} \geq d_{i} \ \ , \ \ d_{D} \geq k_{j} \ \ , \ \ d_{D} \geq b_{k} \ \ , \ \ d_{D} \geq h_{l}
\label{constraints_d_D_di_kj_bk_hl}
\end{equation}
for $1 \leq i \leq D-1$, $1 \leq j \leq q$, $0 \leq k \leq Q$ and $0 \leq l \leq M$. Moreover, the real
numbers $\gamma,\gamma_{0}$ are selected in such a way that
\begin{equation}
\Delta_{D} = \gamma \delta_{D} - \gamma_{0}. \label{constraint_DeltaD_gamma}
\end{equation}
We divide (\ref{main_PDE_u}) by $t^{d_{D}}$ and we focus our attention on the next problem
\begin{multline}
(\sum_{l=1}^{q} a_{l} \epsilon^{m_{l}} t^{k_{l}-d_{D}} +
a_{0}\epsilon^{m_{0}}t^{-d_{D}}) Q(\partial_{z}) v(t,z,\epsilon) \\
+
(\sum_{l=0}^{M} c_{l} \epsilon^{\mu_{l}} t^{h_{l}-d_{D}})Q_{1}(\partial_{z})v(t,z,\epsilon)
Q_{2}(\partial_{z})v(t,z,\epsilon) \\
= \sum_{j=0}^{Q} b_{j}(z) \epsilon^{n_j} t^{b_{j}-d_{D}} + t^{-d_{D}}F^{\theta_{F}}(t,z,\epsilon) +
\epsilon^{\gamma \delta_{D} - \gamma_{0}} \partial_{t}^{\delta_{D}} R_{D}(\partial_{z})v(t,z,\epsilon)\\
+
\sum_{l=1}^{D-1} \epsilon^{\Delta_l} t^{d_{l}-d_{D}} \partial_{t}^{\delta_l} R_{l}(\partial_{z})v(t,z,\epsilon)
\label{main_PDE_u_divided_tdD}
\end{multline}
By means of the identities displayed in Lemma 4 for the classical Laplace transform and in Proposition 2 for the
Fourier inverse transform, we see that $v(t,z,\epsilon)$ given by (\ref{defin_v_Laplace_Fourier})
solves the equation (\ref{main_PDE_u_divided_tdD}) if the related
function $W(\tau,m,\epsilon)$ solves the next nonlinear convolution equation
\begin{multline}
\sum_{l=1}^{q} \frac{a_{l}}{(d_{D} - k_{l} - 1)!}
\frac{\epsilon^{m_{l} + \gamma_{0}}}{\epsilon^{\gamma(d_{D} - k_{l})}} Q(im) \int_{0}^{\tau}
(\tau - s)^{d_{D}-k_{l}-1}W(s,m,\epsilon) ds \\
+ \frac{a_{0}}{(d_{D}-1)!}
\frac{\epsilon^{m_{0} + \gamma_{0}}}{\epsilon^{\gamma d_{D}}}Q(im) \int_{0}^{\tau}
(\tau - s)^{d_{D}-1}W(s,m,\epsilon) ds + \sum_{l=0}^{M}
\frac{c_{l}}{(d_{D}-h_{l}-1)!} \frac{\epsilon^{\mu_{l} + 2\gamma_{0}}}{\epsilon^{\gamma(d_{D} - h_{l})}}\\
\times 
\int_{0}^{\tau} (\tau-s)^{d_{D}-h_{l}-1} \frac{1}{(2\pi)^{1/2}}
\int_{-\infty}^{+\infty} \int_{0}^{s} Q_{1}(i(m-m_{1}))W(s-s',m-m_{1},\epsilon)\\
\times
Q_{2}(im_{1})W(s',m_{1},\epsilon) ds'dm_{1}ds = \sum_{j=0}^{Q}
\frac{1}{(d_{D}-b_{j}-1)!} \frac{\epsilon^{n_j}}{\epsilon^{\gamma(d_{D}-b_{j})}} B_{j}(m)
\tau^{d_{D}-b_{j}-1} \\
+ \Upsilon(\tau,m,\epsilon) + (-\tau)^{\delta_{D}}R_{D}(im)W(\tau,m,\epsilon) +
\sum_{l=1}^{D-1} \frac{1}{(d_{D}-d_{l}-1)!}
\frac{\epsilon^{\Delta_{l} + \gamma_{0}}}{\epsilon^{\gamma( d_{D} - d_{l} + \delta_{l})}} R_{l}(im)\\
\times 
\int_{0}^{\tau} (\tau-s)^{d_{D}-d_{l}-1} (-s)^{\delta_{l}}W(s,m,\epsilon) ds \label{main_conv_eq_W}
\end{multline}
where
$$ \Upsilon(\tau,m,\epsilon) = \frac{1}{(d_{D}-1)!} \frac{\epsilon^{n_{F}}}{\epsilon^{\gamma d_{D}}}
\left( \int_{0}^{\tau} (\tau - s)^{d_{D}-1} \omega_{F}(s,m) ds - 
\tau^{d_{D}-1}\int_{L_{\theta_{F}}} \omega_{F}(u,m)du \right) $$

\subsection{Construction of actual solutions of some auxiliary nonlinear convolution equation with
complex parameter}

The major purpose of this subsection is the construction of a unique solution of the problem
(\ref{main_conv_eq_W}) located in the Banach spaces introduced in Subsection 4.2.

\noindent We first select an unbounded open sector $U_{d}$ with bisecting direction $d \in \mathbb{R}$ taken
in a way that it does not contain any root of the polynomial $F_{2}(\tau)$ appearing in the expression
(\ref{defin_omega_F}).\medskip

\noindent In an initial step, we prove that $\Upsilon(\tau,m,\epsilon)/\tau^{\delta_{D}}R_{D}(im)$ belongs
to $E_{(\nu,\beta,\mu,\Gamma,\epsilon)}^{d}$, for $\beta > 0$ and $\mu>1$ set above in
(\ref{defin_b_j}), for some $\nu>0$ (depending on $K_{F}$, $\Gamma$ and $\epsilon_{0}$),
with $0 \leq \Gamma < \gamma$, for all $\epsilon \in D(0,\epsilon_{0}) \setminus \{ 0 \}$ granting that
\begin{equation}
d_{D} \geq 1 + \delta_{D} \ \ , \ \ \delta_{D} \geq 0 \ \ , \ \ \label{constraint_Upsilon}
n_{F} + \Gamma(d_{D}-1-\delta_{D}) - \gamma d_{D} \geq 0
\end{equation}
hold. As a primary task, we check that the function $\omega_{F}(\tau,m)$ belongs to
$E_{(\nu,\beta,\mu,\Gamma,\epsilon)}^{d}$. Indeed, since $U_{d}$ is taken as above and from the fact that
$\mathrm{deg}(F_{1}) \leq \mathrm{deg}(F_{2})$, we get a constant $C_{F_{1},F_{2}}>0$ with
$$ |\frac{F_{1}(\tau)}{F_{2}(\tau)}| \leq C_{F_{1},F_{2}} $$
for all $U_{d} \cup D(0,\rho)$, for some $\rho>0$ selected small enough. We deduce the next estimates
\begin{multline*}
||\omega_{F}(\tau,m)||_{(\nu,\beta,\mu,\Gamma,\epsilon)} =
\sup_{\tau \in U_{d} \cup \bar{D}(0,\rho),m \in \mathbb{R}} (1 + |m|)^{\mu}
e^{\beta |m|}(1 + |\frac{\tau}{\epsilon^{\Gamma}}|^{2})
\exp( -\nu |\frac{\tau}{\epsilon^{\Gamma}}| )\\
\times |C_{F}(m) e^{-K_{F}\tau}
\frac{F_{1}(\tau)}{F_{2}(\tau)}|\\
\leq ||C_{F}(m)||_{(\beta,\mu)}C_{F_{1},F_{2}} \sup_{\tau \in U_{d} \cup \bar{D}(0,\rho)}
(1 + |\frac{\tau}{\epsilon^{\Gamma}}|^{2}) \exp( -\nu |\frac{\tau}{\epsilon^{\Gamma}}| )
\exp( K_{F}|\epsilon|^{\Gamma} |\frac{\tau}{\epsilon^{\Gamma}}| )\\
\leq ||C_{F}(m)||_{(\beta,\mu)}C_{F_{1},F_{2}} \sup_{x \geq 0} (1 + x^{2})
\exp( (-\nu + K_{F}|\epsilon|^{\Gamma})x )
\end{multline*}
which is finite accepting that $|\epsilon|^{\Gamma} < \nu/K_{F}$. Next in order, from Lemma 6, we get a
constant $B_{2}>0$ (depending on $\delta_{D},d_{D},\nu$) with
\begin{multline}
|| \frac{\epsilon^{n_F}}{\epsilon^{\gamma d_{D}}}
\frac{1}{\tau^{\delta_D}R_{D}(im)}\int_{0}^{\tau} (\tau - s)^{d_{D}-1} \omega_{F}(s,m) ds
||_{(\nu,\beta,\mu,\Gamma,\epsilon)} \\
\leq B_{2} \frac{1}{\mathrm{inf}_{m \in \mathbb{R}}|R_{D}(im)|}
\frac{|\epsilon|^{n_{F}}}{|\epsilon|^{\gamma d_{D}}} |\epsilon^{\Gamma d_{D} - \Gamma \delta_{D}}|
||\omega_{F}(\tau,m)||_{(\nu,\beta,\mu,\Gamma,\epsilon)} \label{bounds_conv_omega_F}
\end{multline}
taking into account that $d_{D} \geq \delta_{D}$ and $\delta_{D} \geq 0$ which follows from
(\ref{constraint_Upsilon}). In order to keep the norm in (\ref{bounds_conv_omega_F}) bounded w.r.t $\epsilon$ near
0, we make the assumption that $n_{F} + \Gamma(d_{D} - \delta_{D}) - \gamma d_{D} \geq 0$ which again 
results from (\ref{constraint_Upsilon}).

Now, we focus on the second piece of $\Upsilon(\tau,m,\epsilon)$. Namely, using Lemma 5, we obtain a
constant $B_{1}>0$ (depending on $d_{D},\delta_{D},\nu$) such that
\begin{multline}
|| \frac{\epsilon^{n_F}}{\epsilon^{\gamma d_{D}}} \frac{ \tau^{d_{D}-1-\delta_{D}} }{R_{D}(im)}
\int_{L_{\theta_{F}}} \omega_{F}(u,m) du ||_{(\nu,\beta,\mu,\Gamma,\epsilon)}
\leq B_{1}| \int_{L_{\theta_{F}}} e^{-K_{F}u} \frac{F_{1}(u)}{F_{2}(u)} du|\\
\times
\frac{ ||C_{F}(m)||_{(\beta,\mu)} }{ \mathrm{inf}_{m \in \mathbb{R}} |R_{D}(im)| }
\frac{|\epsilon|^{n_F}}{|\epsilon|^{\gamma d_{D}}} |\epsilon|^{\Gamma( d_{D} - 1 - \delta_{D})}
\label{bounds_tau_times_integral_omega_F}
\end{multline}
when $d_{D}-1-\delta_{D} \geq 0$ which is part of (\ref{constraint_Upsilon}). Besides, we ask the norm in
(\ref{bounds_tau_times_integral_omega_F}) to be bounded w.r.t $\epsilon$ in the vicinity of the origin which turns
out to be an effect of (\ref{constraint_Upsilon}).\medskip

In the forthcoming proposition, we display suitable conditions under which the main convolution equation
(\ref{main_conv_eq_W}) possesses a unique solution rooted in the Banach space
$E_{(\nu,\beta,\mu,\Gamma,\epsilon)}^{d}$ described in Subsection 4.2, for a convenient choice of its parameters
$\nu,\beta,\mu,\Gamma$ given just above.

\begin{prop} We accredit that the next further constraints hold
\begin{equation}
d_{D} - k_{l} - 1 \geq 0 \ \ , \ \ \delta_{D} \geq 0 \ \ , \ \ \delta_{D} \leq d_{D} - k_{l} \ \ , \ \
m_{l} + \gamma_{0} + (\Gamma - \gamma)(d_{D} - k_{l}) - \Gamma \delta_{D} \geq 0
\label{first_cond_main_conv_eq_W}
\end{equation}
for all $1 \leq l \leq q$,
\begin{equation}
d_{D} \geq 1 \ \ , \ \ \delta_{D} \geq 0 \ \ , \ \ \delta_{D} \leq d_{D} \ \ , \ \
m_{0} + \gamma_{0} + (\Gamma - \gamma)d_{D} - \Gamma \delta_{D} \geq 0 \label{second_cond_main_conv_eq_W},
\end{equation}
\begin{equation}
\delta_{D} \leq d_{D} - h_{l} \ \ , \ \ \delta_{D} \geq 0 \ \ , \ \
\mu_{l} + 2 \gamma_{0} + (\Gamma - \gamma)(d_{D} - h_{l}) - \Gamma(\delta_{D} - 1) \geq 0,
\label{third_cond_main_conv_eq_W}
\end{equation}
for all $0 \leq l \leq M$,
\begin{equation}
d_{D} - b_{j} - 1 \geq \delta_{D} \ \ , \ \ n_{j} - \gamma(d_{D} - b_{j}) +
\Gamma(d_{D} - b{j} - 1 - \delta_{D}) \geq 0 \label{4_cond_main_conv_eq_W} 
\end{equation}
for all $0 \leq j \leq Q$,
\begin{equation}
\delta_{D} \leq d_{D} - d_{l} + \delta_{l} \ \ , \ \ \delta_{D} \geq \delta_{l} \ \ , \ \
\Delta_{l} + \gamma_{0} + (\Gamma - \gamma)(d_{D} - d_{l} + \delta_{l}) - \Gamma \delta_{D} \geq 0
\label{5_cond_main_conv_eq_W}
\end{equation}
as long as $1 \leq l \leq D-1$.

Then, there exist two constants $\varpi_{1}>0$ and $\zeta_{1}>0$ small enough, such that if
\begin{equation}
|a_{i}| \leq \zeta_{1} \ \ , \ \ |c_{j}| \leq \zeta_{1} \ \ , \ \
||B_{k}(m)||_{(\beta,\mu)} \leq \zeta_{1} \ \ , \ \ ||C_{F}(m)||_{(\beta,\mu)} \leq \zeta_{1} \ \ , \ \ 
\sup_{m \in \mathbb{R}} \frac{|R_{l}(im)|}{|R_{D}(im)|} \leq \zeta_{1},
\label{small_coeff_a_c_B_CF_Rl_RD}
\end{equation}
for $0 \leq i \leq q$, $0 \leq j \leq M$, $0 \leq k \leq Q$ and $1 \leq l \leq D-1$,
then, the equation (\ref{main_conv_eq_W}) has a unique solution
$W^{d}(\tau,m,\epsilon)$ stemming from the Banach space $E_{(\nu,\beta,\mu,\Gamma,\epsilon)}^{d}$ which is
governed by the bounds
\begin{equation}
||W^{d}(\tau,m,\epsilon)||_{(\nu,\beta,\mu,\Gamma,\epsilon)} \leq \varpi_{1}
\end{equation}
for all $\epsilon \in D(0,\epsilon_{0}) \setminus \{ 0 \}$, for any directions $d \in \mathbb{R}$ taken in such a 
manner that the sector $U_{d}$ fulfills the constraint proposed at the beginning of this subsection.
\end{prop}
\begin{proof} We depart from a lemma that aims attention at a shrinking map acting on the Banach spaces quoted above
and downsizes our main convolution problem to the existence and unicity of a fixed point for this map.

\begin{lemma} Taking for granted the constraints (\ref{first_cond_main_conv_eq_W}),
(\ref{second_cond_main_conv_eq_W}), (\ref{third_cond_main_conv_eq_W}), (\ref{4_cond_main_conv_eq_W}),
(\ref{5_cond_main_conv_eq_W}) presented above, one can adjust a constant $\varpi_{1}>0$
small enough and a constant $\zeta_{1}>0$ taken in a way that if the smallness condition
(\ref{small_coeff_a_c_B_CF_Rl_RD}) hold, then for all $\epsilon \in D(0,\epsilon_{0}) \setminus \{ 0 \}$, the map
$\mathcal{G}_{\epsilon}$ prescribed as
\begin{multline}
\mathcal{G}_{\epsilon}(w(\tau,m)) := \sum_{l=1}^{q} \frac{a_{l}}{(d_{D} - k_{l} - 1)!}
\frac{\epsilon^{m_{l} + \gamma_{0}}}{\epsilon^{\gamma(d_{D} - k_{l})}}
\frac{Q(im)}{R_{D}(im) (-\tau)^{\delta_{D}} } \int_{0}^{\tau}
(\tau - s)^{d_{D}-k_{l}-1}w(s,m) ds \\
+ \frac{a_{0}}{(d_{D}-1)!}
\frac{\epsilon^{m_{0} + \gamma_{0}}}{\epsilon^{\gamma d_{D}}}
\frac{Q(im)}{R_{D}(im) (-\tau)^{\delta_D}} \int_{0}^{\tau}
(\tau - s)^{d_{D}-1}w(s,m) ds\\
+ \sum_{l=0}^{M}
\frac{c_{l}}{(d_{D}-h_{l}-1)!} \frac{\epsilon^{\mu_{l} + 2\gamma_{0}}}{\epsilon^{\gamma(d_{D} - h_{l})}}\\
\times \frac{1}{R_{D}(im)(-\tau)^{\delta_D}}
\int_{0}^{\tau} (\tau-s)^{d_{D}-h_{l}-1} \frac{1}{(2\pi)^{1/2}}
\int_{-\infty}^{+\infty} \int_{0}^{s} Q_{1}(i(m-m_{1}))w(s-s',m-m_{1})\\
\times
Q_{2}(im_{1})w(s',m_{1}) ds' dm_{1} ds\\
- \sum_{j=0}^{Q}
\frac{1}{(d_{D}-b_{j}-1)!} \frac{\epsilon^{n_j}}{\epsilon^{\gamma(d_{D}-b_{j})}} \frac{B_{j}(m)}{R_{D}(im)}
\frac{\tau^{d_{D}-b_{j}-1}}{(-\tau)^{\delta_{D}}}
-\frac{\Upsilon(\tau,m,\epsilon)}{R_{D}(im)(-\tau)^{\delta_D}}\\
-\sum_{l=1}^{D-1} \frac{1}{(d_{D}-d_{l}-1)!}
\frac{\epsilon^{\Delta_{l} + \gamma_{0}}}{\epsilon^{\gamma( d_{D} - d_{l} + \delta_{l})}} 
\frac{R_{l}(im)}{R_{D}(im)}\frac{1}{(-\tau)^{\delta_D}}\\
\times 
\int_{0}^{\tau} (\tau-s)^{d_{D}-d_{l}-1} (-s)^{\delta_{l}}w(s,m) ds
\end{multline}
undergo the next properties.\\
{\bf i)} The next inclusion
\begin{equation}
\mathcal{G}_{\epsilon}( \bar{B}(0,\varpi_{1}) ) \subset \bar{B}(0,\varpi_{1}) \label{G_epsilon_inclusion} 
\end{equation}
takes place, where $\bar{B}(0,\varpi_{1})$ stands for the closed ball of radius $\varpi_{1}$ centered at 0
in the space $E_{(\nu,\beta,\mu,\Gamma,\epsilon)}^{d}$, for any $\epsilon \in D(0,\epsilon_{0}) \setminus \{ 0 \}$.\\
{\bf ii)} The ensuing shrinking constraint
\begin{equation}
|| \mathcal{G}_{\epsilon}(w_{1}) - \mathcal{G}_{\epsilon}(w_{2}) ||_{(\nu,\beta,\mu,\Gamma,\epsilon)}
\leq \frac{1}{2} ||w_{1} - w_{2}||_{(\nu,\beta,\mu,\Gamma,\epsilon)} \label{G_epsilon_shrinking}
\end{equation}
holds for all $w_{1},w_{2} \in \bar{B}(0,\varpi_{1})$, all $\epsilon \in D(0,\epsilon_{0}) \setminus \{ 0 \}$.
\end{lemma}
\begin{proof} Foremost, we focus on the first property (\ref{G_epsilon_inclusion}). Namely, let
$\epsilon \in D(0,\epsilon_{0}) \setminus \{ 0 \}$ and consider $w(\tau,m) \in
E_{(\nu,\beta,\mu,\Gamma,\epsilon)}^{d}$. We take $\varpi_{1}>0$ with
$||w(\tau,m)||_{(\nu,\beta,\mu,\Gamma,\epsilon)} \leq \varpi_{1}$.\\
Bearing in mind Lemma 6, we get a constant $B_{2}>0$ (depending on
$\nu,\delta_{D},d_{D},k_{l}$, for $1 \leq l \leq q$) with
\begin{multline}
|| \frac{\epsilon^{m_{l} + \gamma_{0}}}{\epsilon^{\gamma(d_{D} - k_{l})}}
\frac{Q(im)}{R_{D}(im) \tau^{\delta_{D}} } \int_{0}^{\tau}
(\tau - s)^{d_{D}-k_{l}-1}w(s,m) ds ||_{(\nu,\beta,\mu,\Gamma,\epsilon)}\\
\leq B_{2} \sup_{m \in \mathbb{R}} \frac{|Q(im)|}{|R_{D}(im)|}
\frac{|\epsilon|^{m_{l} + \gamma_{0}}}{|\epsilon|^{\gamma(d_{D} - k_{l})}} |\epsilon|^{\Gamma( d_{D} - k_{l})
- \Gamma \delta_{D} } \varpi_{1} \label{first_bound_G_inclusion}
\end{multline}
for all $1 \leq l \leq q$, submitted to (\ref{first_cond_main_conv_eq_W}). Likewise, we get a constant
$B_{2}>0$ (depending on $\nu,\delta_{D},d_{D}$) with
\begin{multline}
|| \frac{\epsilon^{m_{0} + \gamma_{0}}}{\epsilon^{\gamma d_{D}}}
\frac{Q(im)}{R_{D}(im) \tau^{\delta_{D}} } \int_{0}^{\tau}
(\tau - s)^{d_{D}-1}w(s,m) ds ||_{(\nu,\beta,\mu,\Gamma,\epsilon)}\\
\leq B_{2} \sup_{m \in \mathbb{R}} \frac{|Q(im)|}{|R_{D}(im)|}
\frac{|\epsilon|^{m_{0} + \gamma_{0}}}{|\epsilon|^{\gamma d_{D} }} |\epsilon|^{\Gamma d_{D}
- \Gamma \delta_{D} } \varpi_{1} \label{second_bound_G_inclusion}
\end{multline}
counting on (\ref{second_cond_main_conv_eq_W}). Now, we put
$$ h(\tau,m) = \frac{1}{R_{D}(im)} \int_{0}^{\tau}
\int_{-\infty}^{+\infty} Q_{1}(i(m-m_{1}))w( \tau - s',m-m_{1})Q_{2}(im_{1})w(s',m_{1}) ds' dm_{1}. $$
From Lemma 7, under the constraint (\ref{constraints_degree_coeff_Q_Rl}),
we get a constant $B_{3}>0$ (depending on $\mu,Q_{1},Q_{2},R_{D}$) such that
$$ ||h(\tau,m)||_{(\nu,\beta,\mu,\Gamma,\epsilon)} \leq B_{3}|\epsilon|^{\Gamma}
||w(\tau,m)||_{(\nu,\beta,\mu,\Gamma,\epsilon)}^{2}
$$
In accordance with Lemma 6, we deduce a constant $B_{2}>0$ (depending on $\nu,\delta_{D},d_{D},h_{l}$ for
$0 \leq l \leq M$) with
\begin{multline}
||  \frac{\epsilon^{\mu_{l} + 2\gamma_{0}}}{\epsilon^{\gamma(d_{D} - h_{l})}}
\frac{1}{\tau^{\delta_{D}}} \int_{0}^{\tau} (\tau - s)^{d_{D} - h_{l} - 1} h(s,m) ds
||_{(\nu,\beta,\mu,\Gamma,\epsilon)}\\
\leq B_{2} \frac{|\epsilon|^{\mu_{l} + 2\gamma_{0}}}{|\epsilon|^{\gamma(d_{D} - h_{l})}}
|\epsilon|^{\Gamma(d_{D} - h_{l}) - \Gamma \delta_{D}}
||h(\tau,m)||_{(\nu,\beta,\mu,\Gamma,\epsilon)} \leq
B_{2}B_{3}\frac{|\epsilon|^{\mu_{l} + 2\gamma_{0}}}{|\epsilon|^{\gamma(d_{D} - h_{l})}}
|\epsilon|^{\Gamma(d_{D} - h_{l}) - \Gamma \delta_{D}} |\epsilon|^{\Gamma} \varpi_{1}^{2}
\label{3_bound_G_inclusion}
\end{multline}
in agreement with (\ref{third_cond_main_conv_eq_W}). Hereafter, we concentrate on the inhomogeneous terms. Namely,
according to Lemma 5, we get a constant $B_{1}>0$ (depending on $\nu,d_{D},\delta_{D},b_{j}$ for
$0 \leq j \leq Q$) such that
\begin{multline}
|| \frac{\epsilon^{n_j}}{\epsilon^{\gamma(d_{D}-b_{j})}} \frac{B_{j}(m)}{R_{D}(im)}
\frac{\tau^{d_{D}-b_{j}-1}}{\tau^{\delta_{D}}} ||_{(\nu,\beta,\mu,\Gamma,\epsilon)}\\
\leq B_{1} \sup_{m \in \mathbb{R}} |\frac{1}{R_{D}(im)}|
||B_{j}(m)||_{(\beta,\mu)} \frac{|\epsilon|^{n_j}}{|\epsilon|^{\gamma(d_{D}-b_{j})}}
|\epsilon|^{\Gamma(d_{D} - b_{j} - 1 - \delta_{D})} \label{4_bound_G_inclusion}
\end{multline}
taking notice that (\ref{4_cond_main_conv_eq_W}) happens. On the other hand, under the assumptions 
(\ref{constraint_Upsilon}), we have seen above that there exists a constant $B_{\Upsilon}>0$ (depending
on $\omega_{F},\theta_{F},\nu,\Gamma,\gamma,n_{F},d_{D},\delta_{D},R_{D}(im)$) for which
\begin{equation}
||\frac{\Upsilon(\tau,m,\epsilon)}{R_{D}(im)\tau^{\delta_D}}||_{(\nu,\beta,\mu,\Gamma,\epsilon)} \leq
B_{\Upsilon} ||C_{F}(m)||_{(\beta,\mu)} \label{4_bis_bound_G_inclusion}
\end{equation}
for all $\epsilon \in D(0,\epsilon_{0}) \setminus \{ 0 \}$. At last, we provide estimates for the remaining
convolution terms. Specifically, Lemma 6 yields a constant $B_{2}>0$ (depending on
$\nu,\delta_{D},d_{D},d_{l},\delta_{l}$ for $1 \leq l \leq D-1$) such that
\begin{multline}
||\frac{\epsilon^{\Delta_{l} + \gamma_{0}}}{\epsilon^{\gamma( d_{D} - d_{l} + \delta_{l})}} 
\frac{R_{l}(im)}{R_{D}(im)}\frac{1}{\tau^{\delta_D}}
\int_{0}^{\tau} (\tau-s)^{d_{D}-d_{l}-1} s^{\delta_{l}}w(s,m)
ds||_{(\nu,\beta,\mu,\Gamma,\epsilon)} \\
\leq B_{2}
\sup_{m \in \mathbb{R}} \frac{|R_{l}(im)|}{|R_{D}(im)|}
\frac{|\epsilon|^{\Delta_{l} + \gamma_{0}}}{|\epsilon|^{\gamma( d_{D} - d_{l} + \delta_{l})}}
|\epsilon|^{\Gamma( d_{D} - d_{l} + \delta_{l} ) - \Gamma \delta_{D} }\varpi_{1}
\label{5_bound_G_inclusion}
\end{multline}
for $1 \leq l \leq D-1$, under the requirement that (\ref{5_cond_main_conv_eq_W}) holds. Finally, we select
both $\varpi_{1}>0$ and $\zeta_{1}>0$ satisfying (\ref{small_coeff_a_c_B_CF_Rl_RD}) in such a way that
\begin{multline}
\sum_{l=1}^{q} \frac{|a_{l}|}{(d_{D} - k_{l} - 1)!}B_{2} \sup_{m \in \mathbb{R}} \frac{|Q(im)|}{|R_{D}(im)|}
\frac{\epsilon_{0}^{m_{l} + \gamma_{0}}}{\epsilon_{0}^{\gamma(d_{D} - k_{l})}} \epsilon_{0}^{\Gamma( d_{D} - k_{l})
- \Gamma \delta_{D} } \varpi_{1} \\
+ \frac{|a_{0}|}{(d_{D}-1)!}B_{2} \sup_{m \in \mathbb{R}}
\frac{|Q(im)|}{|R_{D}(im)|}
\frac{\epsilon_{0}^{m_{0} + \gamma_{0}}}{\epsilon_{0}^{\gamma d_{D} }} \epsilon_{0}^{\Gamma d_{D}
- \Gamma \delta_{D} } \varpi_{1} \\
+ \sum_{l=0}^{M}
\frac{|c_{l}|}{(d_{D}-h_{l}-1)!(2 \pi)^{1/2}}B_{2}B_{3}
\frac{\epsilon_{0}^{\mu_{l} + 2\gamma_{0}}}{\epsilon_{0}^{\gamma(d_{D} - h_{l})}}
\epsilon_{0}^{\Gamma(d_{D} - h_{l}) - \Gamma \delta_{D}} \epsilon_{0}^{\Gamma} \varpi_{1}^{2}\\
+ \sum_{j=0}^{Q}
\frac{1}{(d_{D}-b_{j}-1)!} B_{1} \sup_{m \in \mathbb{R}} |\frac{1}{R_{D}(im)}|
||B_{j}(m)||_{(\beta,\mu)} \frac{\epsilon_{0}^{n_j}}{\epsilon_{0}^{\gamma(d_{D}-b_{j})}}
\epsilon_{0}^{\Gamma(d_{D} - b_{j} - 1 - \delta_{D})} \\
+ B_{\Upsilon} ||C_{F}(m)||_{(\beta,\mu)} + \sum_{l=1}^{D-1} \frac{1}{(d_{D}-d_{l}-1)!}
B_{2} \sup_{m \in \mathbb{R}} \frac{|R_{l}(im)|}{|R_{D}(im)|}\\
\times \frac{\epsilon_{0}^{\Delta_{l} + \gamma_{0}}}{\epsilon_{0}^{\gamma( d_{D} - d_{l} + \delta_{l})}}
\epsilon_{0}^{\Gamma( d_{D} - d_{l} + \delta_{l} ) - \Gamma \delta_{D} }\varpi_{1} \leq \varpi_{1}.
\label{G_epsilon_inclusion_condition}
\end{multline}
From the very definition of $\mathcal{G}_{\epsilon}$, by compiling the bounds (\ref{first_bound_G_inclusion}),
(\ref{second_bound_G_inclusion}), (\ref{3_bound_G_inclusion}), (\ref{4_bound_G_inclusion}),
(\ref{4_bis_bound_G_inclusion}), (\ref{5_bound_G_inclusion}), we recover the inclusion announced in
(\ref{G_epsilon_inclusion}).\medskip

In the next part of the proof, we target the shrinking restriction (\ref{G_epsilon_shrinking}). Namely, let us
choose $w_{1}(\tau,m)$ and $w_{2}(\tau,m)$ in the space $E_{(\nu,\beta,\mu,\Gamma,\epsilon)}^{d}$ inside
the ball $\bar{B}(0,\varpi_{1})$.

Ahead in position, according to the bounds (\ref{first_bound_G_inclusion}), (\ref{second_bound_G_inclusion})
and (\ref{5_bound_G_inclusion}), we get a constant $B_{2}>0$ (depending on $\nu,\delta_{D},d_{D},k_{l}$
for $1 \leq l \leq q$ and $d_{l},\delta_{l}$ for $1 \leq l \leq D-1$) for which
\begin{multline}
|| \frac{\epsilon^{m_{l} + \gamma_{0}}}{\epsilon^{\gamma(d_{D} - k_{l})}}
\frac{Q(im)}{R_{D}(im) \tau^{\delta_{D}} } \int_{0}^{\tau}
(\tau - s)^{d_{D}-k_{l}-1}(w_{1}(s,m) - w_{2}(s,m)) ds ||_{(\nu,\beta,\mu,\Gamma,\epsilon)}\\
\leq B_{2} \sup_{m \in \mathbb{R}} \frac{|Q(im)|}{|R_{D}(im)|}
\frac{|\epsilon|^{m_{l} + \gamma_{0}}}{|\epsilon|^{\gamma(d_{D} - k_{l})}} |\epsilon|^{\Gamma( d_{D} - k_{l})
- \Gamma \delta_{D} } ||w_{1}(\tau,m) - w_{2}(\tau,m)||_{(\nu,\beta,\mu,\Gamma,\epsilon)}
\label{first_bound_G_shrink}
\end{multline}
holds for all $1 \leq l \leq q$, together with
\begin{multline}
|| \frac{\epsilon^{m_{0} + \gamma_{0}}}{\epsilon^{\gamma d_{D}}}
\frac{Q(im)}{R_{D}(im) \tau^{\delta_{D}} } \int_{0}^{\tau}
(\tau - s)^{d_{D}-1}(w_{1}(s,m) - w_{2}(s,m)) ds ||_{(\nu,\beta,\mu,\Gamma,\epsilon)}\\
\leq B_{2} \sup_{m \in \mathbb{R}} \frac{|Q(im)|}{|R_{D}(im)|}
\frac{|\epsilon|^{m_{0} + \gamma_{0}}}{|\epsilon|^{\gamma d_{D} }} |\epsilon|^{\Gamma d_{D}
- \Gamma \delta_{D} } ||w_{1}(\tau,m) - w_{2}(\tau,m)||_{(\nu,\beta,\mu,\Gamma,\epsilon)} \label{second_bound_G_shrink}
\end{multline}
and
\begin{multline}
||\frac{\epsilon^{\Delta_{l} + \gamma_{0}}}{\epsilon^{\gamma( d_{D} - d_{l} + \delta_{l})}} 
\frac{R_{l}(im)}{R_{D}(im)}\frac{1}{\tau^{\delta_D}}
\int_{0}^{\tau} (\tau-s)^{d_{D}-d_{l}-1} s^{\delta_{l}}(w_{1}(s,m) - w_{2}(s,m))
ds||_{(\nu,\beta,\mu,\Gamma,\epsilon)} \\
\leq B_{2}
\sup_{m \in \mathbb{R}} \frac{|R_{l}(im)|}{|R_{D}(im)|}
\frac{|\epsilon|^{\Delta_{l} + \gamma_{0}}}{|\epsilon|^{\gamma( d_{D} - d_{l} + \delta_{l})}}
|\epsilon|^{\Gamma( d_{D} - d_{l} + \delta_{l} ) - \Gamma \delta_{D} }
||w_{1}(\tau,m) - w_{2}(\tau,m)||_{(\nu,\beta,\mu,\Gamma,\epsilon)}
\label{3_bound_G_shrink}
\end{multline}
for $1 \leq l \leq D-1$. We concentrate now on the nonlinear part of $\mathcal{G}_{\epsilon}$. In a similar way
as we have proceed for the map $\mathcal{H}_{\epsilon}$ in the proof of Lemma 3, we may write as a preparation
the next identity
\begin{multline}
Q_{1}(i(m-m_{1}))w_{1}(s-s',m-m_{1})Q_{2}(im_{1})w_{1}(s',m_{1})\\
- Q_{1}(i(m-m_{1}))w_{2}(s-s',m-m_{1})Q_{2}(im_{1})w_{2}(s',m_{1})\\
= Q_{1}(i(m-m_{1}))
\left( w_{1}(s-s',m-m_{1}) - w_{2}(s-s',m-m_{1}) \right)
Q_{2}(im_{1})w_{1}(s',m_{1}) \\
+ Q_{1}(i(m-m_{1}))w_{2}(s-s',m-m_{1})
Q_{2}(im_{1}) \left( w_{1}(s',m_{1}) - w_{2}(s',m_{1}) \right)
\label{factor_Q1w1Q2w1_minus_Q1w2Q2w2_for_G_epsilon}
\end{multline}
For $j=1,2$, we assign
$$ h_{j}(\tau,m) = \frac{1}{R_{D}(im)} \int_{0}^{\tau}
\int_{-\infty}^{+\infty} Q_{1}(i(m-m_{1}))w_{j}( \tau - s',m-m_{1})Q_{2}(im_{1})w_{j}(s',m_{1}) ds' dm_{1}. $$
Keeping in view the latter factorization (\ref{factor_Q1w1Q2w1_minus_Q1w2Q2w2_for_G_epsilon}), accordingly to
Lemma 7, under the assumption (\ref{constraints_degree_coeff_Q_Rl}), we get a constant $B_{3}>0$
(depending on $\mu,Q_{1},Q_{2},R_{D}$) with
\begin{multline*}
||h_{1}(\tau,m) - h_{2}(\tau,m)||_{(\nu,\beta,\mu,\Gamma,\epsilon)} \leq
B_{3}|\epsilon|^{\Gamma}||w_{1}(\tau,m) - w_{2}(\tau,m)||_{(\nu,\beta,\mu,\Gamma,\epsilon)}\\
\times
\left( ||w_{1}(\tau,m)||_{(\nu,\beta,\mu,\Gamma,\epsilon)} +
||w_{2}(\tau,m)||_{(\nu,\beta,\mu,\Gamma,\epsilon)} \right).
\end{multline*}
As a result, with the help of the first inequality of (\ref{3_bound_G_inclusion}), we can select a constant
$B_{2}>0$ (depending on $\nu,\delta_{D},d_{D},h_{l}$ for $0 \leq l \leq M$) such that
\begin{multline}
||  \frac{\epsilon^{\mu_{l} + 2\gamma_{0}}}{\epsilon^{\gamma(d_{D} - h_{l})}}
\frac{1}{\tau^{\delta_{D}}} \int_{0}^{\tau} (\tau - s)^{d_{D} - h_{l} - 1} (h_{1}(s,m) - h_{2}(s,m)) ds
||_{(\nu,\beta,\mu,\Gamma,\epsilon)}\\
\leq B_{2} \frac{|\epsilon|^{\mu_{l} + 2\gamma_{0}}}{|\epsilon|^{\gamma(d_{D} - h_{l})}}
|\epsilon|^{\Gamma(d_{D} - h_{l}) - \Gamma \delta_{D}}
||h_{1}(\tau,m) - h_{2}(\tau,m)||_{(\nu,\beta,\mu,\Gamma,\epsilon)} \\
\leq
B_{2}B_{3}\frac{|\epsilon|^{\mu_{l} + 2\gamma_{0}}}{|\epsilon|^{\gamma(d_{D} - h_{l})}}
|\epsilon|^{\Gamma(d_{D} - h_{l}) - \Gamma \delta_{D}} |\epsilon|^{\Gamma}
||w_{1}(\tau,m) - w_{2}(\tau,m)||_{(\nu,\beta,\mu,\Gamma,\epsilon)}\\
\times
\left( ||w_{1}(\tau,m)||_{(\nu,\beta,\mu,\Gamma,\epsilon)} +
||w_{2}(\tau,m)||_{(\nu,\beta,\mu,\Gamma,\epsilon)} \right).
\label{4_bound_G_shrink}
\end{multline}
We adjust both $\varpi_{1}>0$ and $\zeta_{1}>0$ with (\ref{small_coeff_a_c_B_CF_Rl_RD}) in a manner that
\begin{multline}
\sum_{l=1}^{q} \frac{|a_{l}|}{(d_{D} - k_{l} - 1)!}B_{2} \sup_{m \in \mathbb{R}} \frac{|Q(im)|}{|R_{D}(im)|}
\frac{\epsilon_{0}^{m_{l} + \gamma_{0}}}{\epsilon_{0}^{\gamma(d_{D} - k_{l})}} \epsilon_{0}^{\Gamma( d_{D} - k_{l})
- \Gamma \delta_{D} } \\
+ \frac{|a_{0}|}{(d_{D}-1)!}B_{2} \sup_{m \in \mathbb{R}}
\frac{|Q(im)|}{|R_{D}(im)|}
\frac{\epsilon_{0}^{m_{0} + \gamma_{0}}}{\epsilon_{0}^{\gamma d_{D} }} \epsilon_{0}^{\Gamma d_{D}
- \Gamma \delta_{D} } \\
+ \sum_{l=0}^{M}
\frac{|c_{l}|}{(d_{D}-h_{l}-1)!(2 \pi)^{1/2}}B_{2}B_{3}
\frac{\epsilon_{0}^{\mu_{l} + 2\gamma_{0}}}{\epsilon_{0}^{\gamma(d_{D} - h_{l})}}
\epsilon_{0}^{\Gamma(d_{D} - h_{l}) - \Gamma \delta_{D}} \epsilon_{0}^{\Gamma} 2 \varpi_{1}\\
+ \sum_{l=1}^{D-1} \frac{1}{(d_{D}-d_{l}-1)!}
B_{2} \sup_{m \in \mathbb{R}} \frac{|R_{l}(im)|}{|R_{D}(im)|}\\
\times \frac{\epsilon_{0}^{\Delta_{l} + \gamma_{0}}}{\epsilon_{0}^{\gamma( d_{D} - d_{l} + \delta_{l})}}
\epsilon_{0}^{\Gamma( d_{D} - d_{l} + \delta_{l} ) - \Gamma \delta_{D} } \leq 1/2.
\label{G_epsilon_shrink_condition}
\end{multline}
By grouping the above estimates (\ref{first_bound_G_shrink}), (\ref{second_bound_G_shrink}),
(\ref{3_bound_G_shrink}) and (\ref{4_bound_G_shrink}), we are led to the shrinking constraints
(\ref{G_epsilon_shrinking}).

In order to complete the proof, let us allow the two conditions
(\ref{G_epsilon_inclusion_condition}) and (\ref{G_epsilon_shrink_condition}) to mutually occur for well selected
$\varpi_{1}>0$ and $\zeta_{1}>0$. Then, both
(\ref{G_epsilon_inclusion}) and (\ref{G_epsilon_shrinking}) are conjointly verified.
\end{proof}
Take the ball $\bar{B}(0,\varpi_{1}) \subset E_{(\nu,\beta,\mu,\Gamma,\epsilon)}^{d}$ just
built above in Lemma 8 which furnishes a complete
metric space endowed with the norm $||.||_{(\nu,\beta,\mu,\Gamma,\epsilon)}$. From the lemma above, we get
that $\mathcal{G}_{\epsilon}$ is a
contractive map from $\bar{B}(0,\varpi_{1})$ into itself. Due to the classical contractive mapping theorem, we
deduce that
the map $\mathcal{G}_{\epsilon}$ has a unique fixed point denoted $W^{d}(\tau,m,\epsilon)$
in the ball $\bar{B}(0,\varpi_{1})$, meaning that
\begin{equation}
\mathcal{G}_{\epsilon}(W^{d}(\tau,m,\epsilon)) = W^{d}(\tau,m,\epsilon)
\label{G_epsilon_fix_point_eq}
\end{equation}
for a unique solution $W^{d}(\tau,m,\epsilon) \in E_{(\nu,\beta,\mu,\Gamma,\epsilon)}^{d}$ such
that  
$||W^{d}(\tau,m,\epsilon)||_{(\nu,\beta,\mu,\Gamma,\epsilon)} \leq \varpi_{1}$,
for all $\epsilon \in D(0,\epsilon_{0}) \setminus \{ 0 \}$. Moreover, the function
$W^{d}(\tau,m,\epsilon)$ depends holomorphically on $\epsilon$ in
$D(0,\epsilon_{0}) \setminus \{ 0 \}$.

If one sets apart the term $(-\tau)^{\delta_{D}}R_{D}(im)W(\tau,m,\epsilon)$ in the right handside of
(\ref{main_conv_eq_W}), we notice that (\ref{main_conv_eq_W}) can be scaled down to the equation
(\ref{G_epsilon_fix_point_eq}) above using a mere division by $(-\tau)^{\delta_D}R_{D}(im)$. As a result,
the unique fixed point $W^{d}(\tau,m,\epsilon)$ of $\mathcal{G}_{\epsilon}$ in
$\bar{B}(0,\varpi_{1})$ precisely solves the problem (\ref{main_conv_eq_W}). The proposition 7 follows.
\end{proof}

\subsection{Analytic solutions to the main problem on large $\epsilon-$depending sectorial domains in time}

We go back to the speculative solutions to the main equation (\ref{main_PDE_u}) displayed in Section 4.3 under the
new light shed on the related nonlinear convolution equation (\ref{main_conv_eq_W}) in Section 4.4.
We first provide the definition of the set of $\epsilon-$depending associated sector and directions to a
good covering.

\begin{defin} Let $\iota \geq 2$ be an integer. For all $0 \leq j \leq \iota-1$, we consider an open sector
$\mathcal{E}_{j}^{\infty}$ centered at 0, with radius $\epsilon_{0}^{\infty}>0$ and opening
$\frac{\pi}{\gamma} + \xi_{j}<2\pi$ for some real number $\xi_{j}>0$. We assume that the family
$\{ \mathcal{E}_{j}^{\infty} \}_{0 \leq j \leq \iota-1}$ forms a good covering in $\mathbb{C}^{\ast}$ with aperture
$\pi/\gamma$. For all $0 \leq j \leq \iota-1$, let $\mathfrak{u}_{j}$ be a real number belonging to
$(-\pi/2,\pi/2)$ sorted in a way
that there exists an unbounded sector $U_{\mathfrak{u}_{j}}$ centered at 0 with bisecting direction
$\mathfrak{u}_{j}$ and appropriate aperture not containing any root of the
polynomial $F_{2}(\tau)$ introduced in the formula (\ref{defin_omega_F}). Let $\nu>0$ be fixed as above at the
beginning of Section 4.4.

We assume that one can select a real number $\delta_{1}^{\infty}>0$ and
$\Delta_{\nu,\delta_{1}^{\infty}} > \nu/ \delta_{1}^{\infty}$ such that for all $0 \leq j \leq \iota-1$, all
$\epsilon \in \mathcal{E}_{j}^{\infty}$, all $t \in \mathcal{T}_{\epsilon}^{\infty}$, where
$$ \mathcal{T}_{\epsilon}^{\infty} = \{ t \in \mathbb{C} /
|t| > \Delta_{\nu,\delta_{1}^{\infty}}|\epsilon|^{\gamma - \Gamma} \ \ , \ \
\alpha_{\infty} < \mathrm{arg}(t) < \beta_{\infty} \} $$
there exists some direction $\mathfrak{u}_{j}^{\Delta} \in \mathbb{R}$ (that may depend on
$\epsilon$ and $t$) with $\exp(i \mathfrak{u}_{j}^{\Delta} ) \in U_{\mathfrak{u}_{j}}$ that satisfies the next
requirement
$$ \mathfrak{u}_{j}^{\Delta} + \mathrm{arg}(\frac{t}{\epsilon^{\gamma}}) \in (-\frac{\pi}{2},\frac{\pi}{2}) \ \ , \ \
\cos( \mathfrak{u}_{j}^{\Delta} + \mathrm{arg}(\frac{t}{\epsilon^{\gamma}}) ) \geq \delta_{1}^{\infty},$$
for suitable fixed angles $\alpha_{\infty} < \beta_{\infty}$.

If the above constraints hold, we claim that the family of $\epsilon-$depending sector and directions
$\{ \mathcal{T}_{\epsilon}^{\infty}, \{ \mathfrak{u}_{j} \}_{0 \leq j \leq \iota-1} \}$ is associated to the
good covering $\{ \mathcal{E}_{j}^{\infty} \}_{0 \leq j \leq \iota-1}$.
\end{defin}

In the forthcoming second main outcome of this work, we construct a set of actual holomorphic solutions to the
principal equation (\ref{main_PDE_u}) which we name {\it outer solutions}. These solutions
are well defined on the sectors $\mathcal{E}_{j}^{\infty}$ of a good covering w.r.t $\epsilon$, on an
associated sector $\mathcal{T}_{\epsilon}^{\infty}$ w.r.t $t$ and on an horizontal strip $H_{\beta}$ w.r.t $z$.
Moreover, we can control the difference between any two consecutive solutions on the crossing sector
$\mathcal{E}_{j}^{\infty} \cap \mathcal{E}_{j+1}^{\infty}$ and confirm that it is exponentially flat of order
at most $\gamma$ w.r.t $\epsilon$.

\begin{theo} We focus on the singularly perturbed equation (\ref{main_PDE_u}) and we take for granted that all
the aforementioned constraints (\ref{constraints_degree_coeff_Q_Rl}), (\ref{constraint_m0_ml}), (\ref{defin_b_j}),
(\ref{defin_F_tzepsilon}), (\ref{defin_omega_F}), (\ref{constraints_d_D_di_kj_bk_hl}),
(\ref{constraint_DeltaD_gamma}), (\ref{constraint_Upsilon}), (\ref{first_cond_main_conv_eq_W}),
(\ref{second_cond_main_conv_eq_W}), (\ref{third_cond_main_conv_eq_W}),
(\ref{4_cond_main_conv_eq_W}), (\ref{5_cond_main_conv_eq_W}) hold. Besides, we choose a good covering
$\{ \mathcal{E}_{j}^{\infty} \}_{0 \leq j \leq \iota-1}$ with aperture $\frac{\pi}{\gamma}$ for which
an associated family of a sector $\mathcal{T}_{\epsilon}^{\infty}$ and directions
$\{ \mathfrak{u}_{j} \}_{0 \leq j \leq \iota-1}$ can be singled out.

Then, there exists a constant $\zeta_{1}>0$ for which we assume the restriction
(\ref{small_coeff_a_c_B_CF_Rl_RD}) to take place. As a result, for each $0 \leq j \leq \iota-1$, 
one can build up an actual solution $v^{\mathfrak{u}_{j}}(t,z,\epsilon)$ of
(\ref{main_PDE_u}), where the piece of forcing term $(t,z) \mapsto F^{\theta_{F}}(t,z,\epsilon)$
needs to be specified for $\theta_{F}=\mathfrak{u}_{j}$ and represents a bounded
holomorphic function denoted $(t,z) \mapsto F^{\mathfrak{u}_{j}}(t,z,\epsilon)$ w.r.t $t$ on
$\mathcal{T}_{\epsilon}^{\infty}$, w.r.t $z$ on a strip
$H_{\beta'}$, for any given $0 < \beta' < \beta$, when $\epsilon$ belongs to $\mathcal{E}_{j}^{\infty}$.

Moreover, for each $\epsilon \in \mathcal{E}_{j}^{\infty}$, the function
$(t,z) \mapsto v^{\mathfrak{u}_{j}}(t,z,\epsilon)$
is bounded and holomorphic on $\mathcal{T}_{\epsilon}^{\infty} \times H_{\beta'}$
for any given $0 < \beta' < \beta$, $0 \leq j \leq \iota-1$. Besides, for each
prescribed $t \in \mathcal{T}^{\infty}$, where
\begin{equation}
\mathcal{T}^{\infty} = \{ t \in \mathbb{C}^{\ast} / \alpha_{\infty} < \mathrm{arg}(t) < \beta_{\infty} \},
\label{defin_mathcalT_infty}
\end{equation}
the function
$(z,\epsilon) \mapsto \epsilon^{-\gamma_{0}}v^{\mathfrak{u}_{j}}(t,z,\epsilon)$ is bounded
holomorphic on $(\mathcal{E}_{j}^{\infty} \cap D(0,\sigma_{t})) \times H_{\beta'}$,
for any given $0 < \beta' < \beta$, $0 \leq j \leq \iota-1$ and suffers the next upper bounds: there
exist $K_{j},M_{j}>0$ (independent of $\epsilon$) such that
\begin{equation}
\sup_{z \in H_{\beta'}}
|\epsilon^{-\gamma_{0}}v^{\mathfrak{u}_{j+1}}(t,z,\epsilon) -
\epsilon^{-\gamma_{0}}v^{\mathfrak{u}_{j}}(t,z,\epsilon)| \leq K_{j}
\exp( - \frac{M_{j}|t|}{|\epsilon|^{\gamma}} ) \label{diff_v_uj_exp_small}
\end{equation}
for all $\epsilon \in \mathcal{E}_{j+1}^{\infty} \cap \mathcal{E}_{j}^{\infty} \cap D(0,\sigma_{t})$,
for $0 \leq j \leq \iota-1$ (where
by convention $v^{\mathfrak{u}_{\iota}} = v^{\mathfrak{u}_{0}}$), where
\begin{equation}
\sigma_{t} = (\frac{ \delta_{1}^{\infty} - \delta_{2}^{\infty} }{\nu})^{\frac{1}{\gamma - \Gamma}}
|t|^{\frac{1}{\gamma - \Gamma}} \label{defin_sigma_t}
\end{equation}
for some positive real number $\delta_{2}^{\infty}>0$ chosen in a way that
$\delta_{2}^{\infty} < \delta_{1}^{\infty}$ holds.
\end{theo}
\begin{proof} We select a good covering $\{ \mathcal{E}_{j}^{\infty} \}_{0 \leq j \leq \iota-1}$ in
$\mathbb{C}^{\ast}$ with aperture $\frac{\pi}{\gamma}$ and a family of sectors and directions
$\{ \mathcal{T}_{\epsilon}^{\infty} , \{ \mathfrak{u}_{j} \}_{0 \leq j \leq \iota-1} \}$ associated to this covering
according to Definition 7.

As a result of the estimates (\ref{4_bis_bound_G_inclusion}), we observe that the function
$\Upsilon(\tau,m,\epsilon)$ must be governed by the next bounds
\begin{equation}
|\Upsilon(\tau,m,\epsilon)| \leq B_{\Upsilon} ||C_{F}(m)||_{(\beta,\mu)}
(1 + |m|)^{-\mu} \exp(-\beta |m|)
|R_{D}(im)| \frac{|\tau|^{\delta_D}}{1 + |\frac{\tau}{\epsilon^{\Gamma}}|^{2}}
\exp( \nu |\frac{\tau}{\epsilon^{\Gamma}}| )
\end{equation}
for all $\tau \in U_{\mathfrak{u}_{j}} \cup D(0,\rho)$, all $\epsilon \in D(0,\epsilon_{0}) \setminus \{ 0 \}$.
By construction, we observe that the piece of forcing term
$t^{-d_{D}}F^{\mathfrak{u}_{j}}(t,z,\epsilon)$ for the specific value $\theta_{F}=\mathfrak{u}_{j}$ as
described in Subsection 2.2 may be written as a usual Laplace/Fourier inverse transform along the halfline
$L_{\mathfrak{u}_{j}^{\Delta}}$ of $\Upsilon(\tau,m,\epsilon)$ as follows
$$ t^{-d_{D}}F^{\mathfrak{u}_{j}}(t,z,\epsilon) = \frac{1}{(2 \pi)^{1/2}}
\int_{-\infty}^{+\infty} \int_{L_{\mathfrak{u}_{j}^{\Delta}}}
\Upsilon(u,m,\epsilon) \exp( -(\frac{t}{\epsilon^{\gamma}})u ) e^{izm} du dm.
$$
Furthermore, Proposition 7 permits us, for each direction $\mathfrak{u}_{j}$, to build a solution named
$W^{\mathfrak{u}_{j}}(\tau,m,\epsilon)$ of the convolution equation (\ref{main_conv_eq_W}) which is stemming from
the Banach space $E_{(\nu,\beta,\mu,\Gamma,\epsilon)}^{\mathfrak{u}_{j}}$ and is therefore submitted to the next
bounds
\begin{equation}
|W^{\mathfrak{u}_{j}}(\tau,m,\epsilon)| \leq \varpi_{1}(1 + |m|)^{-\mu} e^{-\beta|m|}
\frac{1}{(1 + |\frac{\tau}{\epsilon^{\Gamma}}|^{2})} \exp( \nu |\frac{\tau}{\epsilon^{\Gamma}}| )
\label{bounds_W_uj}
\end{equation}
for all $\tau \in \bar{D}(0,\rho) \cup U_{\mathfrak{u}_{j}}$, $m \in \mathbb{R}$,
$\epsilon \in D(0,\epsilon_{0}) \setminus \{ 0 \}$, for some well chosen $\varpi_{1}>0$. In particular,
these functions $W^{\mathfrak{u}_{j}}(\tau,m,\epsilon)$ are analytic continuations w.r.t $\tau$ of a common
function set as $\tau \mapsto W(\tau,m,\epsilon)$ on $D(0,\rho)$. We define
$v^{\mathfrak{u}_{j}}(t,z,\epsilon)$ as a usual Laplace and Fourier inverse transform
$$ v^{\mathfrak{u}_{j}}(t,z,\epsilon) = \frac{\epsilon^{\gamma_0}}{(2\pi)^{1/2}}
\int_{-\infty}^{+\infty} \int_{L_{\mathfrak{u}_{j}^{\Delta}}}
W^{\mathfrak{u}_{j}}(u,m,\epsilon) \exp( - (\frac{t}{\epsilon^{\gamma}})u )
e^{izm} du dm. $$
By construction, each function $(t,z) \mapsto t^{-d_{D}}F^{\mathfrak{u}_{j}}(t,z,\epsilon)$ and
$(t,z) \mapsto v^{\mathfrak{u}_{j}}(t,z,\epsilon)$ represents a bounded and holomorphic map on the domain
$\mathcal{T}_{\epsilon}^{\infty} \times H_{\beta'}$, for any given $0 < \beta' < \beta$, for any fixed
$\epsilon \in \mathcal{E}_{j}^{\infty}$, according to Definition 7.

Refering to the basic properties of the classical Laplace and Fourier inverse transforms disclosed in Proposition 2 and
Lemma 4, we notice that the function $(t,z) \mapsto v^{\mathfrak{u}_{j}}(t,z,\epsilon)$ actually solves
the equation (\ref{main_PDE_u_divided_tdD}) and hence the equation (\ref{main_PDE_u}) after multiplication by
$t^{d_D}$, where the expression $F^{\theta_{F}}(t,z,\epsilon)$ needs to be specialized for
$\theta_{F} = \mathfrak{u}_{j}$ and subsequently be replaced by the function
$F^{\mathfrak{u}_{j}}(t,z,\epsilon)$, for all $\epsilon \in \mathcal{E}_{j}^{\infty}$ and
$(t,z) \in \mathcal{T}_{\epsilon}^{\infty} \times H_{\beta'}$. Furthermore, by direct inspection, we can check that
for each $t \in \mathcal{T}^{\infty}$, the function
$(\epsilon,z) \mapsto \epsilon^{-\gamma_{0}}v^{\mathfrak{u}_{j}}(t,z,\epsilon)$ is bounded holomorphic on
$\mathcal{E}_{j}^{\infty} \times H_{\beta'}$, for any $0 < \beta' < \beta$ provided that
$|\epsilon| < \sigma_{t}$, for $\sigma_{t}$ defined in (\ref{defin_sigma_t}).

In the remaining part of the proof, we aim attention at the bounds
(\ref{diff_v_uj_exp_small}). The lines of arguments are bordering those given in Theorem 1 in order to
yield the estimates (\ref{difference_u_dp_exp_small_epsilon}). Namely, the first task consists in splitting
the difference $\epsilon^{-\gamma_0}v^{\mathfrak{u}_{j+1}} - \epsilon^{-\gamma_0}v^{\mathfrak{u}_{j+1}}$
into a sum of three integrals that are easier to handle. More precisely, owing to the fact that the function
$u \mapsto W(u,m,\epsilon)\exp( -(\frac{tu}{\epsilon^{\gamma}}) )$ is holomorphic on
$D(0,\rho)$, for all $(m,\epsilon) \in \mathbb{R} \times (D(0,\epsilon_{0}) \setminus \{ 0 \})$, its integral
along a segment connecting 0 and $(\rho/2)e^{i \mathfrak{u}_{j+1}}$, followed by an arc of circle with radius
$\rho/2$ joining $(\rho/2)e^{i \mathfrak{u}_{j+1}}$ and $(\rho/2)e^{i \mathfrak{u}_{j}}$ and ending with
a segment with edges located at $(\rho/2)e^{i \mathfrak{u}_{j}}$ and 0, is vanishing. As a result, we can
expand the next difference
\begin{multline}
\epsilon^{-\gamma_0}v^{\mathfrak{u}_{j+1}}(t,z,\epsilon) -
\epsilon^{-\gamma_0}v^{\mathfrak{u}_{j}}(t,z,\epsilon) =
\frac{1}{(2\pi)^{1/2}}\int_{-\infty}^{+\infty}
\int_{L_{\rho/2,\mathfrak{u}_{j+1}^{\Delta}}}
W^{\mathfrak{u}_{j+1}}(u,m,\epsilon) \\
\times \exp( -(\frac{t}{\epsilon^{\gamma}})u )
e^{izm} du dm\\ -
\frac{1}{(2\pi)^{1/2}}\int_{-\infty}^{+\infty}
\int_{L_{\rho/2,\mathfrak{u}_{j}^{\Delta}}}
W^{\mathfrak{u}_j}(u,m,\epsilon) \exp( -(\frac{t}{\epsilon^{\gamma}})u )
e^{izm} du dm\\
+ \frac{1}{(2\pi)^{1/2}}\int_{-\infty}^{+\infty}
\int_{C_{\rho/2,\mathfrak{u}_{j}^{\Delta},\mathfrak{u}_{j+1}^{\Delta}}}
W(u,m,\epsilon) \exp( -(\frac{t}{\epsilon^{\gamma}})u )
e^{izm} du dm \label{difference_epsilon_gamma0_v_uj_decomposition}
\end{multline}
where $L_{\rho/2,\mathfrak{u}_{j+1}^{\Delta}} = [\rho/2,+\infty)e^{i\mathfrak{u}_{j+1}^{\Delta}}$,
$L_{\rho/2,\mathfrak{u}_{j}^{\Delta}} = [\rho/2,+\infty)e^{i\mathfrak{u}_{j}^{\Delta}}$ and
$C_{\rho/2,\mathfrak{u}_{j}^{\Delta},\mathfrak{u}_{j+1}^{\Delta}}$ stands for an arc of circle with radius
joining $(\rho/2)e^{i\mathfrak{u}_{j}^{\Delta}}$ and
$(\rho/2)e^{i\mathfrak{u}_{j+1}^{\Delta}}$ with an appropriate orientation.\medskip

\noindent We provide upper bounds for the first integral
$$ J_{1} = \left| \frac{1}{(2\pi)^{1/2}} \int_{-\infty}^{+\infty}
\int_{L_{\rho/2,\mathfrak{u}_{j+1}^{\Delta}}} W^{\mathfrak{u}_{j+1}}(u,m,\epsilon)
\exp( -(\frac{t}{\epsilon^{\gamma}})u )
e^{izm} du dm \right| $$
In accordance with the above estimates (\ref{bounds_W_uj}) and with the constraints disclosed in Definition 7,
we check that
\begin{multline}
J_{1} \leq \frac{1}{(2\pi)^{1/2}}
\int_{-\infty}^{+\infty} \int_{\rho/2}^{+\infty} \varpi_{1}
(1 + |m|)^{-\mu} e^{-\beta|m|} \frac{1}{1 + (\frac{r}{|\epsilon|^{\Gamma}})^{2}}
\exp( \nu \frac{r}{|\epsilon|^{\Gamma}} ) \\
\times \exp( -\frac{|t|}{|\epsilon|^{\gamma}}
r \cos( \mathfrak{u}_{j}^{\Delta} + \mathrm{arg}(\frac{t}{\epsilon^{\gamma}}) )
\exp( -m \mathrm{Im}(z) ) dr dm\\
\leq \frac{\varpi_{1}}{(2\pi)^{1/2}} \int_{-\infty}^{+\infty}
e^{-(\beta - \beta')|m|} dm \int_{\rho/2}^{+\infty}
\exp( -r (-\frac{\nu}{|\epsilon|^{\Gamma}} + \frac{|t|}{|\epsilon|^{\gamma}}\delta_{1}^{\infty} ) )
dr \\
= \frac{2 \varpi_{1}}{(2\pi)^{1/2}(\beta - \beta')} \frac{1}{-\frac{\nu}{|\epsilon|^{\Gamma}} +
\frac{|t|}{|\epsilon|^{\gamma}}\delta_{1}^{\infty} }
\exp( - \frac{\rho}{2}( -\frac{\nu}{|\epsilon|^{\Gamma}} + \frac{|t|}{|\epsilon|^{\gamma}}\delta_{1}^{\infty} ) )\\
\leq \frac{2 \varpi_{1}}{(2\pi)^{1/2}(\beta - \beta')}
\frac{|\epsilon|^{\gamma}}{\delta_{2}^{\infty}|t|} \exp( -\frac{\rho}{2}
\delta_{2}^{\infty} \frac{|t|}{|\epsilon|^{\gamma}} ) \label{J1_bounds}
\end{multline}
for all $\epsilon \in \mathcal{E}_{j+1}^{\infty} \cap \mathcal{E}_{j}^{\infty}$, with
$|\epsilon| < \sigma_{t}$. In a similar manner, we can furnish estimates for the second integral
$$ J_{2} = \left| \frac{1}{(2\pi)^{1/2}} \int_{-\infty}^{+\infty}
\int_{L_{\rho/2,\mathfrak{u}_{j}^{\Delta}}} W^{\mathfrak{u}_{j}}(u,m,\epsilon)
\exp( -(\frac{t}{\epsilon^{\gamma}})u )
e^{izm} du dm \right|. $$
Namely, we can show that
\begin{equation}
J_{2} \leq 
\frac{2 \varpi_{1}}{(2\pi)^{1/2}(\beta - \beta')}
\frac{|\epsilon|^{\gamma}}{\delta_{2}^{\infty}|t|} \exp( -\frac{\rho}{2}
\delta_{2}^{\infty} \frac{|t|}{|\epsilon|^{\gamma}} ) \label{J2_bounds}
\end{equation}
for all $\epsilon \in \mathcal{E}_{j+1}^{\infty} \cap \mathcal{E}_{j}^{\infty}$, assuming that
$|\epsilon| < \sigma_{t}$.\medskip

\noindent At last, we target the third integral along an arc of circle
$$ J_{3} = \left| \frac{1}{(2\pi)^{1/2}}\int_{-\infty}^{+\infty}
\int_{C_{\rho/2,\mathfrak{u}_{j}^{\Delta},\mathfrak{u}_{j+1}^{\Delta}}}
W(u,m,\epsilon) \exp( -(\frac{t}{\epsilon^{\gamma}})u )
e^{izm} du dm  \right| $$
Calling again to mind the bounds (\ref{bounds_W_uj}) and the constraints discussed in Definition 7, we observe that
\begin{multline}
J_{3} \leq \frac{1}{(2\pi)^{1/2}}
\int_{-\infty}^{+\infty} \left| \int_{\mathfrak{u}_{j}^{\Delta}}^{\mathfrak{u}_{j+1}^{\Delta}}
\varpi_{1} (1 + |m|)^{-\mu} e^{-\beta|m|} \frac{1}{1 + (\frac{\rho/2}{|\epsilon|^{\Gamma}})^{2}}
\exp( \nu \frac{\rho/2}{|\epsilon|^{\Gamma}} ) \right. \\
\left. \times \exp( -(\frac{|t|}{|\epsilon|^{\gamma}} \frac{\rho}{2})
\cos( \theta + \mathrm{arg}(\frac{t}{\epsilon^{\gamma}}) ) e^{-m \mathrm{Im}(z)} \frac{\rho}{2} d\theta \right| dm\\
\leq \frac{\varpi_{1} \rho/2}{(2\pi)^{1/2}} \int_{-\infty}^{+\infty}
e^{-(\beta - \beta')|m|} dm |\mathfrak{u}_{j+1}^{\Delta} - \mathfrak{u}_{j}^{\Delta}|
\exp(-\frac{\rho}{2}( -\frac{\nu}{|\epsilon|^{\Gamma}} + \frac{|t|}{|\epsilon|^{\gamma}}\delta_{1}^{\infty}) )\\
\leq \frac{\varpi_{1}\rho}{(2\pi)^{1/2}(\beta - \beta')}
|\mathfrak{u}_{j+1}^{\Delta} - \mathfrak{u}_{j}^{\Delta}| \exp( -\frac{\rho}{2}
\frac{\delta_{2}^{\infty}}{|\epsilon|^{\gamma}}|t| ) \label{J3_bounds}
\end{multline}
for all $\epsilon \in \mathcal{E}_{j+1}^{\infty} \cap \mathcal{E}_{j}^{\infty}$, when
$|\epsilon| < \sigma_{t}$.\medskip

\noindent In an ultimate step, we gather the three above inequalities
(\ref{J1_bounds}), (\ref{J2_bounds}), (\ref{J3_bounds}) and conclude from the splitting
(\ref{difference_epsilon_gamma0_v_uj_decomposition}) that
\begin{multline*}
|\epsilon^{-\gamma_0}v^{\mathfrak{u}_{j+1}}(t,z,\epsilon) -
\epsilon^{-\gamma_0}v^{\mathfrak{u}_{j}}(t,z,\epsilon)| \leq
\frac{4 \varpi_{1}}{(2\pi)^{1/2}(\beta - \beta')}
\frac{|\epsilon|^{\gamma}}{\delta_{2}^{\infty}|t|} \exp( -\frac{\rho}{2}
\delta_{2}^{\infty} \frac{|t|}{|\epsilon|^{\gamma}} )\\
+ \frac{\varpi_{1}\rho}{(2\pi)^{1/2}(\beta - \beta')}
|\mathfrak{u}_{j+1}^{\Delta} - \mathfrak{u}_{j}^{\Delta}| \exp( -\frac{\rho}{2}
\frac{\delta_{2}^{\infty}}{|\epsilon|^{\gamma}}|t| )
\end{multline*}
for all $\epsilon \in \mathcal{E}_{j+1}^{\infty} \cap \mathcal{E}_{j}^{\infty}$, granting that
$|\epsilon| < \sigma_{t}$. As a result, the inequality (\ref{diff_v_uj_exp_small}) shows up.
\end{proof}

\section{Gevrey asymptotic expansions of the inner and outer solutions}

\subsection{The Ramis-Sibuya approach for the $k-$summability of formal series}

We first remind the reader the notion of $k-$summability as defined in classical textbooks such as
\cite{ba}, \cite{ba2}.

\begin{defin} Let $k > 1/2$ be a real number. A formal series
$$\hat{a}(\epsilon) = \sum_{j=0}^{\infty}  a_{j} \epsilon^{j} \in \mathbb{F}[[\epsilon]]$$
whose coefficients belong to the Banach space $( \mathbb{F}, ||.||_{\mathbb{F}} )$ is called $k-$summable
with respect to $\epsilon$ in the direction $d \in \mathbb{R}$ if \medskip

{\bf i)} one can choose a radius $\rho \in \mathbb{R}_{+}$ such that the following formal series, called formal
Borel transform of $\hat{a}$ of order $k$ 
$$ \mathcal{B}_{k}(\hat{a})(\tau) = \sum_{j=0}^{\infty} \frac{ a_{j} \tau^{j}  }{ \Gamma(1 + \frac{j}{k}) }
\in \mathbb{F}[[\tau]],$$
is absolutely convergent for $|\tau| < \rho$, \medskip

{\bf ii)} there exists an aperture $\delta > 0$ such that the series $\mathcal{B}_{k}(\hat{a})(\tau)$ can be
analytically continued w.r.t $\tau$ in a sector
$S_{d,\delta} = \{ \tau \in \mathbb{C}^{\ast} : |d - \mathrm{arg}(\tau) | < \delta \} $. Moreover, there
exist two constants $C,K >0$ with
$$ ||\mathcal{B}_{k}(\hat{a})(\tau)||_{\mathbb{F}}
\leq C e^{ K|\tau|^{k}} $$
for all $\tau \in S_{d, \delta}$.
\end{defin}
If these constraints are fulfilled, the vector valued Laplace transform of order $k$ of
$\mathcal{B}_{k}(\hat{a})(\tau)$ in the direction $d$ is introduced as
$$ \mathcal{L}^{d}_{k}(\mathcal{B}_{k}(\hat{a}))(\epsilon) = \epsilon^{-k} \int_{L_{\gamma}}
\mathcal{B}_{k}(\hat{a})(u) e^{ - ( u/\epsilon )^{k} } ku^{k-1}du,$$
along any half-line $L_{\gamma} = \mathbb{R}_{+}e^{i\gamma} \subset S_{d,\delta} \cup \{ 0 \}$, where
$\gamma$ may depend on
$\epsilon$ and is chosen in such a way that
$\cos(k(\gamma - \mathrm{arg}(\epsilon))) \geq \delta_{1} > 0$, for some fixed $\delta_{1}$, for all
$\epsilon$ in a sector
$$ S_{d,\theta,R^{1/k}} = \{ \epsilon \in \mathbb{C}^{\ast} : |\epsilon| < R^{1/k} \ \ , \ \ |d - \mathrm{arg}(\epsilon) |
< \theta/2 \},$$
where $\frac{\pi}{k} < \theta < \frac{\pi}{k} + 2\delta$ and $0 < R < \delta_{1}/K$. The
function $\mathcal{L}^{d}_{k}(\mathcal{B}_{k}(\hat{a}))(\epsilon)$
is then named the $k-$sum of the formal series $\hat{a}(t)$ in the direction $d$. It turns out to be bounded and
holomorphic on the sector
$S_{d,\theta,R^{1/k}}$ and has the formal series $\hat{a}(\epsilon)$ as Gevrey asymptotic
expansion of order $1/k$ with respect to $\epsilon$ on $S_{d,\theta,R^{1/k}}$. In other words,
for all
$\frac{\pi}{k} < \theta_{1} < \theta$, there exist $C,M > 0$ such that
$$ ||\mathcal{L}^{d}_{k}(\mathcal{B}_{k}(\hat{a}))(\epsilon) - \sum_{p=0}^{n-1}
a_{p} \epsilon^{p}||_{\mathbb{F}} \leq CM^{n}\Gamma(1+ \frac{n}{k})|\epsilon|^{n} $$
for all $n \geq 1$, all $\epsilon \in S_{d,\theta_{1},R^{1/k}}$.\medskip

The next cohomological criterion for $k-$summability of formal series with coefficients in Banach spaces (see
\cite{ba2}, p. 121 or \cite{hssi}, Lemma XI-2-6) is accustomed to be called Ramis-Sibuya Theorem in the literature.
This result appears as a fundamental tool in the proof of our third main result (Theorem 3).\medskip

\noindent {\bf Theorem (R.S.)} {\it Let $(\mathbb{F},||.||_{\mathbb{F}})$ be a Banach space over $\mathbb{C}$ and
$\{ \mathcal{E}_{p} \}_{0 \leq i \leq \varsigma-1}$ be a good covering in $\mathbb{C}^{\ast}$ with aperture
$\pi/k<2\pi$ (as displayed in Definition 4). For all
$0 \leq p \leq \varsigma - 1$, let $G_{p}$ be a holomorphic function from $\mathcal{E}_{p}$ into
the Banach space $(\mathbb{F},||.||_{\mathbb{F}})$ and let the cocycle
$\Theta_{p}(\epsilon) = G_{p+1}(\epsilon) - G_{p}(\epsilon)$
be a holomorphic function from the sector $Z_{p} = \mathcal{E}_{p+1} \cap \mathcal{E}_{p}$ into $\mathbb{F}$
(with the convention that $\mathcal{E}_{\varsigma} = \mathcal{E}_{0}$ and $G_{\varsigma} = G_{0}$).
We make the following assumptions.\medskip

\noindent {\bf 1)} The functions $G_{p}(\epsilon)$ are bounded as $\epsilon \in \mathcal{E}_{p}$ tends to the origin
in $\mathbb{C}$, for all $0 \leq p \leq \varsigma - 1$.\medskip

\noindent {\bf 2)} The functions $\Theta_{p}(\epsilon)$ are exponentially flat of order $k$ on $Z_{p}$, for all
$0 \leq p \leq \varsigma-1$. More specifically, there exist constants $C_{p},A_{p}>0$ such that
$$ ||\Theta_{p}(\epsilon)||_{\mathbb{F}} \leq C_{p}e^{-A_{p}/|\epsilon|^{k}} $$
for all $\epsilon \in Z_{p}$, all $0 \leq p \leq \varsigma-1$.\medskip

Then, for all $0 \leq p \leq \nu - 1$, the functions $G_{p}(\epsilon)$ represent the $k-$sums on $\mathcal{E}_{p}$ of a
common $k-$summable formal series $\hat{G}(\epsilon) \in \mathbb{F}[[\epsilon]]$.}

\subsection{Parametric Gevrey asymptotic expansions of the inner and outer solutions of our main problem}

In this last subsection, we show that both inner solutions (constructed in Section 3) and outer solutions (built up
in Section 4) of our main equation (\ref{main_PDE_u}) have asymptotic expansions in the small parameter
$\epsilon$ near the origin. These asymptotic expansions have the special feature to be of some Gevrey types relying
on data involved in the shape of equation (\ref{main_PDE_u}) and it is worthwhile noting that these two types
turn out to be distinct in general. We are now in position to state the third main result of our work.

\begin{theo} We consider the singularly perturbed PDE (\ref{main_PDE_u}) and we pretend that all the foregoing 
constraints, by merging the ones from both Theorem 1 and Theorem 2, listed as follows
(\ref{constraints_degree_coeff_Q_Rl}), (\ref{constraint_m0_ml}), (\ref{defin_b_j}),
(\ref{defin_F_tzepsilon}), (\ref{defin_omega_F}), (\ref{cond_deltaD_alpha}), (\ref{cond_deltal_alpha}),
(\ref{quotient_Q_RD_in_S}), (\ref{cond_gamma_alpha_chi}), (\ref{constraint_kl_ml}),
(\ref{constraint_chi_bj_ni_alpha}), (\ref{constraint_chi_dlkappa_deltal_Deltal_alpha_m0}),
(\ref{constraint_chi_kappa_hl_mul_m0_alpha_deltaD}) and
(\ref{constraints_d_D_di_kj_bk_hl}),
(\ref{constraint_DeltaD_gamma}), (\ref{constraint_Upsilon}), (\ref{first_cond_main_conv_eq_W}),
(\ref{second_cond_main_conv_eq_W}), (\ref{third_cond_main_conv_eq_W}),
(\ref{4_cond_main_conv_eq_W}), (\ref{5_cond_main_conv_eq_W}), (\ref{small_coeff_a_c_B_CF_Rl_RD})
hold conjointly.

\noindent 1) Let us consider a good covering $\{ \mathcal{E}_{j}^{\infty} \}_{0 \leq j \leq \iota-1}$
with aperture $\frac{\pi}{\gamma}$ for which an associated family of sector $\mathcal{T}_{\epsilon}^{\infty}$ and
directions $\{ \mathfrak{u}_{j} \}_{0 \leq j \leq \iota-1}$ can be picked up. For each fixed value of $t$
belonging to the sector $\mathcal{T}^{\infty}$ (see (\ref{defin_mathcalT_infty})) and $0 \leq j \leq \iota-1$,
we view the function $(z,\epsilon) \mapsto \epsilon^{-\gamma_0}v^{\mathfrak{u}_{j}}(t,z,\epsilon)$, built up
in Theorem 2, as a bounded holomorphic function named $O_{t}^{j}(\epsilon)$ from
$\mathcal{E}_{j}^{\infty} \cap D(0,\sigma_{t})$ into $\mathcal{O}_{b}(H_{\beta'})$
(which stands for the Banach space of bounded
holomorphic functions on $H_{\beta'}$ equipped with the sup norm). Then, for each $0 \leq j \leq \iota-1$,
$O_{t}^{j}(\epsilon)$ is the $\gamma-$sum on $\mathcal{E}_{j}^{\infty} \cap D(0,\sigma_{t})$ of a
common formal series
$$ \hat{O}_{t}(\epsilon) = \sum_{k \geq 0} O_{t,k} \epsilon^{k} \in \mathcal{O}_{b}(H_{\beta'})[[\epsilon]]. $$
In other words, for all $0 \leq j \leq \iota-1$, there exist two constants $C_{j},M_{j}>0$ such that
\begin{equation}
\sup_{z \in H_{\beta'}} |\epsilon^{-\gamma_{0}}v^{\mathfrak{u}_{j}}(t,z,\epsilon) -
\sum_{k=0}^{n-1} O_{t,k} \epsilon^{k} | \leq C_{j} M_{j}^{n} \Gamma(1 + \frac{n}{\gamma})|\epsilon|^{n} 
\end{equation}
for all $n \geq 1$, all $\epsilon \in \mathcal{E}_{j}^{\infty} \cap D(0,\sigma_{t})$.

\noindent 2) We select a good covering $\{ \mathcal{E}_{p} \}_{0 \leq p \leq \varsigma-1}$ with aperture
$\frac{\pi}{\chi \kappa}$ for which a family of open sectors
$\{ (S_{\mathfrak{d}_{p},\theta,\rho_{X}|\epsilon|^{\chi}})_{0 \leq p \leq \varsigma-1},
\mathcal{T}_{\epsilon,\chi-\alpha} \}$ associated to it can be singled out. Then, Theorem 1 asserts that for each
direction $\mathfrak{u}_{j}$, $0 \leq j \leq \iota-1$, one can construct a family of holomorphic functions
$\{ u^{\mathfrak{d}_{p},j}(t,z,\epsilon) \}_{0 \leq p \leq \varsigma-1}$ solving the main equation
(\ref{main_PDE_u}) where the piece of forcing term $F^{\theta_{F}}(t,z,\epsilon)$ is asked to be specialized
for $\theta_{F}=\mathfrak{u}_{j}$. For all $0 \leq p \leq \varsigma-1$, we regard the map
$(x,z,\epsilon) \mapsto \epsilon^{m_0}u^{\mathfrak{d}_{p},j}(x\epsilon^{\chi - \alpha},z,\epsilon)$ as a bounded
holomorphic function called $I^{p,j}(\epsilon)$ from $\mathcal{E}_{p}$ into
$\mathcal{O}_{b}( (X \cap D(0,\sigma)) \times H_{\beta'} )$ (which represents the Banach space of bounded
holomorphic functions on $(X \cap D(0,\sigma)) \times H_{\beta'}$ endowed with the sup norm). Then, for all
$0 \leq p \leq \varsigma-1$, each $I^{p,j}(\epsilon)$ is the $\chi \kappa$-sum on $\mathcal{E}_{p}$ of a common
formal series
$$ \hat{I}^{j}(\epsilon) = \sum_{k \geq 0} I_{k}^{j} \epsilon^{k} \in
\mathcal{O}_{b}((X \cap D(0,\sigma)) \times H_{\beta'} )[[\epsilon]]. $$
Equivalently, for each $0 \leq p \leq \varsigma-1$, there exist two constants $C_{p},M_{p}>0$ such that
\begin{equation}
\sup_{x \in X \cap D(0,\sigma),z \in H_{\beta'}}
|\epsilon^{m_0}u^{\mathfrak{d}_{p},j}(x\epsilon^{\chi - \alpha},z,\epsilon) - \sum_{k=0}^{n-1} I_{k}^{j}
\epsilon^{k} | \leq C_{p}M_{p}^{n}\Gamma(1 + \frac{n}{\chi \kappa})|\epsilon|^{n}
\end{equation}
for all $n \geq 1$, all $\epsilon \in \mathcal{E}_{p}$.
\end{theo}
\begin{proof} Let us concentrate on the first item. We consider the family of functions
$\{ v^{\mathfrak{u}_{j}}(t,z,\epsilon) \}_{0 \leq j \leq \iota-1}$ constructed in Theorem 2. For each prescribed
value of $t$ inside the sector $\mathcal{T}^{\infty}$ and all $0 \leq j \leq \iota-1$, we set
$G_{j}(\epsilon) := z \mapsto \epsilon^{-\gamma_{0}}v^{\mathfrak{u}_{j}}(t,z,\epsilon)$ which defines a bounded
holomorphic function from $\mathcal{E}_{j}^{\infty} \cap D(0,\sigma_{t})$ into the Banach space
$\mathbb{F}$ of bounded holomorphic functions on $H_{\beta'}$ outfitted with the sup norm. Bearing in mind the
estimates (\ref{diff_v_uj_exp_small}), we deduce that the cocycle
$\Theta_{j}(\epsilon) = G_{j+1}(\epsilon) - G_{j}(\epsilon)$ is exponentially flat of order $\gamma$ on
$Z_{j} = \mathcal{E}_{j}^{\infty} \cap \mathcal{E}_{j+1}^{\infty} \cap D(0,\sigma_{t})$. According to
Theorem (R.S.) stated in
Section 5.1, there exists a formal power series $\hat{G}(\epsilon) \in \mathbb{F}[[\epsilon]]$ for which
the functions $G_{j}(\epsilon)$ are the $\gamma-$sums on $\mathcal{E}_{j}^{\infty} \cap D(0,\sigma_{t})$, for
all $0 \leq j \leq \iota-1$.

We next focus on the second item. For each fixed direction $\mathfrak{u}_{j}$, $0 \leq j \leq \iota-1$, we
consider the set of functions $\{ u^{\mathfrak{d}_{p},j}(t,z,\epsilon) \}_{0 \leq p \leq \varsigma-1}$
introduced in Theorem 1 for the choice of the forcing term $F^{\mathfrak{u}_{j}}(t,z,\epsilon)$ in the main equation
(\ref{main_PDE_u}). For all $0 \leq p \leq \varsigma-1$, we define this time
$\tilde{G}_{p}(\epsilon) := (x,z) \mapsto \epsilon^{m_0}u^{\mathfrak{d}_{p},j}(x \epsilon^{\chi - \alpha},z,\epsilon)$
that represents a bounded holomorphic function from $\mathcal{E}_{p}$ into $\mathbb{F}$ which stands now for
the Banach space of bounded holomorphic functions on $(X \cap D(0,\sigma)) \times H_{\beta'}$ supplied with
the sup norm. Keeping in view the estimates (\ref{difference_u_dp_exp_small_epsilon}), we find out that the
cocycle $\tilde{\Theta}_{p}(\epsilon) := \tilde{G}_{p+1}(\epsilon) - \tilde{G}_{p}(\epsilon)$ decays exponentially with order
$\chi \kappa$ on the crossing section $Z_{p} = \mathcal{E}_{p+1} \cap \mathcal{E}_{p}$. In agreement with
Theorem (R.S.) outlined above, there exists a formal power series $\hat{\tilde{G}}(\epsilon) \in \mathbb{F}[[\epsilon]]$
admitting the maps $\tilde{G}_{p}(\epsilon)$ as its $\chi \kappa-$sums on $\mathcal{E}_{p}$, for all
$0 \leq p \leq \varsigma-1$. This ends the proof of Theorem 3.
\end{proof}

In order to illustrate the theorem enounced above, we provide two examples of the main equation (\ref{main_PDE_u})
satisfying conjointly the constraints outlined in Theorem 1 and Theorem 2.\medskip

\noindent {\bf Examples.} We take $q=1$,$M=0$,$Q=0$ and $D=2$.\\
1) We first consider a situation for which $\kappa=1$. We select the powers of
$t$ and $\epsilon$ in the coefficients of (\ref{main_PDE_u}) as follows
\begin{multline*}
m_{0}=5,m_{1}=4,k_{1}=2,\mu_{0}=2,h_{0}=2,n_{0}=5,b_{0}=1,\Delta_{2}=3,d_{2}=4,\delta_{2}=2,\\
\Delta_{1}=3,d_{1}=3,\delta_{1}=1. 
\end{multline*}
In this setting, we choose $\kappa=1$,$\chi=6$,$\Gamma=1$,$\gamma_{0}=0$,$\gamma=3/2$,$\alpha=-1$ and $n_{F}=5$. 
Notice that we cannot sort $\chi$ smaller than 6 due to the inequality
(\ref{constraint_chi_kappa_hl_mul_m0_alpha_deltaD}). For these data, one can check that all the constraints asked
in Theorem 3 (combining the ones of Theorem 1 and 2)
on the coefficients of (\ref{main_PDE_u}) w.r.t $t$ and $\epsilon$ are fulfilled. For this special situation,
the main equation (\ref{main_PDE_u}) is displayed as follows
\begin{multline*}
(a_{1} \epsilon^{4}t^{2} + a_{0}\epsilon^{5})Q(\partial_{z})u(t,z,\epsilon) +
c_{0}\epsilon^{2}t^{2}Q_{1}(\partial_{z})u(t,z,\epsilon)Q_{2}(\partial_{z})u(t,z,\epsilon)\\
= b_{0}(z)\epsilon^{5}t + \frac{\epsilon^{5}}{(2\pi)^{1/2}}
\int_{-\infty}^{+\infty} \int_{L_{\theta_{F}}} \omega_{F}(u,m)( \exp( -\frac{t}{\epsilon^{3/2}}u ) - 1)
e^{izm} dudm \\
+ \epsilon^{3}t^{4}\partial_{t}^{2}R_{2}(\partial_{z})u(t,z,\epsilon) +
\epsilon^{3}t^{3}\partial_{t}R_{1}(\partial_{z})u(t,z,\epsilon)
\end{multline*}
This last equation can be divided by $\epsilon^{2}$ but not by any positive power of $t$. The resulting equation
is still singularly perturbed with an irregular singularity at $t=0$ and carries two movable turning points which
coalesce to 0 as $\epsilon$ tends to the origin.\\
2) The second example concerns the case $\kappa=2$. Namely, let us pick out the powers of $t$ and $\epsilon$ in the
following manner,
\begin{multline*}
m_{0}=9,m_{1}=8,k_{1}=4,\mu_{0}=2,h_{0}=4,n_{0}=9,b_{0}=3,\Delta_{2}=5,d_{2}=6,\delta_{2}=2,\\
\Delta_{1}=6,d_{1}=4,\delta_{1}=1.
\end{multline*}
Under this choice, we set $\kappa=2$,$\chi=12$,$\Gamma=5/2$,$\gamma_{0}=1$,$\gamma=3$,$\alpha=-1$ and $n_{F}=11$.
As in the first example, there is some lower bound for $\chi$ that cannot be taken less than 12 according to
(\ref{constraint_chi_kappa_hl_mul_m0_alpha_deltaD}). Under these conditions, all the requirements needed on
the coefficients of (\ref{main_PDE_u}) w.r.t $t$ and $\epsilon$ demanded in Theorem 3 are favorably completed.
In this case, (\ref{main_PDE_u}) is written as follows
\begin{multline*}
(a_{1} \epsilon^{8}t^{4} + a_{0}\epsilon^{9})Q(\partial_{z})u(t,z,\epsilon) +
c_{0}\epsilon^{2}t^{4}Q_{1}(\partial_{z})u(t,z,\epsilon)Q_{2}(\partial_{z})u(t,z,\epsilon)\\
= b_{0}(z)\epsilon^{9}t^{3} + \frac{\epsilon^{11}}{(2\pi)^{1/2}}
\int_{-\infty}^{+\infty} \int_{L_{\theta_{F}}} \omega_{F}(u,m)( \exp( -\frac{t}{\epsilon^{3}}u ) - 1)
e^{izm} dudm \\
+ \epsilon^{5}t^{6}\partial_{t}^{2}R_{2}(\partial_{z})u(t,z,\epsilon) +
\epsilon^{6}t^{4}\partial_{t}R_{1}(\partial_{z})u(t,z,\epsilon)
\end{multline*}
As above, one can factor out the power $\epsilon^{2}$ from the equation but not any power of $t$. The new
equation obtained remains singularly perturbed in $\epsilon$ with an irregular singularity at $t=0$ and still
possess four turning points which merge at 0 as $\epsilon$ reachs the origin.

\medskip

\noindent {\bf Remark 2.} According to the first remark disclosed in Section 2.2 which states that the forcing
term $F^{\mathfrak{u}_{j}}(t,z,\epsilon)$ solves the special ODE (\ref{ODE_F_theta_F}), we observe that
for all $0 \leq j \leq \iota-1$ and $0 \leq p \leq \varsigma-1$ both outer solution
$v^{\mathfrak{u}_{j}}(t,z,\epsilon)$ constructed in Theorem 2 and related inner solution
$u^{\mathfrak{d}_{p},j}(t,z,\epsilon)$ defined in Theorem 1 actually solve the next singularly PDE which
displays a similar shape as the main equation (\ref{main_PDE_u}) but possesses
rational coefficients in time $t$ and parameter $\epsilon$,
\begin{multline*}
F_{2}(-\epsilon^{\gamma} \partial_{t}) \left( (\sum_{l=1}^{q} a_{l} \epsilon^{m_{l}} t^{k_l} +
a_{0}\epsilon^{m_{0}}) Q(\partial_{z}) u(t,z,\epsilon) \right) \\
+ F_{2}(-\epsilon^{\gamma} \partial_{t})  \left( (\sum_{l=0}^{M} c_{l} \epsilon^{\mu_{l}} t^{h_{l}})Q_{1}(\partial_{z})u(t,z,\epsilon)
Q_{2}(\partial_{z})u(t,z,\epsilon) \right)  \\
= \sum_{j=0}^{Q} b_{j}(z) \epsilon^{n_j} F_{2}(-\epsilon^{\gamma} \partial_{t})t^{b_j}
+ \epsilon^{n_F}c_{F}(z) \left( \sum_{k=0}^{\mathrm{deg}(F_1)} F_{1,k}
\frac{k!}{(K_{F} + \frac{t}{\epsilon^{\gamma}})^{k+1}} - F_{2}(0)c_{F_{1},F_{2},\mathfrak{u}_{j}} \right)\\
+ \sum_{l=1}^{D} \epsilon^{\Delta_l} F_{2}(-\epsilon^{\gamma} \partial_{t})\left( t^{d_l} \partial_{t}^{\delta_l}
R_{l}(\partial_{z})u(t,z,\epsilon) \right).
\end{multline*}

\noindent {\bf Remark 3.} Let the constraints of Theorem 1 and Theorem 2 hold mutually. Select some
integers $0 \leq j \leq \iota-1$ and $0 \leq p \leq \varsigma-1$ for which
$\mathcal{E}_{j}^{\infty} \cap \mathcal{E}_{p} \neq \emptyset$. Then, for any fixed
$\epsilon \in \mathcal{E}_{j}^{\infty} \cap \mathcal{E}_{p}$, the outer solution
$v^{\mathfrak{u}_{j}}(t,z,\epsilon)$ is well defined for all $t \in \mathcal{T}_{\epsilon}^{\infty}$
and the inner solution $u^{\mathfrak{d}_{p},j}(t,z,\epsilon)$ for all
$t \in \mathcal{T}_{\epsilon,\chi - \alpha}$, provided that $z \in H_{\beta'}$. But it turns out that
\begin{equation}
\mathcal{T}_{\epsilon}^{\infty} \cap \mathcal{T}_{\epsilon,\chi - \alpha} = \emptyset
\label{domain_time_in_out_void}
\end{equation}
subjected to the fact that $|\epsilon|$ is taken small enough. Namely, let us first notice that
\begin{equation}
\chi - \alpha > \gamma - \Gamma. \label{chi_alpha_gamma_Gamma_ineq}
\end{equation}
According to (\ref{constraint_DeltaD_gamma}), (\ref{cond_deltaD_alpha}) and (\ref{cond_deltal_alpha}), we get that
\begin{equation}
\gamma\delta_{D} - \gamma_{0} = \alpha \delta_{D} \kappa + m_{0}. \label{emptyset_cond_A}
\end{equation}
Bearing in mind (\ref{second_cond_main_conv_eq_W}) and (\ref{cond_deltal_alpha}), we deduce
\begin{equation}
m_{0} + \gamma_{0} \geq (\gamma - \Gamma)\delta_{D}(\kappa+1) + \Gamma\delta_{D}. \label{emptyset_cond_B} 
\end{equation}
As an offshoot of (\ref{emptyset_cond_A}) and (\ref{emptyset_cond_B}), we obtain
$$ \gamma\delta_{D} \geq \alpha \delta_{D}\kappa + (\gamma-\Gamma)\delta_{D}(\kappa+1) + \Gamma \delta_{D}.$$
Since $\delta_{D},\kappa \geq 1$, we can factor out $\delta_{D}$ and $\kappa$ in this last inequality in order
to obtain
$$ 0 \geq \alpha + \gamma - \Gamma.$$
Since $\chi$ is assumed to be a real number larger than $\frac{1}{2\kappa}$, we deduce that
(\ref{chi_alpha_gamma_Gamma_ineq}) must hold.
In particular, we deduce that
$$ \rho_{X}|\epsilon|^{\chi - \alpha} < \frac{\Delta_{\nu,\delta_{1}^{\infty}}}{2}|\epsilon|^{\gamma-\Gamma}
< \Delta_{\nu,\delta_{1}^{\infty}}|\epsilon|^{\gamma-\Gamma} $$
for all $|\epsilon|$ small enough. This implies the empty intersection (\ref{domain_time_in_out_void}).

As a result, we observe some {\it scaling gap} in time $t$ between these two families of solutions. We postpone for
future investigations the study of possible analytic continuation in time $t$ and matching properties between
the inner and outer solutions constructed above.

\end{document}